\documentclass{amsart}
\usepackage{amssymb}
\usepackage{hyperref}
\usepackage{mathrsfs}

\usepackage{tikz}
\usetikzlibrary{cd}

\newcommand{\C}{\mathbb{C}}
\newcommand{\Z}{\mathbb{Z}}
\newcommand{\Q}{\mathbb{Q}}
\newcommand{\F}{\mathbb{F}}
\newcommand{\bk}{\Bbbk}

\newcommand{\id}{\mathrm{id}}

\newcommand{\Fl}{\mathrm{Fl}}
\newcommand{\Gr}{\mathrm{Gr}}

\newcommand{\rS}{\mathrm{S}}

\newcommand{\si}{{\frac{\infty}{2}}}
\newcommand{\IC}{\mathrm{IC}}

\newcommand{\cE}{\mathcal{E}}
\newcommand{\cF}{\mathcal{F}}
\newcommand{\cG}{\mathcal{G}}
\newcommand{\cI}{\mathcal{I}}
\newcommand{\cJ}{\mathcal{J}}

\newcommand{\cO}{\mathcal{O}}
\newcommand{\cP}{\mathcal{P}}

\newcommand{\pH}{{}^{\mathrm{p}} \hspace{-0.5pt} \mathcal{H}}
\newcommand{\Gaits}{\mathrm{Ga}^\si}

\newcommand{\bX}{\mathbf{X}}
\newcommand{\bY}{\mathbf{Y}}
\newcommand{\fR}{\mathfrak{R}}
\newcommand{\fRs}{\mathfrak{R}_{\mathrm{s}}}

\newcommand{\et}{{\acute{\mathrm{e}}\mathrm{t}}}

\newcommand{\bi}{\mathbf{i}}

\DeclareMathOperator{\rank}{rank}
\DeclareMathOperator*{\colim}{\mathrm{colim}}

\newcommand{\simto}{\xrightarrow{\sim}}
\newcommand{\la}{\langle}
\newcommand{\ra}{\rangle}

\newcommand{\siIw}{\mathrm{I}^{\si}}
\newcommand{\siIwu}{\mathrm{I}_{\mathrm{u}}^{\si}}
\newcommand{\Loop}{\mathcal{L}}
\newcommand{\Shv}{\mathsf{Shv}}

\newcommand{\Vect}{\mathsf{Vect}}

\newcommand{\Spec}{\mathrm{Spec}}

\newcommand{\Bun}{\mathrm{Bun}}
\newcommand{\bBun}{\overline{\Bun}}

\newcommand{\Zas}{\mathbf{Z}}

\newcommand{\CZas}{\mathbf{F}}

\newcommand{\Part}{\mathfrak{P}}
\newcommand{\disj}{{\mathrm{disj}}}
\newcommand{\dist}{{\mathrm{dist}}}

\newcommand{\uZasG}{\overline{\mathbf{Z}}}
\newcommand{\uZas}{\underline{\mathbf{Z}}}
\newcommand{\uCZasG}{\overline{\mathbf{F}}}
\newcommand{\uCZas}{\underline{\mathbf{F}}}

\newcommand{\scO}{\mathscr{O}}
\newcommand{\bbP}{\mathbb{P}}
\newcommand{\bbA}{\mathbb{A}}
\newcommand{\ppl}{[}
\newcommand{\ppr}{]}
\newcommand{\Gm}{\mathbb{G}_{\mathrm{m}}}

\mathchardef\mhyphen="2D

\numberwithin{equation}{section}
\newtheorem{thm}{Theorem}[section]
\newtheorem{lem}[thm]{Lemma}
\newtheorem{prop}[thm]{Proposition}
\newtheorem{cor}[thm]{Corollary}

\theoremstyle{definition}

\theoremstyle{remark}
\newtheorem{rmk}[thm]{Remark}

\title{Modular intersection cohomology of Drinfeld's compactifications}

\author{Pramod N. Achar}
\address{Department of Mathematics\\
  Louisiana State University\\
  Baton Rouge, LA 70803\\
  U.S.A.}
\email{pramod@math.lsu.edu}

\author{Gurbir Dhillon}
\address{UCLA Mathematics Department, Los Angeles, CA 90095-1555, USA.}
\email{gsd@math.ucla.edu}

\author{Simon Riche}
\address{Universit\'e Clermont Auvergne, CNRS, LMBP, F-63000 Clermont-Ferrand, France.}
\email{simon.riche@uca.fr}

\thanks{P.A. was supported by NSF Grant No.~DMS-2202012. G.D. was supported by an NSF Postdoctoral Fellowship under grant No.~DMS-2103387.  This project has received
funding from the European Research Council (ERC) under the European Union's Horizon 2020
research and innovation programme (S.R., grant agreement No.~101002592).
}

\begin{document}

\begin{abstract}
We compute the dimension of the cohomology of stalks of intersection cohomology complexes on Zastava schemes and Drinfeld compactifications associated with a connected reductive algebraic group $G$, in case the characteristic of the coefficients field $\bk$ is good for $G$. In particular, we show that these dimensions do not depend on the choice of $\bk$.
\end{abstract}

\maketitle


\section{Introduction}

\subsection{Drinfeld's compactifications}

Let $\F$ be an algebraically closed field, let $G$ be a connected reductive algebraic group over $\F$ with simply-connected derived subgroup, and let $B \subset G$ be a Borel subgroup. Fix also a smooth projective curve $C$ over $\F$, and consider the stack $\Bun_B$ of $B$-torsors on $C$. This stack has a remarkable ``relative compactification'' $\bBun_B$ first considered by Drinfeld, and which has now become a central player in the geometric Langlands program (see e.g.~\cite{bg, raskin1}), as well as in geometric representation theory (see~\cite{abbgm}). 

This stack is singular, and an interesting ``measure'' of its singularities is provided by its local intersection cohomology, i.e.~the cohomology of the stalks (or, equivalently, costalks) of its intersection cohomology complex $\IC_{\bBun_B}$ with coefficients in a given field $\bk$. In the case when $\bk$ has characteristic $0$, these spaces have been described in~\cite{bfgm}, following an earlier study of the case $C=\mathbb{P}^1$ in~\cite{ffkm}; the answer involves the $q$-analogue of Kostant's partition function defined by Lusztig in~\cite{lusztig}. The main goal of the present paper is to show that this description remains valid for any field whose characteristic is good for $G$ (e.g., as soon as this characteristic is not $2$, $3$ or $5$), see Theorem~\ref{thm:const} and Corollary~\ref{cor:formula-P}. In fact we even deduce a version of this statement for intersection cohomology with coefficients in a principal ideal domain in which all bad primes are invertible, see Remark~\ref{rmk:PID}.

The description of the cohomology of (co)stalks of $\IC_{\bBun_B}$ for characteristic-$0$ coefficients is one of the crucial ingredients for several constructions of interest related to Langlands duality (see e.g.~\cite{cr}, which builds on the methods of~\cite{raskin1}), and we hope that its modular analogue established in this paper can be useful in view of obtaining versions of these constructions for positive-characteristic coefficients.

\begin{rmk}
In~\cite{bfgm} the authors also describe the cohomology of (co)stalks of the intersection cohomology complex on two different relative compactifications of the stack of torsors on $C$ for a parabolic subgroup of $G$. We expect that as above this description remains valid for any field of good characteristic, and that this statement can be obtained by a variation of the techniques used in the present paper. This will be the subject of future work.
\end{rmk}

\subsection{Relation with semiinfinite sheaves}

The way we approach this question is inspired by results of Gaitsgory from~\cite{gaitsgory}, but used ``in the reverse direction.'' Namely, consider the affine Grassmannian $\Gr$ of $G$, an ind-projective ind-scheme over $\F$. Let $T \subset B$ be a maximal torus, and denote by $U$ the unipotent radical of $B$. We call ``semiinfinite Iwahori subgroup'' the group ind-scheme
\[
\siIw := \Loop^+ T \ltimes \Loop U,
\]
where $\Loop^+(-)$, resp.~$\Loop(-)$, denotes the arc space, resp.~loop space, associated with a scheme over $\F$. This group acts on $\Gr$ on the left, and following Gaitsgory one can make sense of the $\infty$-category $\Shv(\siIw \backslash \Gr)$ of $\siIw$-equivariant sheaves on $\Gr$. More specifically the construction of this $\infty$-category requires a choice of coefficient field $\bk$;\footnote{In case $\F=\C$, one can work with sheaves for the analytic topology, and $\bk$ can be arbitrary. For a general field $\F$ we work with \'etale sheaves, and have to impose the usual restrictions on $\bk$ in this setting. See~\S\ref{ss:complexes-sheaves} for details.} Gaitsgory has explicitly studied the case when $\bk$ has characteristic $0$, but as we have explained in~\cite{adr} the definition makes sense for any $\bk$, and most of its formal properties are insensitive to this choice. Objects of this $\infty$-category are called semiinfinite sheaves.

The $\infty$-category $\Shv(\siIw \backslash \Gr)$ admits a natural t-structure called the \emph{perverse} t-structure, whose heart $\Shv(\siIw \backslash \Gr)^\heartsuit$ contains a remarkable object which Gaitsgory calls the \emph{semi-infinite intersection cohomology sheaf}, and that we prefer calling the \emph{Gaitsgory sheaf}. This object is denoted $\Gaits$ here. (In particular, the simple objects in the abelian category $\Shv(\siIw \backslash \Gr)^\heartsuit$ are naturally parametrized by the coweight lattice $\bY=X_*(T)$; its seems natural to call these objects intersection cohomology complexes, but $\Gaits$ is \emph{not} one of them, in fact this object has infinite length.)

One of the main results of~\cite{adr} is the computation of the dimension of the cohomology of stalks and costalks of $\Gaits$ in the case when $\bk$ has good characteristic, generalizing the description in the case when $\mathrm{char}(\bk)=0$ given by Gaitsgory in~\cite{gaitsgory}. Here also the answer naturally involves the $q$-analogue of Kostant's partition function. But, while Gaitsgory's proof uses the known description of the stalks of $\IC_{\bBun_B}$ (in the case he considers) from~\cite{bfgm}, ours is intrinsic to the setting of (constructible) sheaves on $\Gr$, and is used in the present paper as the main input to obtain our version of the results of~\cite{bfgm} for general fields.

\begin{rmk}
\begin{enumerate}
\item
The reason for our restriction on $\mathrm{char}(\bk)$ is that our proof of the description of costalks of $\Gaits$ in~\cite{adr} uses the Mirkovi{\'c}--Vilonen conjecture~\cite[Conjecture~13.3]{mv} on stalks of standard spherical perverse sheaves on $\Gr$. This statement is now known to be true in good characteristic (see~\cite{arider,mr}), but false in general in bad characteristic (see~\cite{juteau}). We do not know the answer to the questions considered above about the (co)stalks of $\IC_{\bBun_B}$ and $\Gaits$ in bad characteristic.
\item
In~\cite{gaitsgory-Ran}, Gaitsgory considers variants of the constructions of~\cite{gaitsgory} ``over the Ran space.'' He explains in~\cite[\S 3.9]{gaitsgory-Ran} that his approach provides an alternative way to prove the results of~\cite{bfgm}. In a sense, the strategy we follow in this paper is similar to that of~\cite{gaitsgory-Ran}, but avoiding the technical complications related to the Ran space. It is likely that the methods of~\cite{gaitsgory-Ran} can be adapted to general coefficients, which would provide an alternative approach to our main result; we plan to come back to these questions in future work.
\end{enumerate}
\end{rmk}

\subsection{Drinfeld's compactifications and the semiinfinite flag variety}

The fact that the question of describing (co)stalks of $\IC_{\bBun_B}$ can be related to the study of the $\infty$-category $\Shv(\siIw \backslash \Gr)$ should not come as a surprise, given the ideas that:
\begin{enumerate}
\item
the stack $\bBun_B$ is supposed to provide ``finite-dimensional models'' for the singularities of Schubert varieties in the (highly infinite-dimensional) semiinfinite flag variety $\Fl^\si = \Loop G / \siIw$;
\item
the $\infty$-category $\Shv(\siIw \backslash \Gr)$ should be considered as a substitute for the $\infty$-category of $\Loop^+ G$-equivariant sheaves on $\Fl^\si$, whose definition presents serious technical difficulties.
\end{enumerate}


Technically, the relation between these problems is established as follows.
By definition there exists a canonical morphism $\bBun_B \to \Bun_T$ (where $\Bun_T$ is the stack of $T$-torsors on $C$), so that we can consider the stack
\[
 \bBun_U = \bBun_B \times_{\Bun_T} \{\cF^0\}
\]
where $\cF^0$ is the trivial $T$-torsor. This stack contains the stack $\Bun_U = \Bun_B \times_{\Bun_T} \{\cF^0\}$ of $U$-torsors on $C$ as an open substack, and we consider the intersection cohomology complex $\IC_{\bBun_U}$ associated with the constant local system on $\Bun_U$. Fixing a closed point $c \in C$, there is also a natural ind-algebraic stack $(\bBun_U)_{\infty \cdot c}$ which contains $\bBun_U$ as a closed substack, and a canonical morphism of stacks
\[
 \pi : \Gr \to (\bBun_U)_{\infty \cdot c}.
\]
(Here the definition of $(\bBun_U)_{\infty \cdot c}$ involves $U$-torsors on $C$ which are allowed to have poles at $c$, and the definition of $\pi$ uses the description of $\Gr$ as the moduli of $G$-torsors on $C$ with a trivialization on $C \smallsetminus \{c\}$.)

In Theorem~\ref{thm:Gaits-Bun} we show that we have
\begin{equation}
\label{eqn:intro-isom-IC-Gaits}
 \pi^! \IC_{\bBun_U} [-\dim(U)] \cong \Gaits.
\end{equation}
Using the known dimensions of costalks of $\Gaits$ in case $\bk$ has good characteristic, this allows us to describe the dimensions of stalks and costalks of $\IC_{\bBun_U}$ along certain strata of its natural stratification. Then, using standard ``factorization'' properties of $\bBun_B$ and of the closely related Zastava spaces established in~\cite{fm,bfgm}, one deduces a description of \emph{all} (co)stalks of $\IC_{\bBun_U}$ and $\IC_{\bBun_B}$. In passing, these considerations also provide a description of local intersection cohomology of Zastava spaces, as in~\cite{bfgm}.

The isomorphism~\eqref{eqn:intro-isom-IC-Gaits} is proved by Gaitsgory in~\cite{gaitsgory} in case $\bk$ has characteristic $0$. The structure of our proof is similar to his, but some of the arguments have to be different because he uses some results from~\cite{bfgm}; here we instead take inspiration from some arguments in~\cite{ffkm}.

\subsection{Contents}

For the reader's convenience, and since this subject is rather technical, a large part of the paper is dedicated to a detailed review (with a few complements) of the construction and main properties of the objects involved, following~\cite{fm, bg, bfgm}. Sections~\ref{sec:prelim-Drinfeld}--\ref{sec:zastavas} therefore contain no result that can be considered original.

In Section~\ref{sec:prelim-Drinfeld} we introduce some basic ingredients of the theory, including symmetric powers of curves. 
In Section~\ref{sec:Drinfeld-compact} we recall the definition of Drinfeld's compactifications, and its natural stratification. 
Section~\ref{sec:zastavas} is devoted to Zastava schemes and stacks. We recall the definition and basic properties of these objects and show that, to some extend, they are independent of the curve that is chosen for their construction, see Lemma~\ref{lem:uZas-compare}. 

In Section~\ref{sec:constructibility} we prove our main results, namely that intersection cohomology complexes of Zastava schemes and Drinfeld compactifications are constructible with respect to the natural stratifications, and that the dimension of their (co)stalks can be expressed in terms of the $q$-analogue of Kostant's partition function. (The constructibility statement alluded to here is another ``standard'' property of these constructions that is claimed in various references, but for which it is difficult to find a detailed explanation in the literature. We treat this question by essentially reducing it to the case $C=\mathbb{P}^1$, which is easy using factorization and equivariance considerations.)

Finally, in Appendix~\ref{sec:PID} we prove a general result that allows us to deduce from our results on intersection cohomology of $\bBun_B$ and Zastava spaces with coefficients in a field, a similar statement for coefficients in a principal ideal domain.

\subsection{Acknowledgements}

We thank Alexis Bouthier for very helpful explanations on various topics related to this work.
We also thank Alexander Braverman and Michael Finkelberg for answering several questions related to Zastava schemes, Matthieu Romagny for providing a reference, and Pierre-Emmanuel Chaput for a useful discussion on curves.

\section{Preliminaries}
\label{sec:prelim-Drinfeld}

\subsection{Group-theoretic data}

Let $\F$ be an algebraically closed field, and let $G$ be a connected reductive algebraic group over $\F$ with simply connected derived subgroup.
We choose a Borel subgroup $B \subset G$ and a maximal torus $T \subset B$, and denote by $B^- \subset G$ the Borel subgroup which is opposite to $B$ with respect to $T$. The unipotent radicals of $B$ and $B^-$ will be denoted $U$ and $U^-$ respectively.
%
%

We will denote by $\bX$ and $\bY$ the character lattice and cocharacter lattice of $T$, respectively. Let also $\fR \subset \bX$ and $\fR^\vee \subset \bY$ be the root and coroot systems respectively. Let $\fR_+ \subset \fR$ be the positive system consisting of the $T$-weights in the Lie algebra of $U$, let $\fR^\vee_+ \subset \fR^\vee$ be the corresponding system of positive coroots, and denote by $\fRs$ and $\fRs^\vee$ be associated bases of $\fR$ and $\fR^\vee$ respectively.
We consider the order $\preceq$ on $\bY$ such that $\lambda \preceq \mu$ iff $\mu-\lambda$ is a sum of positive coroots. Let also $\rho \in \frac{1}{2} \bX$ be the halfsum of the positive roots, and recall that its pairing with any element in the coroot lattice is an integer.

Let $\bX_+ \subset \bX$ be the subset of dominant weights. For $\xi \in \bX_+$ we denote by $\mathsf{N}(\xi)=\mathsf{Ind}_B^G(\F_B(w_\circ(\xi)))$ the induced $G$-module of highest weight $\xi$, and by
$\mathsf{M}(\xi)$ the Weyl module of highest weight $\xi$, i.e.~the dual of $\mathsf{N}(-w_\circ \xi)$. 
(Here, $\F_B(w_\circ(\xi))$ denotes the $1$-dimensional $B$-module of weight $w_\circ(\xi)$, and we follow the conventions of~\cite{jantzen} for induction functors.) 
For $\xi,\xi' \in \bX_+$, by Frobenius reciprocity we have a canonical morphism of $G$-modules
\[
 \mathsf{N}(\xi) \otimes \mathsf{N}(\xi') \to \mathsf{N}(\xi+\xi');
\]
taking duals and twisting by $-w_\circ$ we deduce canonical morphisms
\begin{equation}
\label{eqn:morph-Weyl-modules}
 \mathsf{M}(\xi+\xi') \to \mathsf{M}(\xi) \otimes \mathsf{M}(\xi')
\end{equation}
for any $\xi,\xi' \in \bX_+$. By Frobenius reciprocity we also have a canonical morphism of $B$-modules $\mathsf{N}(\xi) \to \F_B(w_\circ(\xi))$ for any $\xi \in \bX_+$, and similarly a canonical morphism of $B$-modules
\begin{equation}
\label{eqn:morph-character-Weyl}
\F_B(\xi) \to \mathsf{M}(\xi)
\end{equation}
for any $\xi \in \bX_+$.  The latter gives us a canonical vector $v_\xi \in \mathsf{M}(\xi)$ of $T$-weight $\xi$, and~\eqref{eqn:morph-Weyl-modules} sends $v_{\xi+\xi'}$ to $v_\xi \otimes v_{\xi'}$. Below we will also consider the morphism of $B^-$-modules
\begin{equation}
\label{eqn:morph-Weyl-character}
 \mathsf{M}(\xi) \to \F_{B^-}(\xi)
\end{equation}
sending $v_\xi$ to $1$.

\subsection{Symmetric powers of curves}
\label{ss:symmetric-powers}

Let $C$ be a smooth (but not necessarily complete) curve\footnote{As usual, by a curve over $\F$ we mean an integral separated scheme of finite type over $\F$ of dimension $1$.} over $\F$.  For an integer $n \ge 0$, let $C^{(n)}$ denote the $n$-th symmetric power of $C$.  This is a scheme whose $\F$-points are unordered $n$-tuples of points of $C$.  It can be identified with the Hilbert scheme of $n$ points on $C$: that is, $C^{(n)}$ represents the functor whose value on an $\F$-scheme $S$ is given by
\[
C^{(n)}(S) = 
\left\{
\begin{array}{c}
\text{closed subschemes $E \subset C \times S$ such that} \\
\text{$E \hookrightarrow C \times S \to S$ is finite and locally free of rank $n$}
\end{array}
\right\}.
\]
(For the definition of locally free morphisms, see~\cite[\href{https://stacks.math.columbia.edu/tag/02KA}{Tag 02KA}]{stacks-project}.)

One can also consider the open subscheme $C^{(n)}_\dist$ of $C^{(n)}$ whose $\F$-points are unordered $n$-tuples of \emph{distinct} points of $C$.  The functorial description is
\[
C^{(n)}_\dist(S) = 
\{ E \in C^{(n)}(S) \mid \text{$E \hookrightarrow C \times S \to S$ is \'etale} \}.
\]
More generally, given integers $n_1, \ldots, n_k \ge 0$, we define an open subscheme
\[
(C^{(n_1)} \times \cdots \times C^{(n_k)})_\dist \subset C^{(n_1)} \times \cdots \times C^{(n_k)}
\]
of ``pairwise distinct points'' by
\begin{multline*}
(C^{(n_1)} \times \cdots \times C^{(n_k)})_\dist(S) = \\
\left\{
(E_i \in C^{(n_i)}(S): 1 \le i \le k) \,\Big|\,
\begin{array}{c}
\text{$\bigsqcup_i E_i \to C \times S$ is a closed immersion,} \\
\text{and $\bigsqcup_i E_i \hookrightarrow C \times S \to S$ is \'etale}
\end{array}
\right\}.
\end{multline*}

Suppose now that $S = \Spec(\bar k)$ is the spectrum of an algebraically closed field $\bar k$ (containing $\F$), and let $E \in C^{(n)}(\bar k)$.  The structure sheaf $\cO_E$ is supported on finitely many closed points of $C_{\bar k}:=C \times_{\Spec(\F)} \Spec(\bar k)$.  For such a closed point $x$, $\Gamma(x, \cO_E)$ is a finite-dimensional $\bar k$-algebra, and
\[
\sum_{\text{$x$ a closed point of $C_{\bar k}$}} \dim_{\bar k} \Gamma(x, \cO_E) = n.
\]
Let $x_1, \ldots, x_k$ be the finitely many points such that $(\cO_E)_{|x} \ne 0$, and let $\mathbf{a}$ be the (unordered) collection of integers given by
\[
\mathbf{a} = [ \dim_{\bar k} \Gamma(x_1, \cO_E), \ldots, \dim_{\bar k} \Gamma(x_k,\cO_E)].
\]
Then $\mathbf{a}$ is a partition of $n$.  We call it the \emph{Jordan type} of $E$. (In terms of the description of $C^{(n)}$ as the quotient $C^n / \mathfrak{S}_n$, the partition $\mathbf{a}$ records the number of occurrences of the distinct points in the given $n$-tuple of points of $C$.)

Given a partition $\mathbf{a} \vdash n$, we define a functor $C^{(n)}_{\mathbf{a}}$ by 
\[
C^{(n)}_{\mathbf{a}}(S) = 
\left\{ E \in C^{(n)}(S)
\,\Big|\,
\begin{array}{c}
\text{for every geometric point $\bar s \to S$,} \\
\text{the point $E_{\bar s} \in C^{(n)}(\bar s)$ has Jordan type $\mathbf{a}$}
\end{array}
\right\}.
\]
This is a locally closed subscheme of $C^{(n)}$, and we have a stratification\footnote{By \emph{stratification} of a stack we will mean a decomposition as a disjoint union of \emph{smooth} locally closed substacks. We will use the term \emph{decomposition} for a decomposition as a disjoint union of \emph{not necessarily smooth} locally closed substacks.}
\begin{equation}
\label{eqn:stratif-symprod}
C^{(n)} = \bigsqcup_{\mathbf{a} \vdash n} C^{(n)}_{\mathbf{a}}.
\end{equation}
Here, smoothness of the strata follows from
the following classical lemma, which is clear from definitions.

\begin{lem}
\label{lem:curve-sympart}
Let $\mathbf{a}$ be a partition of $n$, and write it as
\[
\mathbf{a} = [\underbrace{a_1, \ldots, a_1}_{\text{$b_1$ copies}}, \underbrace{a_2, \ldots, a_2}_{\text{$b_2$ copies}}, \ldots,
\underbrace{a_k, \ldots, a_k}_{\text{$b_k$ copies}}],
\]
with $a_1, \ldots, a_k$ all distinct.  Then there is an isomorphism of schemes
\[
C^{(n)}_{\mathbf{a}} \cong (C^{(b_1)} \times \cdots \times C^{(b_k)})_\dist.
\]
\end{lem}

Here are two extreme cases of this lemma: we have
\[
C^{(n)}_{[1,1,\ldots,1]} \cong C^{(n)}_\dist
\qquad\text{and}\qquad C^{(n)}_{[n]} \cong C^{(1)} \cong C.
\]

\subsection{\'Etale maps of curves}
\label{ss:etale-maps}

Let $C$ and $D$ be two smooth curves over $\F$, and let $\varphi: C \to D$ be an \'etale map.  For any integer $n \ge 0$, $\varphi$ induces a map of symmetric powers $C^{(n)} \to D^{(n)}$, which will again be denoted $\varphi$.
Define an open subscheme $C^{(n)}_\et \subset C^{(n)}$ by
\[
C^{(n)}_\et(S) = 
\left\{ E \in C^{(n)}(S) \,\Big|\,
\begin{array}{c}
\text{the composition $E \hookrightarrow C \times S \to D \times S$} \\
\text{is a closed immersion}
\end{array}
\right\}.
\]
Of course, $C^{(n)}_\et$ depends on $\varphi$, although this is not shown in the notation.

\begin{lem}\label{lem:sym-etale}
Let $\varphi: C \to D$ be an \'etale map of curves.  Then the induced map $\varphi|_{C^{(n)}_\et}: C^{(n)}_\et \to D^{(n)}$ is \'etale.
\end{lem}

\begin{proof}
At the level of $S$-points, the map $\varphi: C^{(n)}_\et \to D^{(n)}$ sends $E \in C^{(n)}_\et(S)$ to the closed subscheme of $D \times S$ that is the image of $E \hookrightarrow C \times S \to D \times S$. By construction this morphism is locally of finite presentation, hence to prove that it is \'etale it suffices to prove that it is formally \'etale (see~\cite[\href{https://stacks.math.columbia.edu/tag/02HM}{Tag 02HM}]{stacks-project}).

Let $S \hookrightarrow S'$ be a first-order thickening of affine schemes, and suppose we have a commutative diagram given by the solid arrows in the following diagram:
\begin{equation}
\label{eqn:sym-etale1}
\begin{tikzcd}
S \ar[r,"x"] \ar[d, hook] & C^{(n)}_\et \ar[d, "\varphi"] \\
S' \ar[r,"y"] \ar[ur, dashed] & D^{(n)}.
\end{tikzcd}
\end{equation}
In this setting
we must show that there is a unique way to fill in the dotted arrow.  The map $x$ determines a closed subscheme $E \subset C \times S$ such that $E \hookrightarrow C \times S \to D \times S$ is a closed immersion.  Let $E' \subset D \times S$ be its image.  Similarly, $y$ determines a closed subscheme $F \subset D \times S'$.  The commutativity of the diagram means that
\[
E' = F \cap (D \times S).
\]
Note that $E' \hookrightarrow F$ is a first-order thickening of schemes.

Regard $E$ as a closed subscheme of $C \times S'$, and consider the diagram
\begin{equation}\label{eqn:sym-etale2}
\begin{tikzcd}
E \ar[r] \ar[d, hook] & C \times S' \ar[d, "\varphi \times \id_{S'}"] \\
F \ar[r] \ar[ur, dashed] & D \times S'
\end{tikzcd}
\end{equation}
Here, the solid arrows commute, and there is a one-to-one correspondence between dotted arrows completing~\eqref{eqn:sym-etale1} and dotted arrows completing~\eqref{eqn:sym-etale2}.

As explained above the left-hand vertical map in~\eqref{eqn:sym-etale2} is a first-order thickening of schemes, and the right-hand vertical map is \'etale (and hence formally \'etale). There is thus a unique way to complete~\eqref{eqn:sym-etale2}, and hence also a unique way to complete~\eqref{eqn:sym-etale1},
which finishes the proof.
\end{proof}

We record the following immediate properties for later reference.

\begin{lem}
\label{lem:sym-properties}
Let $\varphi: C \to D$ be an \'etale map of curves.
\begin{enumerate}
 \item 
 We have $C^{(n)}_{[n]} \subset C^{(n)}_\et$.
 \item
If $\varphi: C \to D$ is an open immersion, then $C^{(n)}_\et = C^{(n)}$.
\end{enumerate}
\end{lem}

Since $C$ is a curve, any element $E \in C^{(n)}(S)$ is a relative effective Cartier divisor on $C \times S/S$, see~\cite[\href{https://stacks.math.columbia.edu/tag/0B9D}{Tag 0B9D}]{stacks-project}. Its sheaf of ideals is the invertible sheaf
\[
\scO_{C \times S}(-E) \subset \scO_{C \times S},
\]
which satisfies
\[
\scO_{C \times S}(-E)_{|(C \times S) \smallsetminus E} = \scO_{(C \times S) \smallsetminus E}.
\]

\subsection{\texorpdfstring{$\bY_{\succeq 0}$}{Ysucceq0}-graded symmetric powers of curves}
\label{ss:graded-sym-powers}

We now recall a generalization of the constructions of~\S\S\ref{ss:symmetric-powers}--\ref{ss:etale-maps} that adds labels belonging to the positive coroot monoid $\bY_{\succeq 0} = \{\lambda \in \bY \mid \lambda \succeq 0\}$. If $\mu \in \bY_{\succeq 0}$, we will denote by $\Part(\mu)$ the set of partitions of $\mu$, i.e.~unordered collections $\ppl \lambda_j : j \in J \ppr$
of elements of $\bY_{\succeq 0} \smallsetminus \{0\}$ such that $\sum_j \lambda_j = \mu$. (We allow here the empty collection, which provides the unique partition of $0$.)

From now on we fix $\mu \in \bY_{\succeq 0}$, and write $\mu = \sum_{\alpha \in \fRs^\vee} n_\alpha \cdot \alpha$. The associated partially symmetrized power of $C$ is
\[
C^\mu := \prod_{\alpha \in \fRs^\vee} C^{(n_\alpha)}.
\]
Then $C^\mu$ is a smooth scheme of dimension $\langle \mu,\rho \rangle=\sum_\alpha n_\alpha$.
An $S$-point of $C^\mu$ is a collection $E=(E_\alpha: \alpha \in \fRs^\vee)$ such that, for each $\alpha \in \fRs$, $E_\alpha \in C^{(n_\alpha)}(S)$. For brevity, in this setting we will write
\[
(C \times S) \smallsetminus E = (C \times S) \smallsetminus \bigcup_{\alpha \in \fRs^\vee} E_\alpha.
\]
We have a natural morphism
\begin{equation}
\label{eqn:symmetrization-morph}
C^\mu \to C^{(\langle \mu, \rho \rangle)}
\end{equation}
which sends a collection $(E_\alpha: \alpha \in \fRs^\vee)$ to the effective Cartier divisor $\sum_\alpha E_\alpha$.

Let $\bar k$ be an algebraically closed field containing our base field $\F$, and let $E = (E_\alpha: \alpha \in \fR_s^\vee)$ be a point in $C^\mu(\bar k)$.  Given a closed point $x$ in $C_{\bar k}$, we set 
\[
\lambda^E_x = \sum_{\alpha \in \fRs^\vee} \dim_{\bar k} \Gamma(x, \cO_{E_\alpha})\cdot \alpha \in \bY_{\succeq 0}.
\]
Let $x_1, \ldots, x_k$ be the finitely many points such that $\lambda^E_x \ne 0$, and let $\Gamma_E$ be the (unordered) collection of elements of $\bY_{\succeq 0}$ given by
\[
\Gamma_E = \ppl \lambda^E_{x_1}, \ldots, \lambda^E_{x_k} \ppr.
\]
Then $\Gamma_E$ is a partition of $\mu$;
we call it the \emph{Jordan type} of $E$.

Given a partition $\Gamma \in \Part(\mu)$, we define a functor $C^\mu_\Gamma$ by 
\[
C^\mu_\Gamma(S) = 
\left\{ E \in C^\mu(S)
\,\Big|\,
\begin{array}{c}
\text{for every geometric point $\bar s \to S$,} \\
\text{the point $E_{\bar s} \in C^\mu(\bar s)$ has Jordan type $\Gamma$}
\end{array}
\right\}.
\]
This is a locally closed subscheme of $C^\mu$, and as in~\eqref{eqn:stratif-symprod} we have a stratification
\begin{equation}
\label{eqn:stratif-symprod-Y}
C^{\mu} = \bigsqcup_{\Gamma \in \Part(\mu)} C^\mu_\Gamma.
\end{equation}
Here also, smoothness of strata follows from
the following statement, which is a generalization of Lemma~\ref{lem:curve-sympart} taken from~\cite[\S 6.1]{bg}.

\begin{lem}
\label{lem:curve-sympart-Y}
Let $\Gamma \in \Part(\mu)$, and write it as
\[
\Gamma = [\underbrace{\lambda_1, \ldots, \lambda_1}_{\text{$b_1$ copies}}, \underbrace{\lambda_2, \ldots, \lambda_2}_{\text{$b_2$ copies}}, \ldots,
\underbrace{\lambda_k, \ldots, \lambda_k}_{\text{$b_k$ copies}}],
\]
with $\lambda_1, \ldots, \lambda_k$ all distinct.  Then there is an isomorphism of schemes
\[
C^\mu_\Gamma \cong (C^{(b_1)} \times \cdots \times C^{(b_k)})_\dist.
\]
\end{lem}

In particular, the dimension of $C^\mu_\Gamma$ is the number $|\Gamma|$ of parts in $\Gamma$.  As a special case, for $\Gamma = \ppl \mu \ppr$, we have
\begin{equation}
\label{eqn:mu-closed-stratum}
C^\mu_{\ppl\mu\ppr} \cong C.
\end{equation}
The isomorphism $C \simto C^\mu_{\ppl\mu\ppr}$ will be denoted $E \mapsto \mu \cdot E$.

Next, suppose we have an \'etale map of curves $\varphi: C \to D$. Then we have an induced map $C^\mu \to D^\mu$, which we still denote by $\varphi$. Define an open subscheme $C^\mu_\et \subset C^\mu$ by
\[
C^{\mu}_\et(S) = 
\left\{ (E_\alpha)_{\alpha \in \fRs^\vee} \in C^{\mu}(S) \,\Big|\,
\begin{array}{c}
\text{the composition $\sum_\alpha E_\alpha \hookrightarrow C \times S \to D \times S$} \\
\text{is a closed immersion}
\end{array}
\right\}.
\]
(In other words, $C^{\mu}_\et$ is the inverse image of $C^{(\langle \mu, \rho \rangle)}_\et$ under the morphism~\eqref{eqn:symmetrization-morph}.)
The following statement is an immediate consequence of Lemmas~\ref{lem:sym-etale} and~\ref{lem:sym-properties} (noting that $C^\mu_\et \subset \prod_{\alpha \in \fRs^\vee} C^{(n_\alpha)}_\et$).

\begin{lem}
\label{lem:sym-mu}
Let $\varphi: C \to D$ be an \'etale map of curves, and let $\mu \in \bY_{\succeq 0}$.
\begin{enumerate}
\item 
\label{it:sm-etale}
The induced map $\varphi|_{C^\mu_\et}: C^\mu_\et \to D^\mu$ is \'etale.
\item 
\label{it:sm-stratum}
The stratum $C^\mu_{\ppl \mu\ppr}$ is contained in $C^\mu_\et$.
\item 
\label{it:sm-open}
If $\varphi$ is an open immersion, then $C^\mu_\et = C^\mu$.
\end{enumerate}
\end{lem}

Let $E  = (E_\alpha: \alpha \in \fRs^\vee)$ be an element of $C^\mu(S)$. For $\xi \in \bX_+$, we introduce the notation
\[
\la \xi,  E\ra = \sum_{\alpha \in \fRs^\vee} \la \xi, \alpha\ra E_\alpha.
\]
This is a relative effective Cartier divisor on $C \times S / S$, and we have the associated invertible subsheaf
\begin{equation}
\label{eqn:divisor-sheaf}
\scO_{C \times S}(- \la \xi, E\ra) \subset \scO_{C \times S}.
\end{equation}
More generally, for any $\xi \in \bX$ we can consider the (not necessarily effective) Cartier divisor $\la \xi,  E\ra = \sum_{\alpha \in \fRs^\vee} \la \xi, \alpha\ra E_\alpha$, and the associated invertible $\scO_{C \times S}$-module $\scO_{C \times S}(\la \xi, E\ra)$.  This sheaf comes with a canonical isomorphism
\begin{equation}\label{eqn:divisor-sheaf-triv}
\scO_{C \times S}(\la \xi, E\ra)_{|(C \times S) \smallsetminus E} = \scO_{(C \times S) \smallsetminus E}.
\end{equation}
(In case $\xi \in \bX_+$, this isomorphism is simply the restriction of~\eqref{eqn:divisor-sheaf}.)

\subsection{Torsors}
\label{ss:torsors}

Consider a smooth affine group scheme $H$ of finite type over $\F$.
Given an $H$-torsor\footnote{Here, the word ``torsor'' will always mean \emph{right} torsor. For a quick review of some useful basic facts in the theory of torsors over group schemes, see~\cite[\S 2.1.2]{central}.} $\cE$ on a scheme $S$ and a finite-dimensional $H$-module $V$, we will denote by
\begin{equation}
\label{eqn:assoc-bdle-defn}
\cE_V = \cE \times^H V
\end{equation}
the vector bundle on $S$ induced by $\cE$ and $V$. This construction is monoidal in the sense that given $\cE$ as above and finite-dimensional $H$-modules $V,V'$ we have a canonical isomorphism
\begin{equation}
\label{eqn:torsor-tensor-prod}
 \cE_{V \otimes V'} \cong \cE_V \otimes_{\mathscr{O}_S} \cE_{V'}.
\end{equation}
Via this construction, the datum of an $H$-torsor on $S$ is equivalent to the datum of a symmetric monoidal functor from the category of finite-dimensional algebraic representations of $H$ to the category of vector bundles on $S$. 

In particular, in the setting of~\S\ref{ss:graded-sym-powers}, given a $T$-torsor $\cF$ on $C \times S$, an integer $n \in \Z$, and a point $E \in C^\mu(S)$, there exists a unique $T$-torsor $\cF(nE)$ such that for any $\xi \in \bX$ we have
\[
 (\cF(nE))_{\F_T(\xi)} = \cF_{\F_T(\xi)} \otimes_{\scO_{C \times S}} \scO_{C \times S}( n \cdot \langle \xi, E \rangle),
\]
where $\F_T(\xi)$ is the $1$-dimensional $T$-module associated with $\xi$.
(We will mainly consider this construction in case $n=\pm 1$.)

For any scheme $S$ we will denote by $\cF^0$, resp.~$\cE^0$, the trivial $T$-torsor, resp.~$G$-torsor, on $S$. (When we want to emphasize the base scheme we will write $\cF^0_S$ or $\cE^0_S$, but usually it will be clear from context.) For any finite-dimensional $T$-module, resp.~$G$-module, $V$ we have a canonical identification
\[
 (\cF^0)_V = V \otimes \scO_S, \quad \text{resp.} \quad (\cE^0)_V = V \otimes \scO_S.
\]

\subsection{Stacks of torsors}
\label{ss:stacks-torsors}

In this subsection we assume that our curve $C$ is projective, and denote its genus by $g$. Given a smooth affine group scheme $H$ of finite type over $\F$, we will denote by $\Bun_H$ the stack of $H$-torsors on $C$. This is a smooth algebraic stack which is locally of finite type, see e.g.~\cite[Proposition~1]{heinloth}. In case $H$ is reductive, all of its connected components have dimension $(g-1)\dim(H)$, see e.g.~\cite[Proposition~3.6.8]{sorger}.
Given a morphism $K \to H$ between smooth affine group schemes of finite type over $\F$, we have an induced morphism $\Bun_K \to \Bun_H$ sending a $K$-torsor $\cE$ to the induced $H$-torsor $\cE \times^K H$.


We will consider this stack in case $H$ is one of $T$, $B$, $B^-$, $U$, $U^-$ or $G$.
It is a standard fact that the connected components of the stack $\Bun_T$ are in  canonical bijection with $\bY$; we will denote by $\Bun_T^\lambda$ the component corresponding to $\lambda$. (Our normalization is that $\Bun_T^\lambda$ parametrizes torsors of degree $-\lambda$ in the terminology of~\cite[\S 3.1]{bfgm}.) With this notation, for $\mu \in \bY_{\succeq 0}$ and $\lambda \in \bY$ the assignment $(E,\cF) \mapsto \cF(-E)$ from~\S\ref{ss:torsors} defines a morphism of stacks
\begin{equation}
\label{eqn:twist-T-bundles}
 C^\mu \times \Bun_T^\lambda \to \Bun_T^{\lambda+\mu}.
\end{equation}

It is also a standard fact that the morphisms
\begin{equation}
\label{eqn:morph-Bun-B-T}
\Bun_B \to \Bun_T \quad \text{and} \quad \Bun_{B^-} \to \Bun_T
\end{equation}
induced by the morphisms $B \to B/U \cong T$ and $B^- \to B^-/U^- \cong T$ are smooth,
and that they induce bijections between the sets of connected components of these stacks. We will denote by
\[
\Bun_B^\lambda \subset \Bun_B, \quad \Bun_{B^-}^\lambda \subset \Bun_{B^-}
\]
the preimages of $\Bun_T^\lambda$, so that
\[
\Bun_B = \bigsqcup_{\lambda \in \bY} \Bun_B^\lambda, \quad \Bun_{B^-} = \bigsqcup_{\lambda \in \bY} \Bun_{B^-}^\lambda
\]
are the decompositions into connected components.
Note also that, since $T$ is abelian, $\Bun_T$ is naturally a group stack.

As a special case of the discussion above, for any $\lambda \in \bY$ we have
\[
\dim(\Bun_T^\lambda)=(g-1) \dim(T).
\]
As for $B$ and $B^-$, we have\footnote{Let us recall the idea of proof of the formula for $B$, for the reader's convenience. Since $\Bun_B$ is a smooth stack, its dimension is the Euler characteristic of the tangent complex at any point. The tangent complex at $\cE$ is $R\Gamma(C,\cE_{\mathfrak{b}})[1]$, where $\mathfrak{b}$ is the Lie algebra of $B$, and $\cE_{\mathfrak{b}}$ is as in~\eqref{eqn:assoc-bdle-defn}. Now if $\cE$ belongs to $\Bun_B^\lambda$, $\cE_{\mathfrak{b}}$ has a finite filtration with associated graded $\scO_C^{\oplus \dim(T)} \oplus \bigoplus_\alpha \scO_C(-\langle \lambda, \alpha \rangle)$, where $\alpha$ runs over the positive roots. By the Riemann--Roch formula, the Euler characteristic of $R\Gamma(C,\scO_C(-\langle \lambda, \alpha \rangle))[1]$ is $\langle \lambda, \alpha \rangle+g-1$, and that of $R\Gamma(C,\scO_C)[1]$ is $g-1$. Summing all these contributions we get the announced formula.}
\begin{align*}
\dim(\Bun_B^\lambda) &= (g-1)\dim(B) + \langle \lambda, 2\rho \rangle, \\
\dim(\Bun_{B^-}^\lambda) &= (g-1)\dim(B) - \langle \lambda, 2\rho \rangle.
\end{align*}

Following~\cite[\S 3.6]{bfgm},
we denote by $\Bun_T^{\mathrm{r}}$ the open substack of $\Bun_T$ consisting of $T$-torsors $\cF$ such that
\[
\mathsf{H}^1(C, \cF_{\F_T(\alpha)})=0
\]
for any $\alpha \in \fR_+$. We also denote by $\Bun_{B^-}^{\mathrm{r}} \subset \Bun_{B^-}$ the preimage of $\Bun_T^{\mathrm{r}}$ under the morphism $\Bun_{B^-} \to \Bun_T$ 
considered in~\eqref{eqn:morph-Bun-B-T}.

The following statement is~\cite[Lemma~3.7]{bfgm}.

\begin{lem}
\label{lem:smoothness-B-}
The restriction of the canonical morphism $\Bun_{B^-} \to \Bun_G$ to $\Bun_{B^-}^{\mathrm{r}}$ is smooth.
\end{lem}


\begin{rmk}
\phantomsection
\label{rmk:Bunr}
\begin{enumerate}
\item
If $\mathscr{L}$ is a line bundle on $C$ such that $\deg(\mathscr{L}) > 2g-2$ then we have $\mathsf{H}^1(C,\mathscr{L})=0$. Hence if $\cF \in \Bun_T$ has degree $\lambda$, then $\cF$ belongs to $\Bun_T^{\mathrm{r}}$ as soon as $\langle \lambda, \alpha \rangle > 2g-2$ for any $\alpha \in \fR_+$. In other words, $\Bun_T^{\mathrm{r}}$ contains $\Bun_T^\lambda$ as soon as $\langle \lambda,\alpha \rangle < 2-2g$ for any $\alpha \in \fR_+$.
\item
Similarly, if we denote by $\Bun_T^{\mathrm{r},-}$ the open substack of $\Bun_T$ consisting of $T$-torsors $\cF$ such that
\[
\mathsf{H}^1(C, \cF_{\F_T(\alpha)})=0
\]
for any \emph{negative} root $\alpha$, and by $\Bun_{B}^{\mathrm{r},-} \subset \Bun_{B}$ the preimage of $\Bun_T^{\mathrm{r},-}$ under the natural morphism $\Bun_{B} \to \Bun_T$, then the restriction of the canonical morphism $\Bun_{B} \to \Bun_G$ to $\Bun_{B}^{\mathrm{r},-}$ is smooth. In particular, $\Bun_T^{\mathrm{r},-}$ contains $\Bun_T^\lambda$ as soon as $\langle \lambda,\alpha \rangle > 2g-2$ for any $\alpha \in \fR_+$.
\end{enumerate}
\end{rmk}

\subsection{Affine Grassmannian and semiinfinite orbits}
\label{ss:Gr}

In the end of Section~\ref{sec:zastavas} and in Section~\ref{sec:constructibility}, we will consider the affine Grassmannian of $G$; here we introduce some basic notation for later use.

For any affine scheme $X$ over $\F$ we will denote by
\[
\Loop X, \quad \text{resp.} \quad \Loop^+ X,
\]
the ind-affine ind-scheme,
resp.~affine scheme, representing the functor sending an $\F$-algebra $R$ to $X(R( \hspace{-1pt} (z) \hspace{-1pt} ))$, resp.~$X(R[ \hspace{-1pt} [z] \hspace{-1pt} ])$, where $z$ is an indeterminate. 
In case $X$ is a group scheme, $\Loop X$, resp.~$\Loop^+ X$, admits a natural structure of group ind-scheme, resp.~group scheme.
In particular we will consider the groups $\Loop G$, $\Loop^+ G$, $\Loop T$, $\Loop U$, $\Loop U^-$. Each $\lambda \in \bY$ defines an $\F$-point $z^\lambda \in (\Loop T)(\F)$; we will denote similarly its image in $(\Loop G)(\F)$.

The affine Grassmannian of $G$ is the quotient
\[
\Gr = \Loop G/\Loop^+G.
\]
Here one can take the quotient either for the \'etale topology or for the fppf topology, giving rise to the same object, and we obtain a separated ind-scheme of ind-finite type over $\F$, endowed with a natural action of $\Loop G$. Each $\lambda \in \bY$ defines an $\F$-point $z^\lambda \in (\Loop T)(\F)$; we will denote similarly its image in $(\Loop G)(\F)$, and by $[z^\lambda]$ the associated $\F$-point of $\Gr$.

An important role will be played below by the ``semiinfinite orbits'', i.e.~the orbits of $\Loop U$ and $\Loop U^-$ on $\Gr$. We will use standard notation for these orbits:
for $\lambda \in \bY$ we will denote by
\[
\rS_\mu \subset \Gr, \quad \text{resp.} \quad
\rS^-_\mu \subset \Gr,
\]
the $\Loop U$-orbit, resp.~$\Loop U^-$-orbit, of $[z^\mu]$, and by
\[
\bi_\mu : \rS_\mu \to \Gr, \quad \text{resp.} \quad
\bi^-_\mu : \rS^-_\mu \to \Gr,
\]
the embedding.

\section{Drinfeld's compactification}
\label{sec:Drinfeld-compact}

In this section we fix a smooth \emph{projective} curve $C$ over $\F$.

\subsection{Definition}
\label{ss:def-bBunB}

Following Drinfeld, we consider the functor $\bBun_B$ from the category of $\F$-schemes to the category of groupoids sending a scheme $S$ to the groupoid of triples $(\cE,\cF,\kappa)$ where:
\begin{enumerate}
 \item $\cE$ is a $G$-torsor on $C \times S$;
 \item $\cF$ is a $T$-torsor on $C \times S$;
 \item
 $\kappa = (\kappa_\xi : \xi \in \bX_+)$ is a system of morphisms of coherent sheaves
 \[
  \kappa_\xi : \cF_{\F_T(\xi)} \to \cE_{\mathsf{M}(\xi)}
 \]
 such that for any geometric point $\bar s$ of $S$, the pullback of each $\kappa_\xi$ to $C \times \bar s$ is injective, and
that moreover satisfies the Pl\"ucker relations. Explicitly, this means that $\kappa_0$ is the identity map, and for $\xi,\xi' \in \bX_+$ the diagram
\[
\begin{tikzcd}[column sep=large]
\cF_{\F_T(\xi+\xi')} \ar[r, "\kappa_{\xi+\xi'}"] \ar[d, "\eqref{eqn:torsor-tensor-prod}"', "\wr"] & \cE_{\mathsf{M}(\xi+\xi')} \ar[d] \\
\cF_{\F_T(\xi)} \otimes_{\mathscr{O}_{C \times S}} \cF_{\bk_T(\xi')} \ar[r, "\kappa_\xi \otimes \kappa_{\xi'}"]& \cE_{\mathsf{M}(\xi)} \otimes_{\mathscr{O}_{C \times S}} \cE_{\mathsf{M}(\xi')} \overset{\eqref{eqn:torsor-tensor-prod}}{\cong} \cE_{\mathsf{M}(\xi) \otimes \mathsf{M}(\xi')}
\end{tikzcd}
\]
commutes, where the right vertical arrow is the map induced by~\eqref{eqn:morph-Weyl-modules}.
\end{enumerate}

By~\cite[Proposition~1.2.2 and Lemma~5.3.2]{bg}, the functor $\bBun_B$ is an algebraic stack locally of finite type. There exists a canonical morphism of stacks
\begin{equation}
\label{eqn:immersion-BunB-bBunB}
\jmath : \Bun_B \to \bBun_B,
\end{equation}
which sends a $B$-torsor $\cG$ to the triple consisting of the induced $G$-torsor, the induced $T$-torsor (see~\eqref{eqn:morph-Bun-B-T}),
and the morphisms $\kappa_\xi$ induced by~\eqref{eqn:morph-character-Weyl}.
As explained in~\cite[\S 1.2]{bg} this morphism is an open immersion, and its image consists of collections such that each $\kappa_\xi$ is a morphism of vector bundles, i.e.~its cokernel is a vector bundle on $C \times S$. Concretely, this means in particular that the datum of a $B$-torsor on $C$ is equivalent to the datum of a $G$-torsor, a $T$-torsor, and a collection $\kappa$ as above, where we require each $\kappa_\xi$ to be a morphism \emph{of vector bundles}.

We also have an obvious morphism of stacks
\begin{equation}
\label{eqn:morph-bBunB-BunT}
 \bBun_B \to \Bun_T.
\end{equation}
It is well known that this morphism
induces a bijection between the sets of connected components of $\bBun_B$ and $\Bun_T$. For $\lambda \in \bY$, we will denote by $\bBun_B^\lambda \subset \bBun_B$ the component corresponding to $\Bun_T^\lambda \subset \Bun_T$, i.e.~the preimage of $\Bun_T^\lambda$. Then we have
\[
 \Bun_B^\lambda = \bBun_B^\lambda \cap \Bun_B.
\]

\begin{rmk}
\label{rmk:bBunB}
\phantomsection
\begin{enumerate}
\item
\label{it:bBunB-Hom}
 The stack $\overline{\Bun}_B$ can also be described as a ``Hom'' stack as follows. Consider the basic affine space $\overline{G/U} = \Spec(\scO(G/U))$, with the natural actions of $G$ on the left and of $T$ on the right. The quotient stack $[G \backslash \overline{G/U} / T]$ has a natural open substack given by $[G \backslash (G/U) / T] \cong [\mathrm{pt}/B]$. From this point of view, $\overline{\Bun}_B$ identifies with the open substack of the algebraic stack $\mathrm{Maps}(C, [G \backslash \overline{G/U} / T])$ of maps from $C$ to $[G \backslash \overline{G/U} / T]$ given by maps which generically take values in $[G \backslash (G/U) / T]$. (Here, algebraicity of the Hom stack is guaranteed by~\cite[Theorem~1.2]{hr}.) For more on this point of view, see~\cite[\S 2]{raskin1}.
 \item
 \label{it:bBunB-Laumon}
 In case $G=\mathrm{GL}_n$, there is another compactification of the stack $\Bun_B$, defined by Laumon (see e.g.~\cite[\S 0.2.1]{bg} for a discussion of this construction). The latter compactification has the advantage of being smooth, contrary to Drinfeld's compactification, but it has no analogue for other groups. The two constructions coincide when $G=\mathrm{GL}_2$, so that in this case $\bBun_B$ is smooth.
 \end{enumerate}
\end{rmk}

%

\subsection{Stratification}
\label{ss:stratif-bBun}

The stack $\bBun_B$ admits a canonical stratification, defined as follows. 
As explained in~\cite[Proposition~1.2.7]{bg}, for any $\mu \in \bY_{\succeq 0}$ we have a canonical locally closed immersion
\begin{equation}
\label{eqn:immersion-stratum-bBunB}
 C^\mu \times \Bun_B \to \bBun_B
\end{equation}
defined as follows.
Consider a point $E \in C^\mu(S)$ (see~\S\ref{ss:graded-sym-powers})
and a $B$-torsor $\cG$ on $C \times S$. Let $(\cE,\cF,\kappa)$ be the image of $\cG$ under~\eqref{eqn:immersion-BunB-bBunB}. Then, using the notation introduced in~\S\ref{ss:torsors}, the image of $(E,\cG)$
is the triple $(\cE,\cF (-E),\kappa')$ where $\kappa'_\xi$ is the composition
\[
 (\cF(-E))_{\F_T(\xi)} = \cF_{\F_T(\xi)} \otimes_{\scO_{C \times S}} \scO_{C \times S}(-\langle \xi, E \rangle) \hookrightarrow
 \cF_{\F_T(\xi)} \xrightarrow{\kappa_\xi} \cE_{\mathsf{M}(\xi)}
\]
where the first map is the tensor product with $\cF_{\F_T(\xi)}$ of the natural embedding~\eqref{eqn:divisor-sheaf}.

In case $\mu=0$ we have $C^0=\Spec(\F)$, and~\eqref{eqn:immersion-stratum-bBunB} coincides with the morphism~\eqref{eqn:immersion-BunB-bBunB}. It is clear that for any $\lambda \in \bY$ and $\mu \in \bY_{\succeq 0}$ the map~\eqref{eqn:immersion-stratum-bBunB} restricts a to map
\[
 C^\mu \times \Bun_B^\lambda \to \bBun_B^{\lambda+\mu},
\]
see~\eqref{eqn:twist-T-bundles}.


The image of~\eqref{eqn:immersion-stratum-bBunB} will be denoted ${}_{\mu} \bBun_B$. Then by~\cite[Proposition~1.2.5]{bg} we have a decomposition
\[
 \bBun_B = \bigsqcup_{\mu \in \bY_{\succeq 0}} {}_{\mu} \bBun_B.
\]
Here ${}_{\mu} \bBun_B$
might not be smooth, since $C^\mu$ is not smooth in general. To remedy this we use the stratification~\eqref{eqn:stratif-symprod-Y}. Namely, 
given $\Gamma \in \Part(\mu)$ we denote by ${}_\Gamma \bBun_B$
the image of the restriction of~\eqref{eqn:immersion-stratum-bBunB} to $C^\mu_\Gamma \times \Bun_B$. Then we have a stratification
\begin{equation}
\label{eqn:stratification-bBunB}
 \bBun_B = \bigsqcup_{\mu \in \bY_{\succeq 0}} \bigsqcup_{\Gamma \in \Part(\mu)} 
{}_\Gamma \bBun_B,
\end{equation}
where ${}_\Gamma \bBun_B$ is a locally closed \emph{smooth} substack for any $\Gamma$, such that
\begin{equation}
\label{eqn:strata-bBun-isom}
 {}_\Gamma \bBun_B \cong C^\mu_\Gamma \times \Bun_B.
\end{equation}

In terms of connected components, for any $\lambda \in \bY$, $\mu \in \bY_{\succeq 0}$ and $\Gamma \in \Part(\mu)$ we set
\[
{}_\Gamma \bBun_B^\lambda = {}_\Gamma \bBun_B \cap \bBun_B^\lambda.
\]
We then have
\[
\bBun_B^{\lambda} = \bigsqcup_{\mu \in \bY_{\succeq 0}} \bigsqcup_{\Gamma \in \Part(\mu)} 
{}_\Gamma \bBun_B^\lambda
\]
where
\begin{equation}
\label{eqn:strata-bBun-isom-lambda}
{}_\Gamma \bBun_B^\lambda \cong C^\mu_\Gamma \times \Bun_B^{\lambda-\mu}
\end{equation}
is smooth of dimension
\[
|\Gamma| + \langle \lambda-\mu, 2\rho \rangle + (g-1) \dim(B).
\]


\subsection{The stack \texorpdfstring{$\bBun_U$}{BunU}}
\label{ss:BunU}

Let $U$ be the unipotent radical of $B$; then we have
\begin{equation}
\label{eqn:BunU-fiber-product}
\Bun_U = \Bun_B \times_{\Bun_T} \{ \cF^0 \}.
\end{equation}
Here $\Bun_U$ is a smooth connected stack, of dimension $(g-1) \dim(U)$.

We set
\[
\bBun_U := \bBun_B \times_{\Bun_T} \{ \cF^0 \}.
\]
In the moduli description of~\S\ref{ss:def-bBunB}, this means that we impose that $\cF$ is the trivial $T$-torsor, so that $\cF_{\F_T(\xi)} = \mathscr{O}_{C \times S}$ for any $\xi$. 

Setting, for $\mu \in \bY_{\succeq 0}$ and $\Gamma \in \Part(\mu)$,
\[
 {}_\Gamma \bBun_U = {}_\Gamma \bBun_B \times_{\Bun_T} \{\cF^0\},
\]
we have
\[
 \bBun_U = \bigsqcup_{\mu \in \bY_{\succeq 0}} \bigsqcup_{\Gamma \in \Part(\mu)} 
{}_\Gamma \bBun_U.
\]
For any $\mu,\Gamma$ as above we have
\[
 {}_\Gamma \bBun_U \cong C^\mu_\Gamma \times_{\Bun_T} \Bun_B
\]
where the map $C^\mu_\Gamma \to \Bun_T$ is given by $E \mapsto \cF^0(E)$. In particular, ${}_\Gamma \bBun_U$ is smooth, of dimension $|\Gamma| - \langle \mu,2\rho\rangle + (g-1) \dim(U)$.

\subsection{Some complexes of sheaves}
\label{ss:complexes-sheaves}

We now fix a field of coefficients $\bk$ for our sheaves. We assume that we are in one of the following settings:
\begin{enumerate}
 \item 
 $\F$ is arbitrary, $\ell$ is a prime number invertible in $\F$, $\bk$ is either a finite extension or an algebraic closure of $\F_\ell$ or $\Q_\ell$, and we consider \'etale $\bk$-sheaves on separated schemes of finite type over $\F$;
 \item
 $\F=\C$, $\bk$ is arbitrary, and we work with sheaves for the analytic topology on separated schemes of finite type over $\F$.
\end{enumerate}
In each case, for any prestack locally of finite type $X$ over $\F$ we have an $\infty$-category $\Shv(X)$ of $\bk$-sheaves on $X$, see~\cite[\S 2.1]{adr}. (This construction is an obvious translation of the more familiar case when $\bk$ has characteristic $0$, which is prominent in the geometric Langlands literature, see e.g.~\cite[\S 0.5]{gaitsgory}. In the present subsection we could work with the usual constructible sheaves formalism, but when making contact with~\cite{adr} in Section~\ref{sec:constructibility} it will be important to use the $\infty$-categorical version.) In this formalism we have $!$-pullback functors associated with any morphism of prestacks locally of finite type; for a discussion of other functors, see~\cite[\S 2.2]{adr} and the references therein.

Below we will consider the intersection cohomology complex $\IC_{\Bun_B}$ of the smooth stack $\Bun_B$ with coefficients in $\bk$. By smoothness and the dimension formula in~\S\ref{ss:stacks-torsors}, we have
\[
 \IC_{\Bun_B} = \bigoplus_{\lambda \in \bY} \underline{\bk}_{\Bun_B^\lambda} [(g-1)\dim(B) + \langle \lambda,2\rho \rangle].
\]
We will also consider the intermediate extension $\IC_{\bBun_B} = \jmath_{!*} \IC_{\Bun_B}$, and the $!$-extension $\jmath_{!} \IC_{\Bun_B}$ (where $\jmath$ is as in~\eqref{eqn:immersion-BunB-bBunB}). Here $\IC_{\bBun_B}$ is the direct sum of the intersection cohomology complexes $\IC_{\bBun_B^\lambda}$ where $\lambda$ runs over $\bY$.

The following property is proved in~\cite[Theorem~5.1.5]{bg} in the case of $\overline{\mathbb{Q}}_\ell$-coefficients; since it only relies on some geometric properties of the stack $\bBun_B$, the same proof applies in our case.

\begin{prop}
\label{prop:ULA}
 The complexes $\IC_{\bBun_B}$ and $\jmath_{!} \IC_{\Bun_B}$ are 
 universally locally acyclic with respect to the morphism~\eqref{eqn:morph-bBunB-BunT}.
\end{prop}

We will also denote by $\IC_{\bBun_U}$ the intersection cohomology complex associated with the constant local system on the smooth open substack $\Bun_U \subset \bBun_U$.
Let us denote by
\[
 f : \bBun_U \to \bBun_B
\]
the canonical map.

\begin{lem}
\label{lem:pullback-IC-BunBU}
 There exists a canonical isomorphism
 \[
  \IC_{\bBun_U} \cong f^* \IC_{\bBun_B} [-(g-1)\dim(T)].
 \]
\end{lem}

\begin{proof}
 This follows from~\cite[Lemma~7.1.3]{bg}, using Proposition~\ref{prop:ULA} and the fact that the map $\Bun_B \to \Bun_T$ is smooth (see~\S\ref{ss:stacks-torsors}).
\end{proof}

\begin{rmk}
In case $C=\bbP^1$, Lemma~\ref{lem:pullback-IC-BunBU} has a much easier proof. Namely, we have $\Bun_T^0 = \Spec(\F)/T$, so that the morphism $\{\cF^0\} \to \Bun_T$ is smooth. Hence $f$ is smooth as well, which implies that $\IC_{\bBun_U} \cong f^* \IC_{\bBun_B} [-(g-1)\dim(T)]$ by standard properties of smooth pullback functors.
\end{rmk}

\section{Zastava stacks and schemes}
\label{sec:zastavas}

We continue with our smooth projective curve $C$ over $\F$.

\subsection{Definition}
\label{ss:def-Zas}

Let $\mu \in \bY_{\succeq 0}$.
We consider the stack $\Zas^\mu$ 
which sends a scheme $S$ to the groupoid whose objects are the following data:
\begin{itemize}
 \item a point $E \in C^\mu(S)$ (see~\S\ref{ss:graded-sym-powers});
 \item a $T$-torsor $\cF$ on $C \times S$;
 \item a $G$-torsor $\cE$ on $C \times S$;
 \item an isomorphism $\beta$ between the restrictions of $\cE$ and $\cF \times^T G$ to $(C \times S) \smallsetminus E$
\end{itemize}
which satisfy the following condition. Consider for $\xi \in \bX_+$ the following commutative diagram:
\begin{equation}
\label{eqn:diag-Zas}
\begin{tikzcd}[row sep=small, column sep=small]
((\cF(-E))_{\F_T(\xi)})_{|(C \times S) \smallsetminus E} \ar[rr, hook, "\eqref{eqn:divisor-sheaf}"] \ar[dd, "\eqref{eqn:divisor-sheaf-triv}"', "\wr"] \ar[rd] &&
(\cF_{\F_T(\xi)})_{|(C \times S) \smallsetminus E} \ar[dd, equal] \\
& (\cE_{\mathsf{M}(\xi)})_{|(C \times S) \smallsetminus E} \ar[ru] \\
(\cF_{\F_T(\xi)})_{|(C \times S) \smallsetminus E} \ar[rr, equal] \ar[dr, "\eqref{eqn:morph-character-Weyl}"'] && (\cF_{\F_T(\xi)})_{|(C \times S) \smallsetminus E} \\
& (\cF_{\mathsf{M}(\xi)})_{|(C \times S) \smallsetminus E} \ar[uu, leftarrow, crossing over, "\beta" near end, "\wr"' near end] \ar[ur, "\eqref{eqn:morph-Weyl-character}"'] 
\end{tikzcd}
\end{equation}
Then we require that the top face of this diagram admit (for any $\xi \in \bX_+$) an extension to a diagram of sheaves
\[
\begin{tikzcd}[row sep=small]
(\cF(-E))_{\F_T(\xi)} \ar[rr, hook, "\eqref{eqn:divisor-sheaf}"] \ar[dr, "\kappa_\xi"'] && \cF_{\F_T(\xi)}  \\
& \cE_{\mathsf{M}(\xi)} \ar[ur, two heads, "\tau_\xi"']
\end{tikzcd}
\]
in which $\tau_\xi$ is a surjective map of vector bundles.

By construction there is a natural morphism of stacks
\begin{equation}
\label{eqn:morph-Zas-BunT}
\Zas^\mu \to C^\mu \times \Bun_T
\end{equation}
sending $(E,\cF,\cE, \beta)$ to $(E, \cF)$.
This map has a canonical section
\[
s^\mu : C^\mu \times \Bun_T \to \Zas^\mu
\]
sending a pair $(E, \cF)$ to the data $(E,\cF, \cF \times^T G,\mathrm{id})$. 

We will denote by
\[
p_1 : \Zas^\mu \to \Bun_T
\]
the composition of~\eqref{eqn:morph-Zas-BunT} with projection on $\Bun_T$. There is \emph{another} morphism
\[
p_2 : \Zas^\mu \to \Bun_T,
\]
sending $(E,\cF,\cE, \beta)$ to $\cF(-E)$. (In other words, $p_2$ is the composition of~\eqref{eqn:morph-Zas-BunT} with the morphism $C^\mu \times \Bun_T \to \Bun_T$ defined by $(E,\cF) \mapsto \cF (-E)$.)

\begin{rmk}
\begin{enumerate}
 \item 
The stack $\Zas^\mu$ is denoted $Z^\mu_{\Bun_T}$ in~\cite{bfgm}. There exist other versions of this stack that are sometimes considered (but which will not appear below), see e.g.~\cite[\S 20.1.2]{gaitsgory-lysenko}, in particular one which omits the condition that the morphisms $\tau_\xi$ are surjective. Proposition~\ref{prop:Zastava-scheme} below has a variant in this setting, where the analogue of~\eqref{eqn:morph-Zas-BunT} is now representable by schemes and \emph{projective}.
\item
As explained e.g.~in~\cite[\S 2.10]{raskin1},
Zastava stacks can also be described from the point of view of Remark~\ref{rmk:bBunB}\eqref{it:bBunB-Hom}. Namely, the disjoint union of the stacks $\Zas^\mu$ identifies with the open substack in the stack $\mathrm{Maps}(C, [B^- \backslash \overline{G/U} / T])$ of maps from $C$ to $[B^- \backslash \overline{G/U} / T]$ which generically take values in the open substack $[B^- \backslash (B^- B/U) / T] \cong [\mathrm{pt}/T]$.
\end{enumerate}
\end{rmk}

The following statement is probably known to experts. It is a variation on results that appear in the literature (see in particular~\cite{bfgm} or~\cite[\S\S20.1--20.2]{gaitsgory-lysenko}), but we have not found it explicitly in published references. (The proof uses some results and constructions that will be recalled in detail later in this section.)

\begin{prop}
\label{prop:Zastava-scheme}
 The morphism~\eqref{eqn:morph-Zas-BunT} is representable by schemes, and affine of finite type.
\end{prop}

\begin{proof}
 Consider the Hecke stack $\mathbf{Hk}_{G,C^\mu}$, whose groupoid of $S$-points consists of $4$-tuples $(E,\cE_1,\cE_2, \gamma)$ where $E \in C^\mu(S)$, $\cE_1,\cE_2$ are $G$-torsors on $C \times S$, and $\gamma : (\cE_1)_{|(C \times S) \smallsetminus E} \simto (\cE_2)_{|(C \times S) \smallsetminus E}$ is an isomorphism of $G$-bundles. By a variation on standard arguments (see e.g.~\cite[\S 5.2.1]{bd}), one checks that $\mathbf{Hk}_{G,C^\mu}$ is an ind-algebraic stack. We also have a canonical morphism
 \begin{equation}
 \label{eqn:morph-Hk-Bun}
  \mathbf{Hk}_{G,C^\mu} \to C^\mu \times \Bun_G,
 \end{equation}
which sends the $4$-tuple $(E,\cE_1,\cE_2,\gamma)$ to the pair $(E,\cE_1)$. 

We claim that~\eqref{eqn:morph-Hk-Bun} is representable by ind-schemes, in the sense that for any scheme $S$ and any morphism $S \to C^\mu \times \Bun_G$ the fiber product
\begin{equation}
\label{eqn:fiber-product-Hk}
 S \times_{C^\mu \times \Bun_G} \mathbf{Hk}_{G,C^\mu}
\end{equation}
is an ind-scheme. 
(This fact is probably known, although we have not found it in the literature; closely related statements appear e.g.~in~\cite[Remark~5.2.2(ii)]{bd},~\cite[Appendix~A]{varshavsky},~\cite[\S 2]{lafforgue}.) Fix a scheme $S$ endowed with a morphism $S \to C^\mu \times \Bun_G$. This datum provides a point $E \in C^\mu(S)$, and a $G$-torsor $\cE$ on $C \times S$. Recall the Be{\u\i}linson--Drinfeld affine Grassmannian $\Gr_{G,C^\mu}$ over $C^\mu$, which is an ind-scheme whose $S'$-points (for $S'$ a scheme) are triples $(F,\cE',\beta)$, where $F \in C^\mu(S')$, $\cE'$ is a $G$-torsor on $C \times S'$, and $\beta$ is a trivialization of $\cE'$ over $\widehat{F}^\circ$. (Here we use the notation introduced in~\S\ref{ss:indep-curve} below.) We also have the loop group $\Loop_{C^\mu} G$, which is a group ind-scheme over $C^\mu$ classifying pairs of a point $F \in C^\mu(S')$ and a map $\widehat{F}^\circ \to G$. Using $E \in C^\mu(S)$, we can consider the group ind-scheme $S \times_{C^\mu} \Loop_{C^\mu} G$ over $S$, whose $S'$-points (for $S'$ an $S$-scheme) are maps $\widehat{E \times_S S'}^\circ \to G$, or equivalently, maps of $S'$-schemes $\widehat{E \times_S S'}^\circ \to S' \times G$.  Next, using the $G$-torsor $\cE$, we can define a torsor $\widetilde{\cE}$ for $S \times_{C^\mu} \Loop_{C^\mu} G$, whose $S'$-points (for $S'$ an $S$-scheme) consist of maps $\widehat{E \times_S S'}^\circ \to \cE \times_S S'$ making the following diagram commutative:
\[
 \begin{tikzcd}
  & \cE \times_S S' \ar[d] \\
  \widehat{E \times_S S'}^\circ \ar[ru] \ar[r] & C \times S'.
 \end{tikzcd}
\]
Then standard arguments show that the fiber product~\eqref{eqn:fiber-product-Hk} identifies with the twisted product
\[
 \widetilde{\cE} \times_S^{S \times_{C^\mu} \Loop_{C^\mu} G} (S \times_{C^\mu} \Gr_{G,C^\mu}) \cong \widetilde{\cE} \times_{C^\mu}^{\Loop_{C^\mu} G} \Gr_{G,C^\mu},
\]
and that the latter is an ind-scheme.

Consider now the fiber product
\[
 \Gr_{G,C^\mu, \Bun_T} = \mathbf{Hk}_{G,C^\mu} \times_{\Bun_G} \Bun_T.
\]
This stack classifies the data of a point $E \in C^\mu(S)$, a $T$-torsor $\cF$ on $C \times S$, a $G$-torsor $\cE$ on $C \times S$, and an isomorphism between the restrictions of $\cE$ and $\cF \times^T G$ to $(C \times S) \smallsetminus E$; it can therefore be considered as a ``relative version of $\Gr_G$ over $C^\mu \times \Bun_T$.'' 
It is endowed with a canonical morphism to $C^\mu \times \Bun_G \times \Bun_T$, and by what we have established above, the morphism $\Gr_{G,C^\mu, \Bun_T} \to C^\mu \times \Bun_T$ is representable by ind-schemes.
There exists a canonical morphism
\begin{equation}
\label{eqn:morph-Zas-Hecke}
 \Zas^\mu \to \Gr_{G,C^\mu, \Bun_T},
\end{equation}
which is representable by a locally closed immersion; in fact, as e.g.~in~\cite[\S 12.2.2]{gaitsgory-lysenko} the condition that the morphisms in~\eqref{eqn:diag-Zas} extend to morphisms of sheaves on $C \times S$ is a closed condition, and the condition that each $\tau_\xi$ is a surjection of vector bundles is an open condition therein. These facts imply that~\eqref{eqn:morph-Zas-BunT} is representable by ind-schemes.

To show that this morphism is representable by \emph{schemes} and of finite type we use the fact that $\Zas^\mu$ is an algebraic stack locally of finite type, as follows from the relation with the stack $\bBun_B$ recalled in~\S\ref{ss:relation-Zastava-Drinfeld} below; see~\cite[Comments after Proposition~3.2]{bfgm} or~\cite[Corollary~20.2.3]{gaitsgory-lysenko}. 

Finally, the fact that our morphism is affine essentially boils down to the fact that the semi-infinite orbits $\rS^-_\lambda \subset \Gr$ (see~\S\ref{ss:si-sheaves} below) are ind-affine. Namely, we have a stack $\Gr_{B^-,C^\mu, \Bun_T}$ defined as above with $G$ replaced by $B^-$, which is endowed with a canonical morphism $\Gr_{B^-,C^\mu, \Bun_T} \to C^\mu \times \Bun_T$ which is representable by ind-schemes, and moreover ind-affine. By considerations similar to those in the comments following~\eqref{eqn:immersion-BunB-bBunB} (now for $B^-$ instead of $B$, see~\S\ref{ss:relation-Zastava-Drinfeld}) the morphism~\eqref{eqn:morph-Zas-Hecke} factors through a morphism $\Zas^\mu \to \Gr_{B^-,C^\mu, \Bun_T}$, which is representable by a closed immersion. (This morphism takes values in the closed substack $\Gr_{B^-,C^\mu, \Bun_T} \times_{\Bun_{B^-} \times \Bun_T} \Bun_{B^-}$ of $\Gr_{B^-,C^\mu, \Bun_T}$, where the morphism $\Bun_{B^-} \to \Bun_T$ is as in~\eqref{eqn:morph-Bun-B-T}.)
The ind-affineness property stated above then implies that our morphism is affine.
\end{proof}

\begin{rmk}
The proof of Proposition~\ref{prop:Zastava-scheme} given above essentially repeats the classical proof of the fact that the stack $\uZas^\mu$ defined in~\S\ref{sss:classical-Zastava} is a scheme of finite type (affine over $C^\mu$) in a ``relative'' setting over $\Bun_T$. Alternatively, these properties can be deduced from the latter facts using Proposition~\ref{prop:local-isom} below.
\end{rmk}

\subsection{Some fibers}

We continue with our $\mu \in \bY_{\succeq 0}$. In this subsection we introduce two slightly different schemes obtained from $\Zas^\mu$ that appear in many references. See Remark~\ref{rmk:local-isom-uZas-uZasG} below for a relation between these two schemes.

\subsubsection{The ``classical'' Zastava scheme}
\label{sss:classical-Zastava}

We set
\[
\uZas^\mu := \{\cF^0 \} \times_{\Bun_T, p_1} \Zas^\mu.
\]
By Proposition~\ref{prop:Zastava-scheme}, $\uZas^\mu$ is
a scheme of finite type over $\F$. It is equipped with a canonical morphism
$\uZas^\mu \to C^\mu$, which is affine.
Concretely, this scheme classifies the following data:
\begin{itemize}
 \item a point $E$ of $C^\mu(S)$;
 \item a $G$-torsor $\cE$ on $C \times S$;
 \item a trivialization $\beta$ of $\cE_{|(C \times S) \smallsetminus E}$
\end{itemize}
which satisfies the condition that, for any $\xi \in \bX_+$, the top face of the commutative diagram
\begin{equation}
\label{eqn:diag-uZas}
\hspace{-4em}
\begin{tikzcd}[row sep=small, column sep=tiny]
\F_T(\xi) \otimes \scO_{C \times S}(-\langle \xi,E \rangle)_{|(C \times S) \smallsetminus E} \ar[rr, hook, "\eqref{eqn:divisor-sheaf}"] \ar[rd] \ar[dd, "\eqref{eqn:divisor-sheaf-triv}"', "\wr"] &&
\F_T(\xi) \otimes \scO_{(C \times S) \smallsetminus E} \ar[dd, equal] \\
& (\cE_{\mathsf{M}(\xi)})_{|(C \times S) \smallsetminus E} \ar[ru] \\
\F_T(\xi) \otimes \scO_{(C \times S) \smallsetminus E} \ar[rr, equal] \ar[dr, "\eqref{eqn:morph-character-Weyl}"'] && \F_T(\xi) \otimes \scO_{(C \times S) \smallsetminus E} \\
& \mathsf{M}(\xi) \otimes \scO_{(C \times S) \smallsetminus E}\ar[uu, leftarrow, crossing over, "\beta" near end, "\wr"' near end] \ar[ur, "\eqref{eqn:morph-Weyl-character}"'] 
\end{tikzcd}
\hspace{-4em}
\end{equation}
admits an extension to a diagram of sheaves
\begin{equation}\label{eqn:uZas-cond}
\begin{tikzcd}[row sep=small]
\F_T(\xi) \otimes \scO_{C \times S}(-\langle \xi, E \rangle) \ar[rr, hook, "\eqref{eqn:divisor-sheaf}"] \ar[dr, "\kappa_\xi"'] && \F_T(\xi) \otimes \scO_{C \times S}  \\
& \cE_{\mathsf{M}(\xi)} \ar[ur, two heads, "\tau_\xi"']
\end{tikzcd}
\end{equation}
in which $\tau_\xi$ is a surjective map of vector bundles.


\begin{rmk}
The scheme $\uZas^\mu$ is
denoted $Z^\mu$ in~\cite{bfgm}. In some references (in particular,~\cite{gaitsgory-lysenko}) the authors ``gather'' together all these schemes, and consider the disjoint union $\sqcup_{\mu \in \bY_{\succeq 0}} \uZas^\mu$, a scheme over the ``configuration space'' $\mathrm{Conf} = \sqcup_{\mu \in \bY_{\succeq 0}} C^\mu$. This version corresponds to the scheme denoted $\mathscr{Z}$ in~\cite[\S 20.1.2]{gaitsgory-lysenko}.
\end{rmk}

\subsubsection{A variant}
\label{sss:Gaitsgory-Zas}

We also set
\[
\uZasG^\mu = \{\cF^0 \} \times_{\Bun_T, p_2} \Zas^\mu.
\]
Here again, by Proposition~\ref{prop:Zastava-scheme}, $\uZasG^\mu$ is a scheme of finite type over $\F$, equipped with a canonical affine morphism
$\uZasG^\mu \to C^\mu$.
%
%
Concretely, this scheme classifies the following data:
\begin{itemize}
 \item a point $E$ of $C^\mu(S)$;
 \item a $G$-torsor $\cE$ on $C \times S$;
 \item an isomorphism $\beta$ between $\cE_{|(C \times S) \smallsetminus E}$ and $(\cF^0(E) \times^T G)_{|(C \times S) \smallsetminus E}$
\end{itemize}
which satisfies the condition that, for any $\xi \in \bX_+$, the top face of the commutative diagram
{\small
\begin{equation}
\label{eqn:diag-uZasG}
\hspace{-3em}
\begin{tikzcd}[row sep=small, column sep=-5pt]
\F_T(\xi) \otimes \scO_{(C \times S) \smallsetminus E} \ar[rr, hook, "\eqref{eqn:divisor-sheaf}"] \ar[dd, "\eqref{eqn:divisor-sheaf-triv}"', "\wr"] \ar[rd] &&
\F_T(\xi) \otimes \scO(\langle \xi, E \rangle)_{|(C \times S) \smallsetminus E} \ar[dd, equal] \\
& (\cE_{\mathsf{M}(\xi)})_{|(C \times S) \smallsetminus E} \ar[ru] \\
\F_T(\xi) \otimes \scO(\langle \xi, E \rangle)_{|(C \times S) \smallsetminus E} \ar[rr, equal] \ar[dr, "\eqref{eqn:morph-character-Weyl}"'] && \F_T(\xi) \otimes \scO(\langle \xi, E \rangle)_{|(C \times S) \smallsetminus E} \\
& \mathsf{M}(\xi) \otimes \scO(\langle \xi, E \rangle)_{|(C \times S) \smallsetminus E} \ar[uu, leftarrow, crossing over, "\beta" near end, "\wr"' near end] \ar[ur, "\eqref{eqn:morph-Weyl-character}"'] 
\end{tikzcd}
\hspace{-4em}
\end{equation}
}(where we write $\scO(\langle \xi, E \rangle)$ for $\scO_{C \times S}(\langle \xi, E \rangle)$ so that the diagram can fit the page)
admits an extension to a diagram of sheaves
\[
\begin{tikzcd}[row sep=small]
\F_T(\xi) \otimes \scO_{C \times S} \ar[rr, hook, "\eqref{eqn:divisor-sheaf}"] \ar[dr, "\kappa_\xi"'] && \F_T(\xi) \otimes \scO_{C \times S}(\langle \xi, E \rangle)  \\
& \cE_{\mathsf{M}(\xi)} \ar[ur, two heads, "\tau_\xi"']
\end{tikzcd}
\]
in which $\tau_\xi$ is a surjective map of vector bundles.

\begin{rmk}
 \begin{enumerate}
 \item
The scheme $\uZasG^\mu$ is denoted $\mathscr{Z}^{-\mu}$ in~\cite[\S 3.6.1]{gaitsgory}.
\item
In terms of the morphism~\eqref{eqn:morph-Zas-BunT}, 
we have
\begin{equation}
\label{eqn:ZasG-fiber-p1}
\uZasG^\mu = C^\mu \times_{C^\mu \times \Bun_T} \Zas^\mu
\end{equation}
where the morphism $C^\mu \to C^\mu \times \Bun_T$ is given by $E \mapsto (E,\cF^0(E))$.
\end{enumerate}
\end{rmk}

\subsection{Independence of the curve}
\label{ss:indep-curve}

In this subsection, we will compare Zastava schemes attached to different curves.  When there is a risk of ambiguity, we include the curve $C$ in the notation, and write
\[
\uZas^{C,\mu}
\]
for the scheme defined in~\S\ref{sss:classical-Zastava} using the curve $C$.

It is a useful observation that the definition of $\uZas^{C,\mu}$ in~\S\ref{sss:classical-Zastava} makes sense when $C$ is a \emph{possibly nonprojective} smooth curve $C$ over $\F$. If fact, recall that any smooth curve $C$ over $\F$ admits a canonical smooth compactification $\overline{C}$; more specifically there exists a unique (up to isomorphism) projective smooth curve containing $C$ as an open subscheme.\footnote{Let us recall that this fact is a consequence of 3 properties of curves: (a) a scheme which is of pure dimension $1$ is normal iff it is regular, cf.~\cite[\S 15.1]{gw}; (b) over a perfect field, a regular scheme is smooth, see~\cite[Remark~6.33]{gw}; (c) each normal curve over a field admits a canonical normal compactification, see~\cite[Theorem~15.21]{gw}.}
Then one can set
\begin{equation}
\label{eqn:uZas-open-curve}
 \uZas^{C,\mu} = C^\mu \times_{\overline{C}^\mu} \uZas^{\overline{C},\mu}.
\end{equation}
In this subsection we therefore drop the assumption that $C$ is projective.

\begin{rmk}
\label{rmk:Zastava-A1}
 The Zastava schemes $\uZas^{\mathbb{A}^1,\mu}$ are the main objects of study of~\cite{fm,ffkm}.
\end{rmk}

Below we will need a generalization of~\eqref{eqn:uZas-open-curve} which is probably known to experts. Suppose $\varphi: C \to D$ is an \'etale map of (possibly nonprojective) curves, and recall the notation of~\S\ref{ss:graded-sym-powers}.

\begin{lem}
\label{lem:uZas-compare}
Let $\varphi: C \to D$ be an \'etale map of curves.  For any $\mu \in \bY_{\succeq 0}$, there is an isomorphism of schemes
\[
C^\mu_\et \times_{C^\mu} \uZas^{C,\mu} \cong C^\mu_\et \times_{D^\mu} \uZas^{D,\mu}.
\]
\end{lem}

To prove this lemma we will use a ``local'' description of $\uZas^{C,\mu}$, as follows. Given a smooth $\F$-curve $C$, an $\F$-scheme $S$ and a point $E \in C^\mu(S)$, we can consider the scheme
\[
\widehat{E} = \underline{\Spec} \Big(\varprojlim_n \scO_{C\times S}/\scO_{C \times S}(-\la 2n\rho, E\ra)\Big),
\]
where ``$\underline{\Spec}$'' denotes the relative $\Spec$ construction over $C \times S$.  This scheme comes with a flat, affine morphism $\widehat{E} \to C \times S$. Each $E_\alpha$ is naturally an effective Cartier divisor on $\widehat E$, and it makes sense to consider the open subscheme
\[
\widehat{E}^\circ := \widehat{E} \smallsetminus \bigcup_{\alpha \in \fRs^\vee} E_\alpha.
\]
For $\xi \in \bX_+$, we also have the invertible sheaf $\scO_{\widehat{E}}(-\la \xi, E\ra)$.

Suppose now that $\cE$ is a $G$-torsor on $\widehat{E}$, equipped with a trivialization $\beta$ over $\widehat{E}^\circ$.  For any $\xi \in \bX_+$, we have the following counterpart of~\eqref{eqn:diag-uZas}:
\begin{equation}
\label{eqn:diag-uZas-completion}
\begin{tikzcd}[row sep=small]
\F_T(\xi) \otimes \scO_{\widehat{E}}(-\langle \xi, E \rangle)_{|\widehat{E}^\circ} \ar[rr, hook, "\eqref{eqn:divisor-sheaf}"] \ar[dd, "\eqref{eqn:divisor-sheaf-triv}"', "\wr"] \ar[rd] &&
\F_T(\xi) \otimes \scO_{\widehat{E}^\circ} \ar[dd, equal] \\
& (\cE_{\mathsf{M}(\xi)})_{|\widehat{E}^\circ} \ar[ru] \\
\F_T(\xi) \otimes \scO_{\widehat{E}^\circ} \ar[rr, equal] \ar[dr, "\eqref{eqn:morph-character-Weyl}"'] && \F_T(\xi) \otimes \scO_{\widehat{E}^\circ}. \\
& \mathsf{M}(\xi) \otimes \scO_{\widehat{E}^\circ}\ar[uu, leftarrow, crossing over, "\beta" near end, "\wr"' near end] \ar[ur, "\eqref{eqn:morph-Weyl-character}"'] 
\end{tikzcd}
\end{equation}
%
Using standard considerations (based on~\cite[\S 2.12]{bd})
one obtains another functorial description of $\uZas^{C,\mu}$, using~\eqref{eqn:diag-uZas-completion} instead of~\eqref{eqn:diag-uZas}. Namely, an $S$-point of $\uZas^{C,\mu}$ consists of:
\begin{itemize}
\item an element $E \in C^\mu(S)$;
\item a $G$-torsor $\cE$ on $\widehat{E}$;
\item a trivialization $\beta$ of $\cE_{|\widehat{E}^\circ}$
\end{itemize}
such that for all $\xi \in \bX_+$, the top face of~\eqref{eqn:diag-uZas-completion} admits an extension to a diagram
\[
\begin{tikzcd}[row sep=small]
\F_T(\xi) \otimes \scO_{\widehat{E}}(-\la \xi,E\ra) \ar[rr, hook] \ar[dr, "\kappa_\xi"'] && \F_T(\xi) \otimes \scO_{\widehat{E}}  \\
& \cE_{\mathsf{M}(\xi)} \ar[ur, two heads, "\tau_\xi"']
\end{tikzcd}
\]
in which $\tau_\xi$ is a surjective map of vector bundles.

\begin{proof}[Proof of Lemma~\ref{lem:uZas-compare}]
Let us write down the functorial description of $C^\mu_\et \times_{D^\mu} \uZas^{D,\mu}$ based on the ``local'' description of $\uZas^{D,\mu}$ explained above. An $S$-point of $C^\mu_\et \times_{D^\mu} \uZas^{D,\mu}$ consists of the following data:
\begin{itemize}
\item an element $E \in C^\mu_\et(S)$;
\item a $G$-torsor $\cE$ on $\widehat{\varphi(E)}$;
\item a trivialization $\beta$ of $\cE_{|\widehat{\varphi(E)}^\circ}$,
\end{itemize}
such that for all $\xi \in \bX_+$, the top face of diagram~\eqref{eqn:diag-uZas-completion} (for the closed subscheme $\varphi(E) \subset D \times S$) admits an extension to a diagram
\[
\begin{tikzcd}[column sep=tiny, row sep=small]
\F_T(\xi) \otimes \scO_{\widehat{\varphi(E)}}(-\la \xi,\varphi(E)\ra) \ar[rr, hook] \ar[dr, "\kappa_\xi"'] && \F_T(\xi) \otimes \scO_{\widehat{\varphi(E)}} \\
& \cE_{\mathsf{M}(\xi)} \ar[ur, two heads, "\tau_\xi"']
\end{tikzcd}
\]
in which $\tau_\xi$ is a surjective map of vector bundles.

By the definition of $C^\mu_\et$, the morphism $\varphi$ induces an isomorphism of schemes $\sum_\alpha E_\alpha \to \sum_\alpha \varphi(E_\alpha)$.
Moreover, because $\varphi$ is \'etale, the induced map of completions $\widehat{E} \to \widehat{\varphi(E)}$ is also an isomorphism.  Thus, in the preceding paragraph, we could replace every instance of ``$\widehat{\varphi(E)}$'' by ``$\widehat{E}$'' without changing the resulting scheme.  But that replacement yields the description of $C^\mu_\et \times_{C^\mu} \uZas^{C,\mu}$.
\end{proof}

Lemma~\ref{lem:uZas-compare} is equivalent to the statement that there exists a (necessarily \'etale) map 
\[
\varphi_{\mathrm{Zas}}: C^\mu_\et \times_{C^\mu} \uZas^{C,\mu} \to \uZas^{D,\mu}
\]
that makes the right-hand square in the following diagram a pullback square:
\begin{equation}\label{eqn:uZas-compare}
\begin{tikzcd}[column sep=huge]
\uZas^{C,\mu} \ar[d] & 
C^\mu_\et \times_{C^\mu} \uZas^{C,\mu} \ar[l, "\text{open immer.}"] \ar[r, "\varphi_{\mathrm{Zas}}", "\text{\'etale}"'] \ar[d] & \uZas^{D,\mu} \ar[d] \\
C^\mu & C^\mu_\et \ar[l, "\text{open immer.}"] \ar[r, "\varphi", "\text{\'etale}"'] & D^\mu.
\end{tikzcd}
\end{equation}

\subsection{Local triviality}


From now on we assume again that $C$ is projective. The goal of this subsection is to explain the proof of the following result, a (more general) version of which is claimed without details in~\cite[\S 3.1]{bfgm}.

\begin{prop}
\label{prop:local-isom}
Fix an 
$\F$-scheme $S$ 
and a morphism $S \to \Bun_T$.  Then the schemes
\[
\Zas^\mu \times_{p_1,\Bun_T} S
\qquad\text{and}\qquad
\uZas^\mu \times S
\]
are Zariski-locally isomorphic as schemes over $C^\mu \times S$.  That is: there exists 
an open covering $(V_i : i \in I)$ of $C^\mu \times S$ and, for any $i \in I$, an isomorphism
\[
\Zas^\mu \times_{C^\mu \times \Bun_T} V_i
\cong (\uZas^\mu \times \Bun_T) \times_{C^\mu \times \Bun_T} V_i
\]
of schemes over $V_i$.
\end{prop}

The proof will use the following lemma.

\begin{lem}
\label{lem:triviality}
Assume given an
$\F$-scheme 
$S$, a line bundle $\mathscr{L}$ on $C \times S$, and a finite collection $(x_i, i \in I)$ of distinct $\F$-points in $C$. For any $\F$-point $s$ of $S$, there exists an open neighborhood $O$ of $s$ in $S$ and an open subscheme $V \subset C \times O$ containing the subschemes $\{x_i\} \times O$ such that $\mathscr{L}_{|V}$ is trivializable.
\end{lem}

\begin{proof}
 Let $C' \subset C$ be an affine open subscheme containing the $x_i$'s. For any $i$ one can consider the line bundle $\mathscr{L}_{|\{x_i\} \times S}$ on $S$; let $O \subset S$ be an affine open subscheme containing $s$ and such that each of these line bundles is trivial on $O$. Consider the natural surjection
 \[
  \Gamma(C' \times O, \mathscr{L}) \to \Gamma( \sqcup_{i \in I} \{x_i\} \times O, \mathscr{L}_{|\sqcup_i \{x_i\} \times O}). 
 \]
By assumption the right-hand side is free of rank $1$ over $\scO(\sqcup_{i \in I} \{x_i\} \times O)$. Choosing a lift of a generator to $\Gamma(C' \times O, \mathscr{L})$, we obtain a morphism $\scO_{C' \times O} \to \mathscr{L}$ whose restriction to each $\{x_i\} \times O$ is an isomorphism. If $\mathscr{F}$ is the cokernel of this morphism, $M:=\Gamma(C' \times O, \mathscr{F})$ is then a finitely generated $\scO(C' \times O)$-module such that $M \otimes_{\scO(C' \times O)} \scO(\{x_i\} \times O)=0$ for any $i$. The ideal $J \subset \scO(C' \times O)$ of $\sqcup_{i \in I} \{x_i\} \times O$ satisfies $M = J \cdot M$;
by the Nakayama lemma, it follows that there exists $a \in J$ such that $(1+a) \cdot M = 0$. 
Then the principal affine open subscheme $V \subset C' \times O$ associated with $1+a$ contains $\sqcup_{i \in I} \{x_i\} \times O$, and we have $\mathscr{F}_{|V}=0$. The restriction $\scO_{V} \to \mathscr{L}_{|V}$ of our morphism is then surjective, hence an isomorphism because both sheaves are line bundles, which finishes the proof.
\end{proof}

\begin{proof}[Proof of Proposition~\ref{prop:local-isom}]
We can assume that $S$ is affine. Then, since $\Bun_T$ is locally of finite presentation, there exists a finitely generated $\F$-subalgebra $A' \subset \scO(S)$ such that our morphism $S \to \Bun_T$ factors through a morphism $\Spec(A') \to \Bun_T$, see e.g.~\cite[Lemma~2.1.9]{eg}. Hence we can (and will) assume that $S$ is moreover of finite type over $\F$.
In this case, it is enough to show that for each $\F$-point $(x,s)$ of $C^\mu \times S$, there is an open subscheme $V \subset C^\mu \times S$ containing $(x,s)$ for which there is an isomorphism
\begin{equation}\label{eqn:Zas-local-isom}
\Zas^\mu \times_{C^\mu \times \Bun_T} V \cong (\uZas^\mu \times \Bun_T) \times_{C^\mu \times \Bun_T} V
\end{equation}
of schemes over $V$.


Fix an $\F$-point $(x,s)$ as above.  The morphism $S \to \Bun_T$ determines a $T$-torsor $\cG$ on $C \times S$. Choosing an isomorphism $T \cong (\Gm)^r$, this datum is equivalent to the datum of $r$ line bundles on $C \times S$. Let also $(x_i : i \in I)$ be the distinct points that form the support of the divisor in $C$ determined by the point $x \in C^\mu(\F)$.
Applying Lemma~\ref{lem:triviality} to each of the line bundles above, we obtain an open neighborhood $O$ of $s$ in $S$ and an open subscheme $V' \subset C \times O$ containing the subschemes $\{x_i\} \times O$ such that $\cG_{| V'}$ is trivializable.  Fix a trivialization
$\xi: \cG_{|V'} \simto \cF^0_{V'}$.

Next, consider the open subscheme
\[
\underbrace{V' \times_S V' \times_S \cdots \times_S V'}_{\text{$\langle \mu,\rho\rangle$ factors}} \subset C^{\langle \mu,\rho\rangle} \times S.
\]
By definition there exists a flat morphism $C^{\langle \mu,\rho \rangle} \to C^\mu$, and this induces a flat morphism $C^{\langle\mu,\rho\rangle} \times S \to C^\mu \times S$. Let
$V$ be the image of $(V' \times_S V' \times_S \times \cdots \times_S V')$ in $C^\mu \times S$;
then $V$ is an open subscheme of $C^\mu \times S$ which contains $(x,s)$.

We claim that with this choice of $V$ there is an isomorphism as in~\eqref{eqn:Zas-local-isom}, which will finish the proof. In fact, let $S'$ be a scheme with a map $S' \to V$.  The composition $S' \to V \to C^\mu \times S$ determines a point $x' \in C^\mu(S')$ and a map $\phi: S' \to S$.  The latter induces a map
\[
\tilde\phi: C \times S' \to C \times S.
\]
We denote by $\tilde\phi^*\xi$ the trivialization of $(\tilde\phi^*\cG)_{|\tilde\phi^{-1}(V')}$ induced by $\xi$.

An $S'$-point of $\Zas^\mu \times_{C^\mu \times \Bun_T} V$ is a $4$-tuple $(E,\cF, \cE,\beta) \in \Zas^\mu(S')$ subject to the following additional conditions:
\begin{itemize}
\item the point $E \in C^\mu(S')$ is equal to $x'$;
\item the $T$-torsor $\cF$ on $C \times S'$ is equal to $\tilde\phi^*\cG$.
\end{itemize}
Recall that $\beta$ is an isomorphism between the restrictions of $\cE$ and of $\cF \times^T G = (\tilde\phi^*\cG) \times^T G$ to $(C \times S') \smallsetminus E$.

On the other hand, an $S'$-point of $(\uZas^\mu \times \Bun_T) \times_{C^\mu \times \Bun_T} V$ is again a $4$-tuple $(E,\cF,\cE,\beta)$, but where $(E,\cE,\beta) \in \uZas^\mu(S')$ and $\cF \in \Bun_T(S')$.
These are subject to the following additional conditions:
\begin{itemize}
\item the point $E \in C^\mu(S')$ is equal to $x'$;
\item the $T$-torsor $\cF$ on $C \times S'$ is equal to $\tilde\phi^*\cG$.
\end{itemize}
Recall that $\beta$ is a trivialization of $\cE$ over $(C \times S') \smallsetminus E$.

Consider an $S'$-point $(E,\cF,\cE,\beta)$ of $\Zas^\mu \times_{C^\mu \times \Bun_T} V$, and set $U = \tilde\phi^{-1}(V') \smallsetminus E$.  Then $\beta$ and $\tilde\phi^*\xi$ together give a trivialization of $\cE_{|U}$.  Let $\cE'$ be the $G$-torsor on $C \times S'$ obtained by gluing $\cE_{|\tilde\phi^{-1}(V')}$ to the trivial bundle on $(C \times S') \smallsetminus E$ using this trivialization. The torsor $\cE'$ comes with a canonical trivialization over $(C \times S') \smallsetminus E$, which we denote by $\beta'$. Then the assignment
\[
(E,\cF,\cE,\beta) \to (E,\cF,\cE',\beta')
\]
defines a map
\[
\Zas^\mu \times_{C^\mu \times \Bun_T} V \to (\uZas^\mu \times \Bun_T) \times_{C^\mu \times \Bun_T} V.
\]
A very similar construction gives an inverse map in the opposite direction.  
\end{proof}


\begin{rmk}
\label{rmk:local-isom-uZas-uZasG}
Let $S$ be an $\F$-scheme equipped with a morphism $\phi: S \to \Bun_T$, and let $\cF$ be the $T$-torsor on $C \times S$ corresponding to $\phi$. Define a map $\tilde\phi: C^\mu \times S \to \Bun_T$ as follows. An $S'$-point of $C^\mu \times S$ is a pair of a point $E \in C^\mu(S')$ and a morphism $S' \to S$. The image of this point in $\Bun_T(S')$ is the $T$-torsor $(\cF \times_S S')(E)$.
Applying Proposition~\ref{prop:local-isom} to $\tilde\phi$ we deduce that the schemes
\begin{equation}
\label{eqn:local-isom-var}
\Zas^\mu \times_{p_1,\Bun_T} (C^\mu \times S)
\qquad\text{and}\qquad
\uZas^\mu \times C^\mu \times S 
\end{equation}
are locally isomorphic as schemes over $C^\mu \times C^\mu \times S$.

Next, let $\delta: C^\mu \to C^\mu \times C^\mu$ be the diagonal embedding. Taking the pullback of~\eqref{eqn:local-isom-var} along $\delta \times \id: C^\mu \times S \to C^\mu \times C^\mu \times S$ one obtains a local isomorphism between
\[
(\Zas^\mu \times_{p_1,\Bun_T} (C^\mu \times S)) \times_{C^\mu \times C^\mu \times S} (C^\mu \times S) \quad \text{and} \quad \uZas^\mu \times S.
\]
Now the left-hand side identifies with $\Zas^\mu \times_{p_2,\Bun_T} S$. In fact, at the level of $S'$-points, in this fiber product the element $E \in C^\mu(S')$ and the torsor $\cF'$ on $C \times S'$ in the element of $\Zas^\mu(S')$ must satisfy $\cF'=(\cF \times_S S')(E)$, i.e.~$\cF'(-E)=\cF \times_S S'$, which is exactly the condition that should satisfy an $S'$-point of $\Zas^\mu \times_{p_2,\Bun_T} S$.

In conclusion, we have proved that
\[
\Zas^\mu \times_{p_2,\Bun_T} S
\qquad\text{and}\qquad
\uZas^\mu \times S 
\]
are locally isomorphic as schemes over $C^\mu \times S$.  In the special case where $S = \Spec(\F)$ and $\phi$ is the inclusion of the trivial $T$-torsor $\cF^0$, we deduce that
\[
\uZasG^\mu
\qquad\text{and}\qquad
\uZas^\mu
\]
are locally isomorphic as schemes over $C^\mu$.
\end{rmk}

\subsection{Factorization}

Let $\mu_1,\mu_2 \in \bY_{\succeq 0}$. We denote by
\[
(C^{\mu_1} \times C^{\mu_2})_{\disj} \subset C^{\mu_1} \times C^{\mu_2}
\]
the open subscheme defined by
\[
(C^{\mu_1} \times C^{\mu_2})_{\disj}(S) = \{(E,F) \in C^{\mu_1}(S) \times C^{\mu_2}(S) \mid \forall \alpha,\beta \in \fRs^\vee, \, E_\alpha \cap F_\beta = \varnothing\}.
\]
(Here, as above, $E_\alpha$ is the $\alpha$-component of $E$, and $F_\beta$ is the $\beta$-component of $F$. We insist that the subscript is ``$\disj$,'' and not ``$\mathrm{dist}$'' as in~\S\ref{ss:symmetric-powers}; the definitions are different.)
We have a natural ``sum'' map
\[
C^{\mu_1} \times C^{\mu_2} \to C^{\mu_1 + \mu_2}
\]
sending a pair $(E,F)$ to the collection $E+F:=(E_\alpha + F_\alpha : \alpha \in \fRs^\vee)$. This morphism
restricts to an \'etale morphism
\[
(C^{\mu_1} \times C^{\mu_2})_{\disj} \to C^{\mu_1 + \mu_2}.
\]

We consider the stack
\[
\mathfrak{X}^{\mu_1,\mu_2}
\]
whose $S$-points are
$5$-tuples $(E_1,E_2,\cF_1,\cF_2, \gamma)$ with $(E_1,E_2) \in (C^{\mu_1} \times C^{\mu_2})_{\disj}(S)$, $\cF_1,\cF_2 \in \Bun_T(S)$, and $\gamma$ is an isomorphism between the restrictions of $\cF_1$ and $\cF_2$ to $(C \times S) \smallsetminus (E_1+E_2)$. There is an obvious morphism
\[
\mathfrak{X}^{\mu_1,\mu_2} \to C^{\mu_1} \times C^{\mu_2} \times \Bun_T \times \Bun_T.
\]
There is also a canonical morphism
\[
\mathfrak{X}^{\mu_1,\mu_2} \to C^{\mu_1+\mu_2} \times \Bun_T
\]
sending a datum $(E_1,E_2,\cF_1,\cF_2, \gamma)$ to the pair $(E_1+E_2,\cF)$ where $\cF$ is obtained by gluing $(\cF_1)_{| (C \times S) \smallsetminus E_2}$ and $(\cF_2)_{| (C \times S) \smallsetminus E_1}$ using $\gamma$.

The factorization property of Zastava stacks (see~\cite{fm, bfgm}) is the property that for any $\mu_1,\mu_2$ as above there exists a canonical isomorphism
\begin{equation}
\label{eqn:fact-Zas}
(\Zas^{\mu_1} \times \Zas^{\mu_2}) \times_{C^{\mu_1} \times C^{\mu_2} \times \Bun_T \times \Bun_T} \mathfrak{X}^{\mu_1,\mu_2} \cong \Zas^{\mu_1 + \mu_2} \times_{C^{\mu_1 + \mu_2} \times \Bun_T} \mathfrak{X}^{\mu_1, \mu_2}
\end{equation}
of stacks over $\mathfrak{X}^{\mu_1,\mu_2}$.

At the level of $S$-points, the map from the left-hand side to the right-hand sends a datum
\[
\bigl( (E_1, \cF_1, \cE_1, \beta_1),(E_2, \cF_2, \cE_2, \beta_2), (E_1, E_2, \cF_1, \cF_2, \gamma) \bigr)
\]
to the datum
\[
((E_1+E_2,\cF,\cE,\beta), (E_1, E_2, \cF_1, \cF_2, \gamma))
\]
where:
\begin{itemize}
\item $\cF$ is obtained by gluing $(\cF_1)_{| (C \times S) \smallsetminus E_2}$ and $(\cF_2)_{| (C \times S) \smallsetminus E_1}$ using $\gamma$;
\item $\cE$ is obtained by gluing $(\cE_1)_{| (C \times S) \smallsetminus E_2}$ and $(\cE_2)_{| (C \times S) \smallsetminus E_1}$ using the gluing datum given by the restrictions of $\beta_1$ and $\beta_2$, together with $\gamma$;
\item $\beta$ is the isomorphism induced by $\beta_1$ and $\beta_2$.
\end{itemize}
The map from the right-hand side to the left-hand side sends the datum
\[
((E_1+E_2,\cF,\cE,\beta), (E_1, E_2, \cF_1, \cF_2, \gamma))
\]
(where $\cF$ is obtained as above from $(E_1, E_2, \cF_1, \cF_2, \gamma)$)
to the datum
\[
\bigl( (E_1, \cF_1, \cE_1, \beta_1),(E_2, \cF_2, \cE_2, \beta_2), (E_1, E_2, \cF_1, \cF_2, \gamma) \bigr)
\]
where:
\begin{itemize}
\item $\cE_1$ is obtained by gluing $(\cF_1 \times^T G)_{| (C \times S) \smallsetminus E_1}$ and $\cE_{| (C \times S) \smallsetminus E_2}$ using the gluing datum $\beta$;
\item $\cE_2$ is obtained by gluing $\cE_{| (C \times S) \smallsetminus E_1}$ and $(\cF_2 \times^T G)_{| (C \times S) \smallsetminus E_2}$ using the gluing datum $\beta$;
\item $\beta_1$ and $\beta_2$ are the obvious isomorphisms.
\end{itemize}

Taking the fiber product of this isomorphism along the map
\[
(C^{\mu_1} \times C^{\mu_2})_{\disj} \to \mathfrak{X}^{\mu_1,\mu_2}
\]
sending $(E_1,E_2)$ to $(E_1,E_2, \cF^0,\cF^0, \mathrm{id})$,
we deduce an isomorphism
\begin{equation}
\label{eqn:fact-uZas}
(\uZas^{\mu_1} \times \uZas^{\mu_2}) \times_{C^{\mu_1} \times C^{\mu_2}} (C^{\mu_1} \times C^{\mu_2})_{\disj} \cong \uZas^{\mu_1 + \mu_2} \times_{C^{\mu_1 + \mu_2}} (C^{\mu_1} \times C^{\mu_2})_{\disj},
\end{equation}
which recovers~\cite[Proposition~2.4]{bfgm}.

On the other hand,
taking the fiber product along the map
\[
(C^{\mu_1} \times C^{\mu_2})_{\disj} \to \mathfrak{X}^{\mu_1,\mu_2}
\]
sending $(E_1,E_2)$ to $(E_1,E_2, \cF^0(E_1), \cF^0(E_2),\gamma)$ where $\gamma$ is the obvious isomorphism between the restrictions of $\cF_T^0(E_1)$ and $\cF_T^0(E_2)$ to $(C \times S) \smallsetminus (E_1+E_2)$, we deduce an isomorphism
\[
(\uZasG^{\mu_1} \times \uZasG^{\mu_2}) \times_{C^{\mu_1} \times C^{\mu_2}} (C^{\mu_1} \times C^{\mu_2})_{\disj} \cong \uZasG^{\mu_1 + \mu_2} \times_{C^{\mu_1 + \mu_2}} (C^{\mu_1} \times C^{\mu_2})_{\disj},
\]
which recovers~\cite[Eqn~(3.8)]{gaitsgory}.

\subsection{Relation between Zastavas and Drinfeld's compactifications}
\label{ss:relation-Zastava-Drinfeld}

For $\mu \in \bY_{\succeq 0}$ and $\lambda \in \bY$ we set
\[
\Zas^\mu_\lambda = \Bun_T^\lambda \times_{\Bun_T, p_1} \Zas^\mu = \Bun_T^{\lambda+\mu} \times_{\Bun_T, p_2} \Zas^\mu.
\]

\begin{rmk}
 In~\cite{bfgm}, the stack $\Zas^\mu_\lambda$ is denoted $Z^\mu_{\Bun_T^\lambda}$.
\end{rmk}

With this notation, for any $\mu_1,\mu_2,\lambda_1,\lambda_2$ the isomorphism~\eqref{eqn:fact-Zas} restricts to an isomorphism
\begin{equation}
\label{eqn:fact-Zas-lambda}
(\Zas^{\mu_1}_{\lambda_1} \times \Zas^{\mu_2}_{\lambda_2}) \times_{C^{\mu_1} \times C^{\mu_2} \times \Bun_T \times \Bun_T} \mathfrak{X}^{\mu_1,\mu_2} \cong \Zas^{\mu_1 + \mu_2}_{\lambda_1+\lambda_2} \times_{C^{\mu_1 + \mu_2} \times \Bun_T} \mathfrak{X}^{\mu_1, \mu_2}.
\end{equation}
We also have
\[
\uZas^\mu = \{\cF^0 \} \times_{\Bun^0_T, p_1} \Zas_0^\mu, \quad \uZasG^\mu = C^\mu \times_{C^\mu \times \Bun^{-\mu}_T} \Zas_{-\mu}^\mu = \{\cF^0 \} \times_{\Bun^0_T, p_2} \Zas_{-\mu}^\mu,
\]
see~\eqref{eqn:ZasG-fiber-p1}.
Moreover, the section $s^\mu : C^\mu \times \Bun_T \to \Zas^\mu$ (see~\S\ref{ss:def-Zas}) restricts to a section
\[
s^\mu_\lambda : C^\mu \times \Bun^\lambda_T \to \Zas^\mu_\lambda
\]
for any $\lambda \in \bY$.


By~\cite[Proposition~3.2]{bfgm}, for any $\lambda \in \bY$ and $\mu \in \bY_{\succeq 0}$ there exists a natural open immersion
\begin{equation}
\label{eqn:open-immersion-Zastava}
 \Zas^\mu_\lambda \to \bBun_B^{\lambda+\mu} \times_{\Bun_G} \Bun_{B^-}^\lambda.
\end{equation}
Here the morphism
\[
\Zas^\mu_\lambda \to \bBun_B^{\lambda+\mu}
\]
sends a datum $(E,\cE,\cF,\beta)$ as in~\S\ref{ss:def-Zas} to the triple $(\cE,\cF (-E),\kappa)$ where $\kappa_\xi$ is the morphism $(\cF (-E))_{\F_T(\xi)} \to \cE_{\mathsf{M}(\xi)}$ required to exist in the definition of $\Zas^\mu$ in~\S\ref{ss:def-Zas}. To define the morphism
\[
\Zas^\mu_\lambda \to \Bun_{B^-}^\lambda
\]
we use the identification of $\Bun_{B^-}$ with the stack parametrizing triples $(\cE,\cF,\tau)$ where $\cE$ is a $G$-torsor, $\cF$ is a $T$-torsor, and $\tau=(\tau_\xi : \xi \in \bX_+)$ is a collection of morphisms satisfying appropriate Pl\"ucker relations, where now $\tau_\xi$ is a surjective morphism of vector bundles $\cE_{\mathsf{M}(\xi)} \to \cF_{\F_T(\xi)}$. (This is similar to the case of $\Bun_B$ discussed in~\S\ref{ss:def-bBunB}; in particular the morphisms $\tau_\xi$ attached to a $B^-$-torsor are induced by the morphisms~\eqref{eqn:morph-Weyl-character}.) With this identification, the desired morphism sends $(E,\cE,\cF,\beta)$ as above to the triple $(\cE,\cF,\tau)$ where $\tau_\xi$ is the morphism $\cE_{\mathsf{M}(\xi)} \to \cF_{\F_T(\xi)}$ required to exist in the definition of $\Zas^\mu$ in~\S\ref{ss:def-Zas}.
%
%
The image of~\eqref{eqn:open-immersion-Zastava} consists of data $\bigl( (\cE,\cF,\kappa), (\cE,\cF',\tau) \bigr)$ such that the pullback to each geometric point of $S$ of each composition
\[
 \cF'_{\F_T(\xi)} \xrightarrow{\kappa_\xi} \cE_{\mathsf{M}(\xi)} \xrightarrow{\tau_\xi} \cF_{\F_T(\xi)}
\]
is nonzero.

Note that, in the notation of~\S\ref{ss:def-Zas}, the composition
\[
\Zas^\mu_\lambda \to \bBun_B^{\lambda+\mu} \times_{\Bun_G} \Bun_{B^-}^\lambda \to \bBun_B^{\lambda+\mu} \xrightarrow{\eqref{eqn:morph-bBunB-BunT}} \Bun_T^{\lambda+\mu}
\]
is $p_2$, while the composition
\[
\Zas^\mu_\lambda \to \bBun_B^{\lambda+\mu} \times_{\Bun_G} \Bun_{B^-}^\lambda \to \Bun_{B^-}^\lambda \to \Bun_T^{\lambda}
\]
is $p_1$.
The map~\eqref{eqn:open-immersion-Zastava} therefore induces open immersions
\[
\uZas^\mu \hookrightarrow \bBun_B^{\mu} \times_{\Bun_G} \Bun_{U^-}
\quad
\text{and}
\quad
\uZasG^\mu \hookrightarrow \bBun_U \times_{\Bun_G} \Bun_{B^-}^{-\mu}.
\]
(Here we have identified $\Bun_{U^-}$ with $\Bun_{B^-} \times_{\Bun_T} \{\cF^0\}$; this is similar to~\eqref{eqn:BunU-fiber-product}.)

The composition
\[
 \uZasG^\mu \hookrightarrow \bBun_U \times_{\Bun_G} \Bun_{B^-}^{-\mu} \to \bBun_U
\]
(where the second map is the projection on the first component) will be denoted $r_\mu$.

\begin{rmk}
\label{rmk:smoothness-map-uZasmu}
Recall from Lemma~\ref{lem:smoothness-B-} and Remark~\ref{rmk:Bunr} that the natural morphism $\Bun_{B^-}^\lambda \to \Bun_G$ is smooth if $\lambda$ satisfies $\langle \lambda, \alpha \rangle < 2-2g$ for any positive root $\alpha$. If this condition is satisfied then the open immersion~\eqref{eqn:open-immersion-Zastava} shows that the morphism $\Zas^\mu_\lambda \to \bBun_B^{\lambda+\mu}$ is smooth. In particular, if $\mu$ satisfies $\langle \mu, \alpha \rangle > 2g-2$ then the morphism $r_\mu$ is smooth.
\end{rmk}

\subsection{Decompositions of Zastava stacks}
\label{ss:decomp-Zas-stacks}

For $\lambda,\mu$ as above and $\nu \in \bY_{\succeq 0}$, by~\cite[Lemma~3.4]{bfgm} the intersection of
\[
 ({}_\nu \bBun_B^{\lambda+\mu}) \times_{\Bun_G} \Bun_{B^-}^\lambda \subset \bBun_B^{\lambda+\mu} \times_{\Bun_G} \Bun_{B^-}^\lambda
\]
with (the image of) $\Zas^\mu_\lambda$ is empty unless $\nu \preceq \mu$. If this condition is satisfied, this intersection will be denoted ${}_\nu \Zas^\mu_\lambda$; we therefore have a decomposition
\begin{equation}
\label{eqn:strat-Zas}
 \Zas^\mu_\lambda = \bigsqcup_{0 \preceq \nu \preceq \mu} {}_\nu \Zas^\mu_\lambda.
\end{equation}

The following statement is part of~\cite[Lemma~3.4]{bfgm}.

\begin{lem}
\label{lem:uZas-smallest-strata}
In case $\nu=\mu$, we have
${}_\mu \Zas^\mu_\lambda = s^\mu_\lambda (C^\mu \times \Bun_T^\lambda)$.
\end{lem}

The idea of proof of this lemma is as follows.
If a point $(E,\cF,\cE,\beta)$ belongs to ${}_\mu \Zas^\mu_\lambda$, then the corresponding morphisms $(\cF (-E))_{\F_T(\xi)} \to \cE_{\mathsf{M}(\xi)}$ factor through morphisms of vector bundles $\cF_{\F_T(\xi)} \to \cE_{\mathsf{M}(\xi)}$, which define a reduction of $\cE$ to $B$. Since the compositions
\[
\cF_{\F_T(\xi)} \to \cE_{\mathsf{M}(\xi)} \to \cF_{\F_T(\xi)}
\]
are nonzero on each geometric point of $S$ they must be isomorphisms, hence the given reductions of $\cE$ to $B$ and $B^-$ are transversal; they therefore define a reduction to $T$.

For any $\nu$ with $0 \preceq \nu \preceq \mu$, by definition of ${}_\nu \bBun_B^{\lambda+\mu}$ we have a canonical morphism ${}_\nu \Zas^\mu_\lambda \to C^\nu$;
if $\Gamma \in \Part(\nu)$, one defines the locally closed substack ${}_\Gamma \Zas^\mu_\lambda \subset {}_\nu \Zas^\mu_\lambda$ as the preimage of $C^\nu_\Gamma$. We then have a decomposition
\begin{equation}
\label{eqn:strat-Zas-Part}
 \Zas^\mu_\lambda = \bigsqcup_{0 \preceq \nu \preceq \mu} \bigsqcup_{\Gamma \in \Part(\nu)} {}_\Gamma \Zas^\mu_\lambda.
\end{equation}
In case $\nu=\mu$, by Lemma~\ref{lem:uZas-smallest-strata}, for any $\Gamma \in \Part(\mu)$ we have
\[
 {}_\Gamma \Zas^\mu_\lambda = s^\mu_\lambda (C^\mu_\Gamma \times \Bun_T^\lambda).
\]

If a pair $((\cE,\cF,\kappa), (\cE,\cF',\tau))$ belongs to ${}_\nu \Zas^\mu_\lambda$, then there exists a point $E$ in $C^\nu$ and a pair $(\cF'',\kappa')$ such that each $\kappa_\xi' : (\cF'')_{\F_T(\xi)} \to \cE_{\mathsf{M}(\xi)}$ is a morphism of vector bundles, $\cF=\cF''(-E)$,
and for any $\xi \in \bX_+$ the morphism $\kappa_\xi$ is the composition
\[
\cF_{\F_T(\xi)} = (\cF''(-E))_{\F_T(\xi)} \xrightarrow{\eqref{eqn:divisor-sheaf}} (\cF'')_{\F_T(\xi)} \xrightarrow{\kappa_\xi'} \cE_{\mathsf{M}(\xi)}.
\]
The pair
\[
((\cE,\cF'',\kappa'), (\cE,\cF',\tau)) \in \Bun_B^{\lambda+\mu-\nu} \times_{\Bun_G} \Bun_{B^-}^\lambda = {}_0 \bBun_B^{\lambda+\mu-\nu} \times_{\Bun_G} \Bun_{B^-}^\lambda
\]
belongs to the image of $\Zas^{\mu-\nu}_{\lambda}$, and this construction defines an isomorphism of stacks
\begin{equation}
\label{eqn:stratum-Zas}
 {}_\nu \Zas^\mu_\lambda \cong C^\nu \times {}_0 \Zas^{\mu-\nu}_{\lambda}
\end{equation}
which, for any $\Gamma \in \Part(\nu)$, restricts to an isomorphism
\[
{}_\Gamma \Zas^\mu_\lambda \cong C^\nu_\Gamma \times {}_0\Zas^{\mu-\nu}_\lambda.
\]
Under the identification~\eqref{eqn:stratum-Zas}, the morphism~\eqref{eqn:morph-Zas-BunT} on the left-hand side corresponds to the composition
\[
C^\nu \times {}_0 \Zas^{\mu-\nu}_{\lambda} \to C^\nu \times C^{\mu-\nu} \times \Bun_T \to C^\mu \times \Bun_T
\]
where the first morphism is induced by the similar morphism on the factor ${}_0 \Zas^{\mu-\nu}_{\lambda}$, and the second one is induced by the sum map $C^\nu \times C^{\mu-\nu} \to C^\mu$.

Recall the factorization isomorphism~\eqref{eqn:fact-Zas-lambda}. For any $\mu_1,\mu_2,\lambda_1,\lambda_2$ and any $\nu_1,\nu_2$ such that $0 \preceq \nu_1 \preceq \mu_1$ and $0 \preceq \nu_2 \preceq \mu_2$, this isomorphism restricts to an isomorphism
\begin{multline}
\label{eqn:fact-Zas-lambda-nu}
({}_{\nu_1} \Zas^{\mu_1}_{\lambda_1} \times {}_{\nu_2} \Zas^{\mu_2}_{\lambda_2}) \times_{C^{\mu_1} \times C^{\mu_2} \times \Bun_T \times \Bun_T} \mathfrak{X}^{\mu_1,\mu_2} \\
\cong {}_{\nu_1+\nu_2} \Zas^{\mu_1 + \mu_2}_{\lambda_1+\lambda_2} \times_{C^{\mu_1 + \mu_2} \times \Bun_T} \mathfrak{X}^{\mu_1, \mu_2}.
\end{multline}
Given furthermore $\Gamma_1 \in \Part(\nu_1)$ and $\Gamma_2 \in \Part(\nu_2)$, we obtain an isomorphism
\begin{multline}
\label{eqn:fact-Zas-lambda-Gamma}
({}_{\Gamma_1} \Zas^{\mu_1}_{\lambda_1} \times {}_{\Gamma_2} \Zas^{\mu_2}_{\lambda_2}) \times_{C^{\mu_1} \times C^{\mu_2} \times \Bun_T \times \Bun_T} \mathfrak{X}^{\mu_1,\mu_2} \\
\cong {}_{\Gamma_1 \cup \Gamma_2} \Zas^{\mu_1 + \mu_2}_{\lambda_1+\lambda_2} \times_{C^{\mu_1 + \mu_2} \times \Bun_T} \mathfrak{X}^{\mu_1, \mu_2},
\end{multline}
where $\Gamma_1 \cup \Gamma_2$ is the partition of $\nu_1 + \nu_2$ obtained by concatenating $\Gamma_1$ and $\Gamma_2$.

\subsection{Consequences for Zastava stacks}
\label{ss:decomp-Zas}

The decomposition~\eqref{eqn:strat-Zas}
induces similar decompositions for the schemes $\uZas^\mu$ and $\uZasG^\mu$; namely, if we set
\[
{}_\nu \uZas^\mu = \{\cF^0\} \times_{\Bun_T^0, p_1} {}_\nu \Zas^\mu_0, \quad
{}_\nu \uZasG^\mu = 
\{\cF^0\} \times_{\Bun_T^0, p_2}
{}_\nu \Zas^\mu_{-\mu}
\]
then we have
\begin{equation}
\label{eqn:decomp-uZas}
\uZas^\mu = \bigsqcup_{0 \preceq \nu \preceq \mu} {}_\nu \uZas^\mu, \qquad
\uZasG^\mu = \bigsqcup_{0 \preceq \nu \preceq \mu} {}_\nu \uZasG^\mu.
\end{equation}
Similarly, using the decomposition~\eqref{eqn:strat-Zas-Part} and setting
\[
{}_\Gamma \uZas^\mu = \{\cF^0\} \times_{\Bun_T^0, p_1} {}_\Gamma \Zas^\mu_0, \quad
{}_\Gamma \uZasG^\mu = 
\{\cF^0\} \times_{\Bun_T^0, p_2}
{}_\Gamma \Zas^\mu_{-\mu}
\]
then we have
\begin{equation}
\label{eqn:decomp-uZas-Part}
\uZas^\mu = \bigsqcup_{0 \preceq \nu \preceq \mu} \bigsqcup_{\Gamma \in \Part(\nu)} {}_\Gamma \uZas^\mu, \qquad
\uZasG^\mu = \bigsqcup_{0 \preceq \nu \preceq \mu} \bigsqcup_{\Gamma \in \Part(\nu)} {}_\Gamma \uZasG^\mu.
\end{equation}

\begin{rmk}
\phantomsection
\label{rmk:local-isom-strata}
\begin{enumerate}
\item
\label{it:local-isom-strata}
Recall the setting of Proposition~\ref{prop:local-isom}, and assume that the map $S \to \Bun_T$ factors through a map $S \to \Bun_T^\lambda$. Then we have
\[
S \times_{\Bun_T, p_1} \Zas^\mu = S \times_{\Bun_T^\lambda, p_1} \Zas^\mu_\lambda = \bigsqcup_{0 \preceq \nu \preceq \mu} S \times_{\Bun_T^\lambda, p_1} {}_\nu \Zas^\mu_\lambda.
\]
By~\eqref{eqn:decomp-uZas} we also have
\[
S \times \uZas^\mu = \bigsqcup_{0 \preceq \nu \preceq \mu} S \times {}_\nu \uZas^\mu.
\]
It is clear from the proof of that proposition that the local isomorphisms between $S \times_{\Bun_T, p_1} \Zas^\mu$ and $S \times \uZas^\mu$ constructed in that proof identify the intersection with $S \times_{\Bun_T^\lambda, p_1} {}_\nu \Zas^\mu_\lambda$ with the intersection with $S \times {}_\nu \uZas^\mu$, for any $\nu$ such that $0 \preceq \nu \preceq \mu$. A similar claim holds for the decompositions~\eqref{eqn:strat-Zas-Part} and~\eqref{eqn:decomp-uZas-Part}.
\item
The stratum ${}_\nu \uZas^\mu$ is denoted ${}_\nu Z^\mu$ in~\cite[\S 3.5]{bfgm}. (In case $\nu=0$, the notation $Z^\mu_{\max}$ is also used for ${}_0 Z^\mu$.) 
In~\cite[\S 3.6.1]{gaitsgory}, the open subscheme ${}_0 \uZasG^\mu \subset \uZasG^\mu$ is denoted $\mathring{\mathscr{Z}}^{-\mu}$.
\end{enumerate}
\end{rmk}


In~\eqref{eqn:stratum-Zas},
choosing $\lambda=0$ and taking the fiber product with $\{\cF^0\}$ over $\Bun_T$ (for $p_1$), we deduce an isomorphism
\[
{}_\nu \uZas^\mu \cong C^\nu \times {}_0\uZas^{\mu-\nu},
\]
see~\cite[\S 3.5]{bfgm}.
For any $\Gamma \in \Part(\nu)$, this isomorphism restricts to an isomorphism
\begin{equation}
\label{eqn:strata-uZas-isom}
{}_\Gamma \uZas^\mu \cong C^\nu_\Gamma \times {}_0\uZas^{\mu-\nu}.
\end{equation}
(These statements do not have counterparts for the scheme $\uZasG^\mu$.)

\begin{rmk}
Regarding dimensions, for any $\mu \in \bY_{\succeq 0}$ we have
\[
\dim({}_0 \uZas^\mu)=\langle \mu, 2\rho \rangle,
\]
see e.g.~\cite[\S 5.10]{bfgm}. Using the considerations in~\S\ref{ss:graded-sym-powers}, we deduce that if $0 \preceq \nu \preceq \mu$ we have
\[
\dim({}_\nu \uZas^\mu)= \langle \nu,\rho \rangle + \langle \mu-\nu, 2\rho \rangle,
\]
and if $\Gamma \in \Part(\nu)$ we have
\begin{equation}
\label{eqn:dim-strata-uZas}
\dim({}_\Gamma \uZas^\mu)= | \Gamma | + \langle \mu-\nu, 2\rho \rangle.
\end{equation}
\end{rmk}

In the ``factorization'' isomorphisms~\eqref{eqn:fact-Zas-lambda-nu}--\eqref{eqn:fact-Zas-lambda-Gamma}, setting $\lambda_1=\lambda_2=0$ and taking fiber products with $\{\cF^0\}$ we obtain isomorphisms
\begin{multline}
\label{eqn:fact-uZas-nu}
({}_{\nu_1} \uZas^{\mu_1} \times {}_{\nu_2} \uZas^{\mu_2}) \times_{C^{\mu_1} \times C^{\mu_2}} (C^{\mu_1} \times C^{\mu_2})_{\disj} \cong \\
 {}_{\nu_1 + \nu_2} \uZas^{\mu_1 + \mu_2} \times_{C^{\mu_1 + \mu_2}} (C^{\mu_1} \times C^{\mu_2})_{\disj}
\end{multline}
and
\begin{multline}
\label{eqn:fact-uZas-Delta}
({}_{\Gamma_1} \uZas^{\mu_1} \times {}_{\Gamma_2} \uZas^{\mu_2}) \times_{C^{\mu_1} \times C^{\mu_2}} (C^{\mu_1} \times C^{\mu_2})_{\disj} \cong \\
 {}_{\Gamma_1 \cup \Gamma_2} \uZas^{\mu_1 + \mu_2} \times_{C^{\mu_1 + \mu_2}} (C^{\mu_1} \times C^{\mu_2})_{\disj}.
\end{multline}
Similar isomorphisms hold for the schemes $\uZasG^\mu$.

\subsection{Intersection cohomology complex}
\label{ss:IC-Zastavas}

Consider the canonical map
\[
r_\mu : \uZasG^\mu \to \bBun_U,
\]
see~\S\ref{ss:relation-Zastava-Drinfeld}. Recall the intersection cohomology complex $\IC_{\bBun_U}$, see~\S\ref{ss:complexes-sheaves}, and consider similarly the intersection cohomology complex $\IC_{\uZasG^\mu}$ on $\uZasG^\mu$ associated with the constant local system on the smooth open subscheme ${}_0 \uZasG^\mu$. 
The following statement is~\cite[Proposition~3.6.5(a)]{gaitsgory}.

\begin{prop}
\label{prop:pullback-IC-bBun-Zas}
For any $\mu \in \bY_{\succeq 0}$ we have a canonical isomorphism
\[
(r_\mu)^* \IC_{\bBun_U}[\langle \mu, 2\rho \rangle - (g-1) \dim(U)] \cong \IC_{\uZasG^\mu}.
\]
\end{prop}

In~\cite{gaitsgory} it is assumed that $\bk$ has characteristic $0$, but the same proof applies in general. (The author also uses $!$-pullback instead of $*$-pullback, but the two versions are equivalent by Grothendieck--Verdier duality.)
Namely, in case $\mu$ satisfies $\langle \mu, \alpha \rangle > 2g-2$ for any positive root $\alpha$, the morphism $r_\mu$ is smooth by Remark~\ref{rmk:smoothness-map-uZasmu}. Since $\bBun_U$ has dimension $(g-1)\dim(U)$ and $\uZasG^\mu$ has dimension $\langle \mu, 2\rho \rangle$, the formula of the proposition follows from a classical property of smooth pullback. The proof in general proceeds by reduction to this case using factorization; see~\cite[\S 3.9]{gaitsgory} for details.


\subsection{Central fiber}
\label{ss:central-fiber}

Recall from~\eqref{eqn:mu-closed-stratum} that the closed stratum in the stratification~\eqref{eqn:stratif-symprod-Y} (corresponding to $\Gamma=\ppl\mu\ppr$) identifies with $C$. We set
\[
\CZas^\mu := C \times_{C^\mu} \Zas^\mu,
\]
where the morphism $C \to C^\mu$ is the embedding of that stratum.
(This is a stack over $C \times \Bun_T$.)
We also set
\[
\uCZas^\mu = C \times_{C^\mu} \uZas^\mu = \{\cF^0\} \times_{\Bun_T} \CZas^\mu = C \times_{C^\mu \times \Bun_T} \Zas^\mu
\]
(where in the right-hand side, the map $C \to C^\mu \times \Bun_T$ is 
given by $E \mapsto (\mu \cdot E, \cF^0)$)
and
\[
\uCZasG^\mu = C \times_{C^\mu} \uZasG^\mu = C \times_{C \times \Bun_T} \CZas^\mu = C \times_{C^\mu \times \Bun_T^{-\mu}} \Zas^\mu
\]
(where in the third term the map $C \to C \times \Bun_T$ is given by $E \mapsto (E, \cF^0(\mu \cdot E))$, and in the right-hand side the map $C \to C^\mu \times \Bun_T^{-\mu}$ is given by $E \mapsto (\mu \cdot E, \cF^0(\mu \cdot E))$).
These are schemes over $C$. The decompositions~\eqref{eqn:decomp-uZas} induce decompositions
\[
\uCZas^\mu = \bigsqcup_{0 \preceq \nu \preceq \mu} {}_\nu \uCZas^\mu \quad \text{where ${}_\nu \uCZas^\mu = C \times_{C^\mu} {}_\nu \uZas^\mu \cong {}_0 \uCZas^{\mu-\nu}$}
\]
and
\[
\uCZasG^\mu = \bigsqcup_{0 \preceq \nu \preceq \mu} {}_\nu \uCZasG^\mu \quad \text{where ${}_\nu \uCZasG^\mu = C \times_{C^\mu} {}_\nu \uZasG^\mu$.}
\]

Given a closed point $c \in C$, we set
\[
\uCZas^{\mu,c} = \{c\} \times_C \uCZas^\mu, \quad \uCZasG^{\mu,c} = \{c\} \times_C \uCZasG^\mu.
\]
If $\nu$ satisfies $0 \preceq \nu \preceq \mu$,
we also set
\[
{}_\nu \uCZas^{\mu,c} = \{c\} \times_{C} {}_\nu \uCZas^\mu, \quad 
{}_\nu \uCZasG^{\mu,c} = \{c\} \times_{C} {}_\nu \uCZasG^\mu.
\]

\begin{rmk}
The scheme $\uCZas^{\mu,c}$ is denoted $\mathbb{S}^\mu$ in~\cite[\S 2.5]{bfgm}, and the open subscheme ${}_0 \uCZas^{\mu,c}$ is denoted ${}_0 \mathbb{S}^\mu$. The scheme $\uCZasG^{\mu,c}$ is denoted $\mathfrak{F}^{-\mu}$ in~\cite[\S 3.6.3]{gaitsgory}, and the open subscheme ${}_0 \uCZasG^{\mu,c}$ is denoted $\mathring{\mathfrak{F}}^{-\mu}$.
\end{rmk}

We continue with our closed point $c \in C$, and
consider the affine Grassmannian $\Gr_{c}$ for the group $G$, defined as in~\S\ref{ss:Gr} but with respect to the completion of the local ring of $C$ at $c$. (For any choice of a local coordinate at $c$ this completion identifies with $\F ( \hspace{-1pt} ( z ) \hspace{-1pt} )$, which provides an identification of $\Gr_c$ with $\Gr$; but the definition can be phrased without any choice of coordinate, which is what we mean by $\Gr_c$.) We have therein the semi-infinite orbits $\rS_{\lambda,c}$ and $\rS^-_{\lambda,c}$ ($\lambda \in \bY$) defined using the action of the loop groups of $U$ and $U^-$ respectively. By~\cite[Proposition~2.6]{bfgm}, there exists a canonical isomorphism
\[
{}_0 \uCZas^{\mu,c} \cong \rS_{\mu,c} \cap \rS^-_{0,c}.
\]
This isomorphism is roughly defined as follows. An $\F$-point in the left-hand side is a triple $(\mu\cdot c, \cE,\beta)$ where $\cE$ is $G$-torsor on $C$ and $\beta$ is a trivialization over $C \smallsetminus \{c\}$, subject to the conditions from~\eqref{eqn:diag-uZas} and~\eqref{eqn:uZas-cond}. 
The maps $\kappa_\xi$ and $\tau_\xi$ from~\eqref{eqn:uZas-cond} define a $B^-$-reduction and a $B$-reduction, respectively, of $\cE$.  That is, we get a $B^-$-torsor $\cG^-$ and a $B$-torsor $\cG$.  These torsors inherit the trivialization $\beta$ away from $c$, so we get points in $\Gr_{B^-}$ and in $\Gr_B$.  By examining the domain of $\kappa_\xi$ and the codomain of $\tau_\xi$, we find that $\cG^-$ must lie in $\rS^-_{0,c} \subset \Gr_{B^-}$, and $\cG$ in $\rS_{\mu,c} \subset \Gr_B$.


From this identification and standard facts about semiinfinite orbits (see e.g.~\cite[Proposition~3.9]{baugau}) we deduce that
\[
\dim(\uCZas^{\mu,c})=\langle \mu, \rho \rangle.
\]

On the other hand, it is explained in~\cite[\S 3.6.3]{gaitsgory} that we have a canonical identification
\begin{equation}
\label{eqn:central-fiber-G-intersection-si-orbits}
\uCZasG^{\mu,c} \cong \overline{\rS_{0,c}} \cap \rS^-_{-\mu,c}
\end{equation}
which restricts to an identification
\begin{equation}
\label{eqn:central-fiber-G-0-intersection-si-orbits}
{}_0 \uCZasG^{\mu,c} \cong \rS_{0,c} \cap \rS^-_{-\mu,c}.
\end{equation}

\subsection{Smoothness}


\begin{prop}
\label{prop:smoothness-uZas}
For any $\mu \in \bY_{\succeq 0}$ the open subscheme
${}_0 \uZas^\mu \subset \uZas^\mu$
is smooth and connected.
\end{prop}

\begin{proof}
The first assertion (regarding smoothness) is~\cite[Corollary~3.8]{bfgm}. We recall the proof for the reader's convenience (and since we have recalled all of its ingredients).
By definition, for any $\lambda \in \bY$ we have an open immersion
\[
 {}_0 \Zas^\mu_\lambda \to \Bun_B^{\lambda+\mu} \times_{\Bun_G} \Bun_{B^-}^\lambda,
\]
see~\eqref{eqn:open-immersion-Zastava}. By
Remark~\ref{rmk:Bunr}, if $\lambda$ satisfies $\langle \lambda+\mu, \alpha \rangle > 2g-2$ for any positive root $\alpha$ then the morphism $\Bun_{B}^{\lambda+\mu} \to \Bun_G$ is smooth; as a consequence, in this case the morphism 
\[
{}_0 \Zas^\mu_\lambda \to \Bun_{B^-}^\lambda
\]
is smooth, hence so is the composition
\[
{}_0 \Zas^\mu_\lambda \to \Bun_{B^-}^\lambda \to \Bun_T^\lambda,
\]
which coincides with $p_1$. Hence, for any $\cF \in \Bun_T^\lambda(\F)$ the scheme
\[
\{\cF\} \times_{\Bun_T, p_1}
{}_0 \Zas^\mu
\]
is smooth. By Proposition~\ref{prop:local-isom} and Remark~\ref{rmk:local-isom-strata}\eqref{it:local-isom-strata}, this scheme is locally isomorphic to ${}_0 \uZas^\mu$, so that the latter scheme is smooth as well.

Connectedness in case $C=\mathbb{A}^1$ is proved in~\cite[Proposition~2.25]{bfg}. (More precisely, this statement claims connectedness of a certain space of maps from $\mathbb{P}^1$ to the flag variety $G/B$, which by~\cite[Proposition~2.21]{bfg} identifies with ${}_0 \uZas^{\mathbb{A}^1,\mu}$.) The general case can be deduced e.g.~as follows. First, the restriction of ${}_0 \uZas^\mu$ to the open stratum of $C^\mu$ (consisting of collections of pairwise distinct points) is a locally trivial fibration with fibers given by products of copies of $\Gm$, see e.g.~\cite[\S 6.4.2]{fm}; it is therefore connected. In particular, if $\mu$ is a simple root this implies the desired claim. One then proceeds by induction (with respect to the order $\preceq$). Given $\mu$ which is not a simple root and assuming the claim for smaller elements, using factorization (see~\eqref{eqn:fact-uZas-nu}) one sees that the restriction of ${}_0 \uZas^\mu$ to the complement of the closed stratum $C^\mu_{[\mu]} \subset C^\mu$ is connected. If ${}_0 \uZas^\mu$ was disconnected, it would therefore contain a connected component supported over $C^\mu_{[\mu]}$. In view of the analysis of central fibers in~\S\ref{ss:central-fiber}, this component would have dimension at most $\langle \rho, \mu\rangle +1 < 2\langle \rho, \mu\rangle$. This would mean that ${}_0 \uZas^\mu$ would possess a point whose image in $C^\mu$ belongs to $C^\mu_{[\mu]} \cong C$ and such that the dimension at this point is $< 2\langle \rho, \mu\rangle$. Now there exists an open neighborhood $C' \subset C$ of the image of that point and an \'etale map $C' \to \mathbb{A}^1$ (see e.g.~the proof of Lemma~\ref{lem:const-closed-stratum} below for details). Using Lemma~\ref{lem:uZas-compare} and Lemma~\ref{lem:sym-mu}\eqref{it:sm-stratum} we deduce that the dimension of ${}_0 \uZas^\mu$ at that point is equal to $2\langle \rho, \mu\rangle$, which provides a contradiction.
\end{proof}

Using~\eqref{eqn:strata-uZas-isom}, this proposition shows that for any $\nu$ such that $0 \preceq \nu \preceq \mu$ and any $\Gamma \in \Part(\nu)$ the stratum ${}_\Gamma \uZas^\mu$ is smooth. By Remark~\ref{rmk:local-isom-uZas-uZasG}
, the same claim holds for each stratum ${}_\Gamma \uZasG^\mu$. In other words, the decompositions in~\eqref{eqn:decomp-uZas-Part} are stratifications.

\section{Constructibility and applications}
\label{sec:constructibility}


\subsection{Constructibility}
\label{ss:constructibility}

Fix a smooth (but possibly nonprojective) curve $C$, and consider the associated Zastava schemes $\uZas^\mu$, see~\S\ref{ss:indep-curve}. Considering the smooth compactification $\overline{C}$ of $C$ (see~\S\ref{ss:indep-curve}) we obtain (for any $\mu \in \bY_{\succeq 0}$) an open embedding of $\uZas^\mu$ into the corresponding Zastava scheme associated with $\overline{C}$. From the stratification~\eqref{eqn:decomp-uZas-Part} of the latter scheme we deduce a similar stratification of $\uZas^\mu$.

Given $\mu,\nu \in \bY_{\succeq 0}$ such that $\nu \preceq \mu$ and $\Gamma \in \Part(\nu)$, we consider the following condition:
\[
 (\star)^\mu_{\Gamma} \quad \text{for any $n \in \Z$, the sheaf $\mathcal{H}^n((\IC_{\uZas^\mu})_{| {}_\Gamma \uZas^\mu})$ is locally constant.}
\]
We will eventually prove that this condition is always satisfied, but the proof will require an induction argument. When it is satisfied for some triple $(\mu,\nu,\Gamma)$, it makes sense
to consider the polynomial
\[
P^{C,\mu}_{\Gamma}(q)
:= \sum_{n \geq 0} \rank \bigl( \mathcal{H}^{-|\Gamma|+\langle \nu-\mu, 2\rho\rangle-n}((\IC_{\uZas^\mu})_{| {}_\Gamma \uZas^\mu}) \bigr) \cdot q^n \quad \in \Z_{\geq 0}[q].
\]
In fact, when the curve $C$ is clear from the context we will omit it from the notation.  (Eventually, we will see that this polynomial depends neither on $C$, nor on $\mu$.)  In view of~\eqref{eqn:dim-strata-uZas}, the standard support condition for perverse sheaves implies that $P^{C,\mu}_{\Gamma}(q)$ is indeed a polynomial in $q$, which moreover has 0 constant term unless $\nu=0$ (i.e.~$\Gamma=\varnothing$). In the latter case, $(\star)^\mu_{\varnothing}$ is indeed satisfied, and we have $P^{C,\mu}_{\varnothing}=1$. Note also that, by Grothendieck--Verdier duality, when $(\star)^\mu_{\Gamma}$ is satisfied we have
\[
 P^{C,\mu}_{\Gamma}(q) = \sum_{n \geq 0} \rank \bigl( \mathcal{H}^{-|\Gamma|+\langle \nu-\mu, 2\rho\rangle+n}((i^\mu_{\Gamma})^! \IC_{\uZas^\mu}) \bigr) \cdot q^n
\]
where $i^\mu_{\Gamma} : {}_\Gamma \uZas^\mu \to \uZas^\mu$ is the embedding.

The case $\Gamma = \ppl \mu \ppr$ (hence $\nu=\mu$) will be
particularly important below, and for this case we introduce the simplified notation
\[
P^{C,\mu} = P^\mu = P^{C,\mu}_{\ppl\mu\ppr}.
\]

The following statement is clear from Remark~\ref{rmk:local-isom-uZas-uZasG}.

\begin{lem}
\label{lem:const-equiv-Zas}
 For any $\mu,\nu \in \bY_{\succeq 0}$ such that $\nu \preceq \mu$ and $\Gamma \in \Part(\nu)$, the condition $(\star)^\mu_{\Gamma}$ is equivalent to the condition that for any $n \in \Z$, the sheaf $\mathcal{H}^n((\IC_{\uZasG^\mu})_{| {}_\Gamma \uZasG^\mu})$ is locally constant. Moreover, if these conditions are satisfied we have
 \[
  P^{C,\mu}_{\Gamma}(q) = \sum_{n \geq 0} \rank \bigl( \mathcal{H}^{-|\Gamma|+\langle \nu-\mu, 2\rho\rangle-n}((\IC_{\uZasG^\mu})_{| {}_\Gamma \uZasG^\mu}) \bigr) \cdot q^n.
 \]
\end{lem}

When $C$ is projective of genus $g$, we will consider in parallel the following condition, for any $\lambda \in \bY$, $\nu \in \bY_{\succeq 0}$ and $\Gamma \in \Part(\nu)$:
\[
 (\star\star)^\lambda_{\Gamma} \quad \text{for any $n \in \Z$, the sheaf $\mathcal{H}^n((\IC_{\bBun_B^\lambda})_{|{}_\Gamma \bBun^\lambda_B})$ is locally constant.}
\]
When this condition is satisfied it makes sense to consider the polynomial
\[
 Q^{C,\lambda}_{\Gamma}(q)
 := \sum_{n \geq 0} \rank \bigl( \mathcal{H}^{-|\Gamma| + \langle \nu-\lambda, 2\rho \rangle - (g-1) \dim(B)-n}((\IC_{\bBun^\lambda_B})_{|{}_\Gamma \bBun^\lambda_B}) \bigr) \cdot q^n.
\]
Here again the curve $C$ will be omitted from notation when it is clear from the context. As above we have $Q^{C,\lambda}_{\Gamma} \in \Z_{\geq 0}[q]$, and this polynomial has constant term $0$ unless $\nu=0$ (hence $\Gamma=\varnothing$), in which case we have $Q^{C,\lambda}_{\varnothing}=1$. (Eventually we will see that $(\star\star)^\lambda_{\Gamma}$ is always satisfied, and that the polynomials $Q^{C,\lambda}_{\Gamma}(q)$ depend neither on $C$, nor on $\lambda$.)

%

We start with an easy case.

\begin{lem}
\phantomsection
\label{lem:constr-etale}
\begin{enumerate}
\item 
\label{it:cet-pull}
Let $\varphi: C \to D$ be an \'etale map of curves. If condition $(\star)^\mu_{\ppl \mu \ppr}$ holds on $D$, then it holds on $C$, and we have $P^{C,\mu}(q) = P^{D,\mu}(q)$.
\item 
\label{it:cet-cover}
Let $D$ be a curve, and let $(\varphi_i: C_i \to D)_{i \in I}$ be an \'etale covering of $D$.  If $(\star)^\mu_{\ppl \mu \ppr}$ holds for each $C_i$, then it holds for $D$, and we have $P^{C_i,\mu} = P^{D,\mu}$ for all $i$.
\end{enumerate}
\end{lem}

\begin{proof}
\eqref{it:cet-pull}~Recall the diagram~\eqref{eqn:uZas-compare} associated with the map $\varphi: C \to D$.  That diagram shows that
\begin{equation}\label{eqn:uZas-IC-compare}
(\IC_{\uZas^{C,\mu}})_{|C^\mu_\et \times_{C^\mu} \uZas^{C,\mu}} \cong (\varphi_{\mathrm{Zas}})^* \IC_{\uZas^{D,\mu}}.
\end{equation}
Now, recall that $C^{\mu}_{\ppl \mu \ppr} \subset C^\mu_{\et}$, see Lemma~\ref{lem:sym-mu}\eqref{it:sm-stratum}.
Taking the base change of~\eqref{eqn:uZas-compare} along the map $D \cong D^\mu_{\ppl \mu \ppr} \hookrightarrow D^\mu$ and then restricting to $C^\mu_{\ppl \mu \ppr}$, we obtain the following diagram in which every square is a pullback:
\begin{equation*}
\begin{tikzcd}[row sep=tiny, column sep=tiny]
{}_{\ppl \mu \ppr} \uZas^{C,\mu} \ar[dr] \ar[dd, "\wr"']
  && {}_{\ppl \mu \ppr} \uZas^{C,\mu} \ar[dr] \ar[rr, "\varphi'_{\mathrm{Zas}}"] \ar[dd, "\wr"' near start] \ar[ll, equal]
  && {}_{\ppl \mu \ppr} \uZas^{D,\mu} \ar[dr] \ar[dd, "\wr" near start] \\
& \uZas^{C,\mu} 
  && C^\mu_\et \times_{C^\mu} \uZas^{C,\mu} \ar[rr,crossing over, "\varphi_{\mathrm{Zas}}" near start] \ar[ll, crossing over]
  && \uZas^{D,\mu} \ar[dd] \\
C = C^\mu_{\ppl \mu \ppr} \ar[dr] 
  && C = C^\mu_{\ppl \mu \ppr} \ar[dr] \ar[rr, "\varphi" near start] \ar[ll, equal]
  && D = D^\mu_{\ppl \mu \ppr} \ar[dr] \\
& C^\mu  \ar[uu, leftarrow, crossing over] && C^\mu_\et \ar[rr] \ar[ll] \ar[uu, leftarrow, crossing over] && D^\mu.
\end{tikzcd}
\end{equation*}
Here, the vertical maps ${}_{\ppl\mu\ppr}\uZas^{C,\mu} \to C$ and ${}_{\ppl\mu\ppr}\uZas^{D,\mu} \to D$ are isomorphisms as a consequence of Lemma~\ref{lem:uZas-smallest-strata}.  In view of~\eqref{eqn:uZas-IC-compare}, we have
\begin{equation}\label{eqn:uZas-IC-smallest}
(\IC_{\uZas^{C,\mu}})_{|{}_{\ppl \mu\ppr}\uZas^{C,\mu}} \cong (\varphi_{\mathrm{Zas}}')^* \left( (\IC_{\uZas^{D,\mu}})_{|{}_{\ppl \mu\ppr}\uZas^{D,\mu}} \right).
\end{equation}
The desired assertions follow.

\eqref{it:cet-cover}~For each $\varphi_i: C_i \to D$, we can repeat the construction from part~\eqref{it:cet-pull}.  We obtain a collection of maps
\[
\Big(\varphi_{i,\mathrm{Zas}}': {}_{\ppl\mu\ppr}\uZas^{C_i,\mu} \to {}_{\ppl\mu\ppr}\uZas^{D,\mu}\Big)_{i \in I}
\]
that constitutes an \'etale covering of ${}_{\ppl\mu\ppr}\uZas^{D,\mu}$.  For each $i$, we have an instance of~\eqref{eqn:uZas-IC-smallest}:
\[
(\IC_{\uZas^{C_i,\mu}})_{|{}_{\ppl \mu\ppr}\uZas^{C_i,\mu}} \cong (\varphi_{i,\mathrm{Zas}}')^* \left( (\IC_{\uZas^{D,\mu}})_{|{}_{\ppl \mu\ppr}\uZas^{D,\mu}} \right).
\]
If the left-hand side has locally constant cohomology sheaves for all $i \in I$, then $(\IC_{\uZas^{D,\mu}})_{|{}_{\ppl \mu\ppr}\uZas^{D,\mu}}$ has locally constant cohomology sheaves.  The equality $P^{C_i,\mu} = P^{D,\mu}$ follows as in part~\eqref{it:cet-pull}.
\end{proof}

\begin{lem}
\label{lem:const-closed-stratum}
For any $\mu \in \bY_{\succ 0}$, the condition $(\star)^\mu_{\ppl \mu \ppr}$ is satisfied.  Moreover, the polynomial $P^{C,\mu}$ is independent of the curve $C$.
\end{lem}

\begin{proof}
Consider first the special case where $C = \mathbb{A}^1$.  Since $\IC_{\uZas^{C,\mu}}$ is abstractly constructible, there exists a Zariski-open subset 
\[
C_1 \subset {}_{\ppl\mu\ppr}\uZas^{C,\mu}
\]
such that
\begin{equation}\label{eqn:consmall-a1}
\text{$(\IC_{\uZas^{C,\mu}})_{|C_1}$ has locally constant cohomology sheaves.}
\end{equation}
Recall from Lemma~\ref{lem:uZas-smallest-strata} that we may identify
\[
{}_{\ppl\mu\ppr}\uZas^{C,\mu} \cong C = \mathbb{A}^1.
\]
Via this identification, we regard $C_1$ as an open subscheme of $C$, and thus a curve in its own right.  Consider the open immersion $\varphi: C_1 \hookrightarrow C$. By Lemma~\ref{lem:sym-mu}\eqref{it:sm-open} and~\eqref{eqn:uZas-IC-smallest}, we may identify
\[
(\IC_{\uZas^{C_1,\mu}})_{|{}_{\ppl \mu\ppr}\uZas^{C_1,\mu}} \cong  \big((\IC_{\uZas^{C,\mu}})_{|{}_{\ppl \mu\ppr}\uZas^{C,\mu}}\big)_{|C_1}.
\]
Now~\eqref{eqn:consmall-a1} says that condition $(\star)^\mu_{\ppl \mu \ppr}$ holds for $C_1$.

Since $C = \mathbb{A}^1$ can be covered by copies of $C_1$ (after taking suitable translates), Lemma~\ref{lem:constr-etale}\eqref{it:cet-cover} tells us that condition $(\star)^\mu_{\ppl \mu \ppr}$ holds for $\mathbb{A}^1$.

We now prove the proposition for an arbitrary curve $C$.  Let $x$ be a closed point of $C$.  Let $t_x$ be a local parameter at $x$, i.e., a generator for the maximal ideal of the local ring $\scO_{C,x}$.  Then there is an affine neighborhood $C_x$ of $x$ on which $t_x$ defines a morphism $C_x \to \mathbb{A}^1$, and moreover this morphism is \'etale at $x \in C_x$.  By shrinking $C_x$, we may assume that $t_x: C_x \to \mathbb{A}^1$ is an \'etale morphism.  By Lemma~\ref{lem:constr-etale}\eqref{it:cet-pull} and the special case above, we see that condition $(\star)^\mu_{\ppl \mu \ppr}$ holds for $C_x$, and that $P^{C_x,\mu} = P^{\mathbb{A}^1,\mu}$.  Since $C$ can be covered by such open sets $C_x$, Lemma~\ref{lem:constr-etale}\eqref{it:cet-cover} tells us that condition $(\star)^\mu_{\ppl \mu \ppr}$ holds for $C$, and that $P^{C,\mu} = P^{\mathbb{A}^1,\mu}$.
\end{proof}

%

Using factorization, from Lemma~\ref{lem:const-closed-stratum} we deduce a similar claim in case $\nu=\mu$, as follows.

\begin{lem}
 \label{lem:const-partitions-mu}
 Let $\mu \in \bY_{\succeq 0}$. For any $\Gamma \in \Part(\mu)$, the condition $(\star)^\mu_{\Gamma}$ is satisfied. Moreover, the polynomial $P^{C,\mu}_{\Gamma}$ is independent of the curve $C$, and if $\Gamma = \ppl \mu_1, \dots, \mu_r \ppr$ we have
 \[
  P^{C,\mu}_{\Gamma} = \prod_{i=1}^r P^{\mu_i}.
 \]
\end{lem}

\begin{proof}
 If $r=1$ then the claim has already been proved in Lemma~\ref{lem:const-closed-stratum}. Otherwise, write $\Gamma = \ppl \mu_1, \dots, \mu_r \ppr$ as in the statement. We consider the open subscheme
 \[
  (C^{\mu_1} \times \cdots \times C^{\mu_r})_{\mathrm{disj}} \subset C^{\mu_1} \times \cdots \times C^{\mu_r}
 \]
consisting of collections of colored effective divisors with disjoint supports, and the natural \'etale sum morphism
\[
 (C^{\mu_1} \times \cdots \times C^{\mu_r})_{\mathrm{disj}} \to C^\mu.
\]
The induced (\'etale) morphism
\[
 (C^{\mu_1} \times \cdots \times C^{\mu_r})_{\mathrm{disj}} \times_{C^\mu} {}_\Gamma \uZas^{\mu} \to {}_\Gamma \uZas^{\mu}
\]
is surjective.

On the other hand, by factorization (see~\eqref{eqn:fact-uZas}) we have an isomorphism
\[
 (C^{\mu_1} \times \cdots \times C^{\mu_r})_{\mathrm{disj}} \times_{C^\mu} \uZas^{\mu} \cong (C^{\mu_1} \times \cdots \times C^{\mu_r})_{\mathrm{disj}} \times_{C^\mu} (\uZas^{\mu_1} \times \cdots \times \uZas^{\mu_r}),
\]
which identifies the stratum $(C^{\mu_1} \times \cdots \times C^{\mu_r})_{\mathrm{disj}} \times_{C^\mu} {}_\Gamma \uZas^{\mu}$ with 
\[
 (C^{\mu_1} \times \cdots \times C^{\mu_r})_{\mathrm{disj}} \times_{C^\mu} ({}_{\ppl \mu_1 \ppr} \uZas^{\mu_1} \times \cdots \times {}_{\ppl \mu_r \ppr} \uZas^{\mu_r}).
\]
We deduce that $(\star)^\mu_{\Gamma}$ holds for $C$, and that the polynomial $P^{C,\mu}_{\Gamma}$ is equal to the product $\prod_{i=1}^r P^{\mu_i}$. A fortiori, 
this polynomial
is independent of $C$.
\end{proof}

We next relate the singularities of Zastava schemes along the strata considered above with those of Drinfeld's compactifications, paraphrasing~\cite[Item~(1)]{bfgm-err}.

\begin{lem}
\label{lem:local-model}
Assume $C$ is a projective curve.
  Let $\lambda \in \bY$, $\nu \in \bY_{\succeq 0}$, and $\Gamma \in \Part(\nu)$. For any $\F$-point $x$ of ${}_\Gamma \bBun^\lambda_B$, there exists a separated $\F$-scheme of finite type $Y$, $\F$-points $y$ of $Y$ and $z$ of ${}_\Gamma \uZas^\nu$, and smooth morphisms
  \[
   \bBun^\lambda_B \leftarrow Y \to \uZas^\nu
  \]
sending $y$ to $x$ and $z$ respectively.
\end{lem}

\begin{proof}
 Set $x=(\cF,\cE,\kappa)$.
As explained in~\cite[Item~(1)]{bfgm-err}, there exists a $B^-$-reduction of $\cE$ which is transversal at any point in the support of the image of $x$ in $C_\Gamma^\nu$ to the $B$-reduction of $\cE$ given by the projection
\[
{}_\Gamma \bBun^\lambda_B \overset{\eqref{eqn:strata-bBun-isom-lambda}}{\cong} C^\nu_\Gamma \times \Bun^{\lambda-\nu}_B \to \Bun^{\lambda-\nu}_B,
\]
and moreover whose associated $T$-torsor belongs to some component $\Bun_T^\eta \subset \Bun_T^{\mathrm{r}}$.
This $B^-$-reduction defines a preimage of $x$ in
\[
{}_\Gamma \Zas^{\lambda-\eta}_\eta \subset {}_\Gamma \bBun_B^{\lambda} \times_{\Bun_G} \Bun_{B^-}^\eta,
\]
whose image in $C^\nu \times C^{\lambda-\eta-\nu}$ (see~\eqref{eqn:stratum-Zas}) belongs to $(C^\nu \times C^{\lambda-\eta-\nu})_{\mathrm{disj}}$. Here the morphism
\[
\bBun_B^{\lambda} \times_{\Bun_G} \Bun_{B^-}^\eta \to \bBun_B^{\lambda}
\]
is smooth by Lemma~\ref{lem:smoothness-B-}. Choose a smooth scheme and a smooth morphism $S \to \Bun_T^\eta$ whose image contains the image of our point in $\Bun_T^\eta$, and consider the fiber product
\[
S \times_{\Bun_T^\eta} \Zas^{\lambda-\eta}_\eta,
\]
which is a scheme by Proposition~\ref{prop:Zastava-scheme}.
We therefore have a smooth morphism
\begin{equation}
\label{eqn:local-model-morph}
S \times_{\Bun_T^\eta} \Zas^{\lambda-\eta}_\eta \to \bBun_B^{\lambda}
\end{equation}
and a point in $S \times_{\Bun_T^\eta} {}_\Gamma \Zas^{\lambda-\eta}_\eta$ whose image in $\bBun_B^{\lambda}$ is $x$, and whose image in $C^\nu \times C^{\lambda-\eta-\nu}$ belongs to $(C^\nu \times C^{\lambda-\eta-\nu})_{\mathrm{disj}}$. By Proposition~\ref{prop:local-isom} the left-hand side of~\eqref{eqn:local-model-morph} identifies in a neighborhood of our point with
\[
S \times \uZas^{\lambda-\eta},
\]
in such a way that the stratum $S \times_{\Bun_T^\eta} {}_\Gamma \Zas^{\lambda-\eta}_\eta$ corresponds to $S \times {}_\Gamma \uZas^{\lambda-\eta}$, see Remark~\ref{rmk:local-isom-strata}.

On the other hand we have the factorization isomorphism
\[
\uZas^{\lambda-\eta} \times_{C^{\lambda-\eta}} (C^\nu \times C^{\lambda-\eta-\nu})_{\mathrm{disj}} \cong (\uZas^\nu \times \uZas^{\lambda-\eta-\nu}) \times_{C^\nu \times C^{\lambda-\eta-\nu}} (C^\nu \times C^{\lambda-\eta-\nu})_{\mathrm{disj}},
\]
see~\eqref{eqn:fact-uZas}.
We have a lift of the image of our point in $\uZas^{\lambda-\eta}$ to the left-hand side, whose corresponding point in the right-hand side belongs to
\[
({}_\Gamma \uZas^\nu \times {}_0 \uZas^{\lambda-\eta-\nu}) \times_{C^\nu \times C^{\lambda-\eta-\nu}} (C^\nu \times C^{\lambda-\eta-\nu})_{\mathrm{disj}}.
\]
Here ${}_0 \uZas^{\lambda-\eta-\nu}$ is smooth (see Proposition~\ref{prop:smoothness-uZas}), so that there exists a smooth morphism
\[
(\uZas^\nu \times {}_0 \uZas^{\lambda-\eta-\nu}) \times_{C^\nu \times C^{\lambda-\eta-\nu}} (C^\nu \times C^{\lambda-\eta-\nu})_{\mathrm{disj}} \to \uZas^\nu.
\]
Finally, we have thus constructed a scheme $Y$, smooth morphisms
\[
\bBun_B \leftarrow Y \to \uZas^\nu,
\]
and an $\F$-point $y$ of $Y$ which maps to $x$ in the left-hand side, and to a point in ${}_\Gamma \uZas^\nu$ in the right-hand side, which finishes the proof.
\end{proof}


Finally we can prove that the conditions $(\star)^\mu_{\Gamma}$ and $(\star\star)^\lambda_{\Gamma}$ are always satisfied, and express the associated polynomials in terms of the polynomials $P^\nu$.

\begin{thm}
\phantomsection
\label{thm:const}
\begin{enumerate}
 \item 
 \label{it:const-Bun}
Let $C$ be a smooth projective curve.
  For any $\lambda \in \bY$, $\nu \in \bY_{\succeq 0}$ and $\Gamma \in \Part(\nu)$, the condition $(\star\star)^\lambda_{\Gamma}$ is satisfied. Moreover, the polynomial $Q^{C,\lambda}_{\Gamma}$ depends only on $\Gamma$; more specifically, if $\Gamma = \ppl \nu_1, \ldots, \nu_r \ppr$ we have
 \[
  Q^{C,\lambda}_{\Gamma} = \prod_{i=1}^r P^{\nu_i}.
 \]
 \item
 \label{it:const-Zas}
 Let $C$ be a smooth curve.
  For any $\mu, \nu \in \bY_{\succeq 0}$ such that $\nu \preceq \mu$ and any $\Gamma \in \Part(\nu)$, the condition $(\star)^\mu_{\Gamma}$ is satisfied. Moreover, the polynomial $P^{C,\mu}_{\Gamma}$ depends only on $\Gamma$; more specifically, if $\Gamma = \ppl \nu_1, \ldots, \nu_r \ppr$ we have
 \[
  P^{C,\mu}_{\Gamma} = \prod_{i=1}^r P^{\nu_i}.
 \]
\end{enumerate}
\end{thm}

\begin{proof}
 \eqref{it:const-Bun}
 The claim follows directly from Lemma~\ref{lem:local-model} and Lemma~\ref{lem:const-partitions-mu}.
 
 \eqref{it:const-Zas}
 Replacing if necessary $C$ by its smooth completion (see~\S\ref{ss:indep-curve}) we can assume that $C$ is projective.
 By~\eqref{it:const-Bun} and Lemma~\ref{lem:pullback-IC-BunBU} we know that for any $\nu \in \bY_{\succeq 0}$, $\Gamma \in \Part(\nu)$ and $n \in \Z$ the sheaf $\mathcal{H}^n(\IC_{\bBun_U}{}_{|{}_\Gamma \bBun_U})$ is locally constant, and that moreover we have
 \begin{equation}
 \label{eqn:ranks-costalks-bBunU}
  \sum_{n \geq 0} \rank \bigl( \mathcal{H}^{-|\Gamma| + \langle \nu, 2\rho \rangle - (g-1)\dim(U)-n}((\IC_{\bBun_U})_{|{}_\Gamma \bBun_U}) \bigr) \cdot q^n = \prod_{i=1}^r P^{\nu_i},
 \end{equation}
where $\Gamma = \ppl \nu_1, \ldots, \nu_r \ppr$. The desired claim follows, in view of Proposition~\ref{prop:pullback-IC-bBun-Zas} and Lemma~\ref{lem:const-equiv-Zas}.
\end{proof}

\subsection{Semiinfinite sheaves on the affine Grassmannian}
\label{ss:si-sheaves}

Recall from~\S\ref{ss:complexes-sheaves} that for any prestack locally of finite type $X$ we have an $\infty$-category $\Shv(X)$ of sheaves of $\bk$-vector spaces on $X$. (Here $\bk$ is a field of coefficients, with the constraints explained in~\S\ref{ss:complexes-sheaves}.) In particular, for the ind-scheme $\Gr$ of~\S\ref{ss:Gr} we have the $\infty$-category $\Shv(\Gr)$. Consider now the ``seminfinite Iwahori subgroup''
\[
\siIw := \Loop^+ T \ltimes \Loop U,
\]
a group ind-scheme over $\F$ which acts naturally on $\Gr$. (The orbits for this action coincide with the $\Loop U$-orbits, which we have described in~\S\ref{ss:Gr}.) 
Consider also the ``pro-unipotent radical'' $\siIwu$ of $\siIw$, namely the kernel of the composition of natural morphisms $\Loop^+ T \ltimes \Loop U \to \Loop^+ T \to T$.

In~\cite[\S 2.6]{adr} it is explained how to make sense of the full subcategory
\[
\Shv(\siIwu \backslash \Gr) \subset \Shv(\Gr)
\]
of complexes which are equivariant under the action of $\siIwu$. (This construction copies a definition given by Gaitsgory in~\cite{gaitsgory}, assuming that $\bk$ has characteristic $0$.) We explained (again following Gaitsgory) in~\cite[\S 6.1]{adr} how to define a ``perverse'' t-structure on $\Shv(\siIwu \backslash \Gr)$, and in~\cite[\S 6.1]{adr} how to construct a remarkable object $\Gaits$ in the heart of this t-structure. This object is supported on $\overline{\rS_0}$. It admits several equivalent descriptions; we recall here the one which will be used below. Consider the $\infty$-category $\Shv(\Loop^+ G \backslash \Gr)$ of $\Loop^+ G$-equivariant complexes on $\Gr$, with its natural perverse t-structure. The $\Loop^+G$-orbits on $\Gr$ are naturally labelled by $\bY_+$, and for $\lambda \in \bY_+$ we denote by $\cI_*(\lambda)$ the perverse-$0$ cohomology of the $*$-extension of the perversely shifted constant local system on the orbit labelled by $\lambda$. (These perverse sheaves correspond to induced modules for the Langlands dual group under the geometric Satake equivalence; see~\cite[\S 4.3]{adr} for details and references.) For $\lambda \in \bY$, we also denote by $z^\lambda \cdot (-)$ the functor of pushforward under the left action of $z^\lambda$ on $\Gr$. Then in~\cite[\S 8.3]{adr} we explain that there exists a canonical functor from (the $\infty$-category associated with) the directed set $\bY_+$ (endowed with the preorder defined by $\lambda \unlhd \mu$ iff $\mu - \lambda \in \bY_+$) to $\Shv(\Loop^+ G \backslash \Gr)$, sending $\lambda$ to $z^{-\lambda} \cdot \cI_*(\lambda) [\langle \lambda, 2\rho \rangle]$, and that
\begin{equation}
\label{eqn:Gaits-colim}
\Gaits \cong \colim_{\lambda \in \bY_+} z^{-\lambda} \cdot \cI_*(\lambda) [\langle \lambda, 2\rho \rangle].
\end{equation}

\begin{rmk}
As for any $\infty$-category of sheaves on an ind-scheme of ind-finite type, the $\infty$-category $\Shv(\Gr)$ carries a natural perverse t-structure associated with the middle perversity. The perverse t-structure on $\Shv(\siIwu \backslash \Gr)$ is however \emph{not} a restriction of this t-structure. This is related to the fact that the $\siIwu$-orbits on $\Gr$ are infinite-dimensional, hence do not interact well with the usual game of perverse sheaves.
\end{rmk}

For any $\lambda \in \bY$, one can also consider the $\infty$-category $\Shv(\rS_\lambda)$, and the full subcategory $\Shv(\siIwu \backslash \rS_\lambda) \subset \Shv(\rS_\lambda)$ of complexes equivariant under the action of $\siIwu$; here since the action is transitive we have a canonical equivalence
\begin{equation}
\label{eqn:equiv-Shv-orbit-Vect}
\Shv(\siIwu \backslash \rS_\lambda) \cong \Vect_\bk
\end{equation}
induced by taking the $!$-restriction to $[z^\lambda]$, see~\cite[Eqn.~(3.6)]{adr}. (In the right-hand side, $\Vect_\bk$ is the $\infty$-category of complexes of $\bk$-vector spaces.) Using this equivalence one can define a t-structure on $\Shv(\siIwu \backslash \rS_\lambda)$, by taking the cohomological shift of the transfer of the tautological t-structure on $\Vect_\bk$ such that the complex $\omega_{\rS_\lambda}[-\langle \lambda, 2\rho \rangle]$ is perverse. (Here, $\omega_{\rS_\lambda}$ is the dualizing complex on $\rS_\lambda$, obtained by $!$-pullback from the sheaf $\bk$ on a point; this complex corresponds to $\bk$ under the equivalence~\eqref{eqn:equiv-Shv-orbit-Vect}.) The cohomology functors associated with this t-structure will be denoted $(\pH^n : n \in \Z)$. Objects in its heart are direct sums of copies of $\omega_{\rS_\lambda}[-\langle \lambda, 2\rho \rangle]$; in case the multiplicity in an object $\cF$ is finite we call it the rank of $\cF$, and denote it $\rank(\cF)$.

We have canonical functors
\[
(\bi_\lambda)^! : \Shv(\siIwu \backslash \Gr) \to \Shv(\siIwu \backslash \rS_\lambda), \quad 
(\bi_\lambda)^* : \Shv(\siIwu \backslash \Gr) \to \Shv(\siIwu \backslash \rS_\lambda).
\]
Recall also the $q$-analogue of Kostant's partition function $\cP(\lambda,q) \in \Z[q]$, which counts the number of ways to write $\lambda$ as a sum of positive roots; see~\cite[\S 6.7]{adr} for our normalization. The following statement is part of~\cite[Theorem~8.2]{adr}.

\begin{thm}
\phantomsection
\label{thm:adr}
\begin{enumerate}
\item
\label{it:adr-stalks}
For any $\nu \in \bY$ and $n \in \Z$, the complex $\pH^n((\bi_\nu)^* \Gaits)$ vanishes unless $n=0$.
\item
\label{it:adr-costalks}
Assume that $\mathrm{char}(\bk)$ is good for $G$. Then for any $\nu \in \bY$ and $n \in \Z$ the complex $\pH^n((\bi_\nu)^! \Gaits)$ has finite rank, and for any $\nu \in \bY$ we have
\[
\sum_{n \in \Z} \rank \left( \pH^n((\bi_\nu)^! \Gaits) \right) \cdot q^n = \cP(-\nu, q^2).
\] 
\end{enumerate}
\end{thm}


\subsection{Relation between \texorpdfstring{$\IC_{\bBun_U}$}{ICBunU} and the Gaitsgory sheaf}

In this subsection we explain how to compute the polynomials $P^\mu$, which will conclude the description of dimensions of stalks of intersection cohomology complexes on Zastava schemes and Drinfeld's compactifications in view of Theorem~\ref{thm:const}. This procedure will involve the scheme $\bBun_U$, hence here we require $C$ to be projective, and denote its genus by $g$.

Fix a closed point $c \in C$.  
As explained in~\cite[\S 3.1.2]{gaitsgory}, there is a version $(\bBun_U)_{\infty \cdot c}$ of $\bBun_U$ where the maps $\kappa_\xi$ in the moduli description of~\S\ref{ss:def-bBunB} (now with $\cF=\cF^0$, see~\S\ref{ss:BunU}) are allowed to have poles at $c$; here $(\bBun_U)_{\infty \cdot c}$ is an ind-algebraic stack. One can also bound the order of the poles using an element $\lambda \in \bY$: there is a closed substack
\[
 (\bBun_U)_{\leq \lambda \cdot c} \subset (\bBun_U)_{\infty \cdot c}
\]
characterized by the condition that each $\kappa_\xi$ is allowed to have a pole of order${}\le \la \xi,\lambda \ra$ at $c$ for any $\xi$. (By convention, a pole of order $n < 0$ means a zero of order $-n$.) In case $\lambda=0$, we have
\begin{equation}
\label{eqn:bBunU-infty}
 (\bBun_U)_{\leq 0 \cdot c} = \bBun_U.
\end{equation}

Following~\cite[\S 3.1.3]{gaitsgory}, for $\lambda \in \bY$ we also have an open substack
\[
 (\bBun_U)_{= \lambda \cdot c} \subset (\bBun_U)_{\leq \lambda \cdot c}
\]
where the pole at $c$ of each $\kappa_\xi$ should be of order exactly $\langle \xi, \lambda \rangle$,
and moreover these maps should not have zeros on $C \smallsetminus \{c\}$. 
Here $(\bBun_U)_{= \lambda \cdot c}$ 
is a smooth stack, of dimension
\begin{equation*}
 \dim (\bBun_U)_{= \lambda \cdot c} = (g-1) \dim(U) + \langle \lambda, 2\rho \rangle.
\end{equation*}
In case $\lambda=-\nu$ for some $\nu \in \bY_{\succeq 0}$, we have
\begin{equation}
\label{eqn:strata-BunU-Gaits}
 (\bBun_U)_{= (-\nu) \cdot c} = \{c\} \times_C {}_{\ppl \nu \ppr} \bBun_U,
\end{equation}
where we have identified $C^\nu_{\ppl \nu \ppr}$ with $C$ (see~\eqref{eqn:mu-closed-stratum}).


Now recall the ind-scheme $\Gr_c$ considered in~\S\ref{ss:central-fiber}.
As explained in~\cite[\S 3.1.6]{gaitsgory} there is a canonical morphism of stacks
\[
 \pi : \Gr_c \to (\bBun_U)_{\infty \cdot c}.
\]
The definition of $\pi$ uses the moduli description of $\Gr_c$: it classifies pairs $(\cE,\beta)$ where $\cE$ is a $G$-torsor on $C$ and $\beta$ is a trivialization on $C \smallsetminus \{c\}$. The map $\pi$ sends $(\cE,\beta)$ to the pair consisting of $\cE$ and, for $\xi \in \bX_+$, the map $\kappa_\xi$ defined on $C \smallsetminus \{c\}$ by the trivialization of $\cE$ and the morphism~\eqref{eqn:morph-character-Weyl}. The preimage under this map of $(\bBun_U)_{\leq \lambda \cdot c}$, resp.~of $(\bBun_U)_{= \lambda \cdot c}$, is $\overline{\rS_{\lambda,c}}$, resp.~$\rS_{\lambda,c}$.

Consider the functor
\[
 \pi^! : \Shv((\bBun_U)_{\infty \cdot c}) \to \Shv(\Gr_c).
\]
%
%
The rest of this section will be devoted to the proof of the following statement. (Here we view $\IC_{\bBun_U}$ as a complex on $(\bBun_U)_{\infty \cdot c}$ via the identification~\eqref{eqn:bBunU-infty}. Also, we consider $\Gaits$ as a complex of sheaves on $\Gr_c$ using any choice of local parameter at $c$ to identify $\Gr_c$ with $\Gr$; the resulting complex does not depend on the choice of parameter.)

\begin{thm}
\label{thm:Gaits-Bun}
 There exists a canonical isomorphism
 \[
  \pi^! \IC_{\bBun_U}[(g- 1)\dim(U)] \cong \Gaits.
 \]
\end{thm}

\begin{rmk}
\label{rmk:Gaits-omega-GL2}
Recall from Remark~\ref{rmk:bBunB}\eqref{it:bBunB-Laumon} that, in case $G=\mathrm{GL}_2$, the stack $\bBun_B$ is smooth. Using (the Grothendieck--Verdier dual of) Lemma~\ref{lem:pullback-IC-BunBU}, it follows that in this case we have $\IC_{\bBun_U} \cong \omega_{\bBun_U} [-(g- 1)\dim(U)]$, hence by Theorem~\ref{thm:Gaits-Bun} we have $\Gaits \cong \omega_{\overline{\rS_{0,c}}}$. 
\end{rmk}

In case $\bk$ has characteristic $0$, this statement is~\cite[Theorem~3.2.2]{gaitsgory}, and our proof (given in~\S\ref{ss:proof-Drinfeld} below) will follow the same strategy as in~\cite{gaitsgory}. 

\subsection{Applications}

Before giving the proof of Theorem~\ref{thm:Gaits-Bun}, we explain our main application of this result. 

\begin{cor}
\label{cor:formula-P}
Assume that $\mathrm{char}(\bk)$ is good for $G$.
 For any $\mu \in \bY_{\succ 0}$ we have
 \[
 P^\mu = q^{-1} \cdot \cP(\mu,q^2).
 \]
\end{cor}

\begin{proof}
 As explained above, for any $\mu \in \bY_{\succ 0}$ the map $\pi$ sends $\rS_{-\mu,c}$ into the substack $(\bBun_U)_{= (-\mu) \cdot c}$, hence a fortiori into ${}_{\ppl \mu \ppr} \bBun_U$, see~\eqref{eqn:strata-BunU-Gaits}. Denoting 
 by $\mathfrak{i}_{\mu} : {}_{\ppl \mu \ppr} \bBun_U\to (\bBun_U)_{\infty \cdot c}$ the embedding and 
 by $\pi_{-\mu} : \rS_{-\mu,c} \to {}_{\ppl \mu \ppr} \bBun_U$ the restriction of $\pi$, Theorem~\ref{thm:Gaits-Bun} implies that we have
 \[
  (\pi_{-\mu})^! (\mathfrak{i}_{\mu})^! \IC_{\bBun_U}[(g-1)\dim(U)] \cong (\bi_{-\mu})^! \Gaits.
 \]
 By Theorem~\ref{thm:adr}\eqref{it:adr-costalks}, the ranks of the cohomology objects of the right-hand side satisfy
 \[
  \sum_{n \in \Z} \rank \left( \pH^n((\bi_{-\mu})^! \Gaits) \right) \cdot q^n = \cP(\mu, q^2).
 \]
On the other have we have $(\pi_{-\mu})^! \omega_{{}_{\ppl \mu \ppr} \bBun_U} = \omega_{\rS_{-\mu,c}}$, so that for a complex $\cF$ on ${}_{\ppl \mu \ppr} \bBun_U$ with locally constant cohomology sheaves, we have
\[
\rank \left( \pH^{n + \langle \mu, 2\rho \rangle}((\pi_{-\mu})^!\cF) \right) = \rank \left( \mathcal{H}^{n-2 \dim({}_{\ppl -\mu \ppr} \bBun_U)}(\cF) \right),
\]
or, using the dimension formulas from~\S\ref{ss:BunU},
\[
\rank \pH^{n+\langle \mu, 2\rho\rangle}((\pi_{-\mu})^!\cF) = \rank \mathcal{H}^{n-2 +2\langle\mu,2\rho\rangle - 2(g-1)\dim(U)}(\cF).
\]
We now compute:
 \begin{align*}
  \sum_{n \in \Z} \rank& \pH^n((\pi_{-\mu})^! (\mathfrak{i}_{\mu})^! \IC_{\bBun_U}[(g-1)\dim(U)]) \cdot q^n \\
  &= 
  \sum_{n \in \Z} \rank \pH^{n+ (g-1)\dim(U)}((\pi_{-\mu})^! (\mathfrak{i}_{\mu})^! \IC_{\bBun_U}) \cdot q^n \\
  &= 
  \sum_{n \in \Z} \rank \mathcal{H}^{n-(g-1)\dim(U)+ \langle \mu, 2\rho \rangle-2}((\mathfrak{i}_{\mu})^! \IC_{\bBun_U}) \cdot q^n \\
  &= 
  q\sum_{n \in \Z} \rank \mathcal{H}^{n-(g-1)\dim(U)+ \langle \mu, 2\rho \rangle-1}((\mathfrak{i}_{\mu})^! \IC_{\bBun_U}) \cdot q^n.
 \end{align*}
 By~\eqref{eqn:ranks-costalks-bBunU} and Grothendieck--Verdier duality, the last term in these equalities equals $q \cdot P^\mu$, which finishes the proof.
\end{proof}

Note in particular that the formula in Corollary~\ref{cor:formula-P} 
does not depend on $\bk$ (as long as $\mathrm{char}(\bk)$ is good for $G$). As explained above, this implies that the dimensions of stalks of the complexes $\IC_{\bBun_B}$ and $\IC_{\uZas^\mu}$ do not depend on $\bk$ either (under this assumption).

\begin{rmk}
\label{rmk:PID}
We can also consider intersection cohomology with coefficients in a principal ideal domain $\mathbb{O}$.  More precisely, consider one of the following settings:
\begin{enumerate}
\item $\mathbb{O}$ is the ring of integers in a finite extension of $\Q_\ell$, where $\ell$ is a good prime for $G$, and we work with \'etale $\mathbb{O}$-sheaves.
\item $\F = \C$, $\mathbb{O}$ is an arbitrary principal ideal domain in which all bad primes for $G$ are invertible, and we work with $\mathbb{O}$-sheaves for the analytic toplogy.
\end{enumerate}
Denoting that $\IC_{\bBun_B}(\mathbb{O})$ and $\IC_{\uZas^\mu}(\mathbb{O})$ the intersection cohomology complexes with coefficients in $\mathbb{O}$ associated with the constant local system on $\Bun_B$ and ${}_0 \uZas^\mu$ respectively,
Proposition~\ref{prop:IC-indep} implies that these complexes are constructible with respect to the natural stratifications of $\bBun_B$ and $\uZas^\mu$ respectively, that the cohomology of their
stalks and costalks are free over $\mathbb{O}$, and that their ranks are described by the formulas in Theorem~\ref{thm:const}. (Proposition~\ref{prop:IC-indep} only considers the case of schemes; the statement for $\bBun_B$ can be deduced using a smooth morphism from a scheme.)
\end{rmk}

In the following statement we consider the notions of even and odd complexes from~\cite{jmw} (for the constant pariversity).

\begin{cor}
\label{cor:parity}
Assume that $\mathrm{char}(\bk)$ is good for $G$.
\begin{enumerate}
\item For any $\lambda \in \bY$, $\IC_{\bBun_B^\lambda}$ is even if $(g-1)\dim(B)$ is even, and odd if $(g-1)\dim(B)$ is odd.
\item For any $\mu \in \bY_{\succeq 0}$, $\IC_{\uZas^\mu}$ is even.
\end{enumerate}
\end{cor}

\begin{proof}
According to Corollary~\ref{cor:formula-P}, for any $\nu \in \bY_{\succ 0}$, $P^\nu$ contains only odd powers of $q$, so by Theorem~\ref{thm:const}, the polynomials $Q^\lambda_\Gamma$ and $P^\mu_\Gamma$ involve only powers of $q$ that are congruent to $|\Gamma|$ modulo $2$.  Comparing with the definitions of $Q^\lambda_\Gamma$ and $P^\mu_\Gamma$, we deduce that
\begin{align*}
\mathcal{H}^m((\IC_{\bBun^\lambda_B})_{| {}_\Gamma\bBun^\lambda_B}) \ne 0 
&& \text{implies} &&
m &\equiv (g-1)\dim B \pmod 2 \\
\mathcal{H}^m((\IC_{\uZas^\mu})_{| {}_\Gamma\uZas^\mu}) \ne 0 
&& \text{implies} &&
m &\equiv 0 \pmod 2
\end{align*}
These conditions are also satisfied when $\Gamma=\varnothing$.
Since $\IC_{\bBun_B^\lambda}$ and $\IC_{\uZas^\mu}$ are self-dual under Grothendieck--Verdier duality, the same conditions will hold for the corestrictions to the strata, which proves our claims.
\end{proof}

\begin{rmk}
The theory of parity complexes is developed in~\cite{jmw} under certain assumptions on the strata of the given stratification, see~\cite[\S 2.1]{jmw}. We do not claim that (and do not know if) these technical conditions are satisfied in the cases considered in Corollary~\ref{cor:parity}, but only that the complexes $\IC_{\bBun_B^\lambda}$ and $\IC_{\uZas^\mu}$ satisfy the parity vanishing conditions for even or odd complexes from~\cite[Definition~2.4]{jmw}.
\end{rmk}

\subsection{Cohomology of Zastava schemes with support in the central fiber}

The main step towards the proof of Theorem~\ref{thm:Gaits-Bun} will be the following lemma, where we use the notation introduced in~\S\ref{ss:central-fiber}.


\begin{prop}
\label{prop:IC-Zastava}
Let $C$ be any smooth curve with a marked closed point $c$, and let $\mu \in \bY_{\succeq 0}$.
\begin{enumerate}
\item
\label{it:IC-Zastava-1}
For any $n \in \Z$ such that $n \neq 0$ we have
\[
\mathsf{H}^{n}_{\uCZasG^{\mu,c}}(\uZasG^\mu, \IC_{\uZasG^\mu})=0.
\]
\item
\label{it:IC-Zastava-2}
The restriction morphism
\[
\mathsf{H}^{0}_{\uCZasG^{\mu,c}}(\uZasG^\mu, \IC_{\uZasG^\mu}) \to \mathsf{H}^{0}_{{}_0 \uCZasG^{\mu,c}}({}_0 \uZasG^\mu, \IC_{{}_0 \uZasG^\mu}) 
\]
is an isomorphism.
\end{enumerate}
\end{prop}



In case $\bk$ has characteristic $0$, this statement is~\cite[Proposition~3.6.6]{gaitsgory}. This proposition is not proved in~\cite{gaitsgory}; the author instead refers to~\cite{bfgm}. The proof given below is different; using the considerations in~\S\ref{ss:constructibility} we reduce the proof to the case $C=\bbA^1$, where we use 
some constructions from~\cite[\S 2.4--2.5]{ffkm} which exploit specific features of this case (see Remark~\ref{rmk:Zastava-A1}).


\begin{proof}
The same considerations as in the proofs of Lemma~\ref{lem:constr-etale} and Lemma~\ref{lem:const-closed-stratum} allow us to reduce the proof to the case $C=\bbA^1$ and $c=0$. (Note in particular that $\uCZasG^{\mu,c}$ does not depend on the choice of $C$, see~\S\ref{ss:central-fiber}.) We will therefore assume in this proof that we are in this setting.
%
%
Also, Remark~\ref{rmk:local-isom-uZas-uZasG} shows that proving the proposition is equivalent to proving that
\[
\mathsf{H}^{n}_{\uCZas^{\mu,c}}(\uZas^\mu, \IC_{\uZas^\mu})=0
\]
for any $n \neq 0$, and that
the restriction morphism
\[
\mathsf{H}^{0}_{\uCZas^{\mu,c}}(\uZas^\mu, \IC_{\uZas^\mu}) \to \mathsf{H}^{0}_{{}_0 \uCZas^{\mu,c}}({}_0 \uZas^\mu, \IC_{{}_0 \uZas^\mu}) 
\]
is an isomorphism. 
Note that,
by Grothendieck--Verdier duality, for any $n \in \Z$ we have
\begin{equation}
\label{eqn:cohom-compact-Zmu-Fmu}
\mathsf{H}^{-n}_{\uCZas^{\mu,c}}(\uZas^\mu, \IC_{\uZas^\mu})^* \cong \mathsf{H}^{n}_{\mathrm{c}}(\uCZas^{\mu,c}, (\IC_{\uZas^\mu}) {}_{| \uCZas^{\mu,c}}),
\end{equation}
and similarly for $\mathsf{H}^{0}_{{}_0 \uCZas^{\mu,c}}({}_0 \uZas^\mu, \IC_{{}_0 \uZas^\mu})$.

It is clear from definitions that, if $0 \preceq \nu \preceq \mu$ and $\Gamma \in \Part(\nu)$, the stratum ${}_\Gamma \uZas^\mu$ intersects $\uCZas^{\mu,c}$ if and only if $\Gamma=\ppl \nu \ppr$; moreover, in this case the intersection is ${}_\nu \uCZas^{\mu,c}$, and it identifies with ${}_0 \uCZas^{\mu-\nu,c}$. We deduce a decomposition
\[
\uCZas^{\mu,c} = \bigsqcup_{0 \preceq \nu \preceq \mu} {}_{\ppl \nu \ppr} \uZas^\mu \cap \uCZas^{\mu,c} \quad \text{with} \quad {}_{\ppl \nu \ppr} \uZas^\mu \cap \uCZas^{\mu,c} \cong {}_0 \uCZas^{\mu-\nu,c},
\]
which we will use to study the spaces~\eqref{eqn:cohom-compact-Zmu-Fmu}. (Here, for simplicity of notation we interpret $\ppl 0 \ppr$ as the empty partition $\varnothing$.)

By Theorem~\ref{thm:const},
the perverse sheaf $\IC_{\uZas^\mu}$ is constructible with respect to (the analogue for our curve $C=\bbA^1$ of) the stratification~\eqref{eqn:decomp-uZas-Part}.
By the standard characterization of the intersection cohomology complex, if $\nu \neq 0$ the restriction of $\IC_{\uZas^\mu}$ to the stratum ${}_{\ppl \nu \ppr} \uZas^\mu$ is concentrated in degrees strictly smaller than
\[
-\dim({}_{\ppl \nu \ppr} \uZas^\mu)=-\dim({}_0 \uZas^{\mu-\nu}) - 1 = - \langle \mu-\nu, 2\rho \rangle -1.
\]
Hence the same is true for its further restriction to ${}_\nu \uCZas^{\mu,c} = {}_{\ppl \nu \ppr} \uZas^\mu \cap \uCZas^{\mu,c}$.
Since this intersection has dimension $\dim(\uCZas^{\mu-\nu,c})=\langle \mu-\nu, \rho \rangle$, we deduce that
\[
\mathsf{H}^{n}_{\mathrm{c}}({}_\nu \uCZas^{\mu,c}, (\IC_{\uZas^\mu}) {}_{| {}_\nu \uCZas^{\mu,c}})
\]
vanishes unless $n < -1$.

For the case $\nu=0$, we know that $(\IC_{\uZas^\mu}) {}_{| {}_0 \uZas^\mu} \cong \underline{\bk}_{{}_0 \uZas^\mu}[\langle \mu, 2\rho \rangle]$, and the same considerations show that
\[
\mathsf{H}^{n}_{\mathrm{c}}({}_0 \uCZas^{\mu,c}, (\IC_{\uZas^\mu}) {}_{| {}_0 \uCZas^{\mu,c}})
\]
vanishes unless $n \leq 0$. This implies that the space~\eqref{eqn:cohom-compact-Zmu-Fmu} vanishes if $n>0$, and that the natural map
\[
\mathsf{H}^{0}_{\mathrm{c}}({}_0 \uCZas^{\mu,c}, (\IC_{{}_0 \uZas^\mu}) {}_{| {}_0 \uCZas^{\mu,c}})
\to \mathsf{H}^{0}_{\mathrm{c}}(\uCZas^{\mu,c}, (\IC_{\uZas^\mu}) {}_{| \uCZas^{\mu,c}})
\]
is an isomorphism. Dualizing (see~\eqref{eqn:cohom-compact-Zmu-Fmu}), we deduce~\eqref{it:IC-Zastava-1} in case $n<0$, and~\eqref{it:IC-Zastava-2}.

Next, we consider the natural map
$p_\mu : \uZas^\mu \to C^\mu$.
Recall from~\S\ref{ss:relation-Zastava-Drinfeld} that this map has a canonical section
$s_\mu : C^\mu \to \uZas^\mu$. 
%
As explained in~\cite[\S 2.3]{ffkm} or~\cite[Proposition~5.2]{bfgm}, using a contracting $\mathbb{G}_{\mathrm{m}}$-action one checks that there exists a canonical isomorphism
\[
(s_\mu)^* \IC_{\uZas^\mu} \cong (p_\mu)_* \IC_{\uZas^\mu}.
\]
Next, these complexes are equivariant with respect to the dilation action of $\mathbb{G}_{\mathrm{m}}$ on $C^\mu$, so that (as in~\cite[\S 2.3]{ffkm}) their cohomology with compact support identifies with their costalk at the central point $\mu \cdot c$. (This argument is where we use our assumption that $C=\bbA^1$.) By base change, the cohomology of the costalk of the right-hand side identifies with $\mathsf{H}^{\bullet}_{\uCZas^{\mu,c}}(\uZas^\mu, \IC_{\uZas^\mu})$, so that we obtain an isomorphism of graded vector spaces
\begin{equation}
\label{eqn:local-cohom-IC-Zas}
\mathsf{H}^\bullet_{\mathrm{c}}(C^\mu, (s_\mu)^* \IC_{\uZas^\mu}) \cong \mathsf{H}^{\bullet}_{\uCZas^{\mu,c}}(\uZas^\mu, \IC_{\uZas^\mu}).
\end{equation}

We now study the left-hand side of~\eqref{eqn:local-cohom-IC-Zas}.  Consider a stratum ${}_\Gamma \uZas^\mu$ with $\Gamma \in \Part(\nu)$ (and $0 \preceq \nu \preceq \mu$).  By~Lemma~\ref{lem:uZas-smallest-strata}, this stratum meets the image of $s_\mu$ if and only if $\nu = \mu$, and in this case the stratum is contained in the image.
We next observe that for any $\Gamma \in \Part(\mu)$,
the restriction of $\IC_{\uZas^\mu}$ to ${}_\Gamma \uZas^\mu$ is concentrated in degrees $\leq -2|\Gamma|$. 
Indeed,
this follows from the formula in Theorem~\ref{thm:const}\eqref{it:const-Zas}, since each polynomial $P^{\nu_i}$ is divisible by $q$.

Since $\dim {}_\Gamma \uZas^\mu = |\Gamma|$ for $\Gamma \in \Part(\mu)$ (see~\eqref{eqn:dim-strata-uZas}), we deduce that each
\[
\mathsf{H}^n_{\mathrm{c}}({}_\Gamma \uZas^\mu, (\IC_{\uZas^\mu}) {}_{| {}_\Gamma \uZas^\mu})
\]
vanishes unless $n \leq 0$. In view of~\eqref{eqn:local-cohom-IC-Zas}, this implies that $\mathsf{H}^n_{\uCZas^{\mu,c}}(\uZas^\mu, \IC_{\uZas^\mu})$ vanishes when $n > 0$, which finishes the proof.
\end{proof}

\subsection{Proof of Theorem~\ref{thm:Gaits-Bun}}
\label{ss:proof-Drinfeld}

In this subsection we explain how to deduce Theorem~\ref{thm:Gaits-Bun} from Proposition~\ref{prop:IC-Zastava}, following~\cite[\S 3.5--3.8]{gaitsgory}. We fix a choice of local parameter at $c$, which allows to identify $\Gr_c$ with $\Gr$.

First one constructs a canonical map
\begin{equation}
\label{eqn:map-Gaits-Bun}
\Gaits \to \pi^! \IC_{\bBun_U}[-\dim(U)].
\end{equation}
This construction uses the description of the left-hand side given in~\eqref{eqn:Gaits-colim};
it is explained in~\cite[\S 3.4]{gaitsgory} for characteristic-$0$ coefficients, and the arguments apply verbatim for general coefficients. Next, one checks that the object $\pi^! \IC_{\bBun_U}[-\dim(U)]$ belongs to $\Shv(\siIwu \backslash \Gr)$. In case $\mathrm{char}(\bk)=0$, this fact is~\cite[Proposition~3.5.2]{gaitsgory}, which is proved in~\cite[\S 3.8]{gaitsgory}. Here again, the exact same arguments apply in our setting. Finally, to conclude the proof, 
using~\cite[Lemma~3.7]{adr} it suffices to show
that
for any $\mu \in \bY_{\succeq 0}$ the induced morphism
\begin{equation}
\label{eqn:map-Gaits-Bun-mu}
(\bi_{-\mu})^* \Gaits \to (\bi_{-\mu})^* \pi^! \IC_{\bBun_U}[-\dim(U)]
\end{equation}
is an isomorphism.

From now on we fix $\mu \in \bY_{\succeq 0}$.
In view of the equivalence~\eqref{eqn:equiv-Shv-orbit-Vect}, to show that~\eqref{eqn:map-Gaits-Bun-mu} is an isomorphism it suffices to show that its image under the $!$-pullback functor associated by the embedding $\{[z^{-\mu}]\} \hookrightarrow \rS_{-\mu}$ is an isomorphism. Now, by Braden's hyperbolic localization theorem (see~\cite[Lemma~2.2.4]{gaitsgory}) this $!$-pullback functor identifies with cohomology with compact supports, and can be described in terms of the orbit $\rS^-_{-\mu}$; more specifically, we have canonical isomorphisms of graded vector spaces
\begin{align}
\label{eqn:hyp-loc-Gaits}
\mathsf{H}^\bullet_{\mathrm{c}}(\rS_{-\mu}, (\bi_{-\mu})^* \Gaits) &\cong \mathsf{H}^\bullet(\rS^-_{-\mu}, (\bi_{-\mu}^-)^! \Gaits), \\
\mathsf{H}^\bullet_{\mathrm{c}}(\rS_{-\mu}, (\bi_{-\mu})^* \pi^! \IC_{\bBun_U}) &\cong \mathsf{H}^\bullet(\rS^-_{-\mu}, (\bi_{-\mu}^-)^! \pi^! \IC_{\bBun_U}).
\end{align}
So, what we have to show is that for any $n \in \Z$ the morphism
\begin{equation}
\label{eqn:map-Gaits-Bun-mu-2}
\mathsf{H}^n(\rS^-_{-\mu}, (\bi_{-\mu}^-)^! \Gaits) \to \mathsf{H}^{n-\dim(U)}(\rS^-_{-\mu}, (\bi_{-\mu}^-)^! \pi^! \IC_{\bBun_U})
\end{equation}
induced by~\eqref{eqn:map-Gaits-Bun} is an isomorphism.

Now, recall the identification~\eqref{eqn:central-fiber-G-intersection-si-orbits}, and the map $r_\mu$ from~\S\ref{ss:IC-Zastavas}. By construction of these maps the following diagram commutes:
\begin{equation*}
\begin{tikzcd}
\overline{\rS_0} \cap \rS^-_{-\mu}  \ar[d, "\wr"] \ar[r] & \overline{\rS_0} \ar[r, "\bi_0"] & \Gr \ar[d, "\pi"] \\
\uCZasG^{\mu,c} \ar[r] & \uZasG^\mu \ar[r, "r_\mu"] & (\bBun_U)_{\infty \cdot c}.
\end{tikzcd}
\end{equation*}
Using this diagram and Proposition~\ref{prop:pullback-IC-bBun-Zas},
we obtain an isomorphism of graded vector spaces
\begin{equation}
\label{eqn:IC-Bun-Zastava-cohom}
\mathsf{H}^{\bullet-\dim(U)}(\rS^-_{-\mu}, (\bi_{-\mu}^-)^! \pi^! \IC_{\bBun_U}) \cong \mathsf{H}^{\bullet + \langle \mu, 2\rho \rangle}_{\uCZasG^{\mu,c}}(\uZasG^{\mu}, \IC_{\uZasG^{\mu}}).
\end{equation}

%
%

From there, one finishes the proof as follows.
By~\eqref{eqn:hyp-loc-Gaits} and Theorem~\ref{thm:adr}\eqref{it:adr-stalks}, the domain of the map~\eqref{eqn:map-Gaits-Bun-mu-2} vanishes unless $n=-\langle \mu,2\rho \rangle$. The same is true for its codomain by~\eqref{eqn:IC-Bun-Zastava-cohom} and Proposition~\ref{prop:IC-Zastava}\eqref{it:IC-Zastava-1}. 

For the case $n=-\langle \mu,2\rho \rangle$, by~\eqref{eqn:Gaits-colim} the space $\mathsf{H}^{-\langle \mu,2\rho \rangle}(\rS^-_{-\mu}, (\bi_{-\mu}^-)^! \Gaits)$ is the colimit over $\lambda$ of the spaces
\[
\mathsf{H}^{\langle \lambda-\mu,2\rho \rangle}(\rS^-_\mu, (\bi_\mu^-)^! (z^{-\lambda} \cdot \cI_*(\lambda)))
\cong \mathsf{H}^{\langle \lambda-\mu,2\rho \rangle}(\rS^-_{\lambda-\mu}, (\bi_{\lambda-\mu}^-)^! \cI_*(\lambda)).
\]
Now it is a standard fact (see~\cite{mv} or~\cite[Proposition~1.11.1]{br}) that there exists a canonical isomorphism
\[
\mathsf{H}^{\langle \lambda-\mu,2\rho \rangle}(\rS^-_{\lambda-\mu}, (\bi_{\lambda-\mu}^-)^! \cI_*(\lambda)) \cong \mathsf{H}^{-\langle \mu,2\rho \rangle}(\Gr^{\lambda} \cap \rS^-_{\lambda-\mu}, \omega_{\Gr^{\lambda} \cap \rS^-_{\lambda-\mu}}).
\]

As for the codomain of our map~\eqref{eqn:map-Gaits-Bun-mu-2}, by~\eqref{eqn:IC-Bun-Zastava-cohom} and Proposition~\ref{prop:IC-Zastava}\eqref{it:IC-Zastava-2} it identifies with
\[
\mathsf{H}^{0}_{{}_0 \uCZasG^{\mu,c}}({}_0 \uZasG^{\mu}, \IC_{{}_0 \uZas^{\mu}}) = 
\mathsf{H}^{-\langle \mu, 2\rho \rangle}({}_0 \uCZasG^{\mu,c}, \omega_{{}_0 \uCZasG^{\mu,c}}) \overset{\eqref{eqn:central-fiber-G-0-intersection-si-orbits}}{=}
\mathsf{H}^{-\langle \mu, 2\rho \rangle}(\rS_0 \cap \rS^-_{-\mu}, \omega_{\rS_0 \cap \rS^-_{-\mu}}).
\]
To conclude our proof, it therefore suffices to check that our morphism
\begin{equation}
\label{eqn:restriction-intersections}
\mathsf{H}^{-\langle \mu,2\rho \rangle}(\Gr^{\lambda} \cap \rS^-_{\lambda-\mu}, \omega_{\Gr^{\lambda} \cap \rS^-_{\lambda-\mu}}) \to \mathsf{H}^{-\langle \mu, 2\rho \rangle}(\rS_0 \cap \rS^-_{-\mu}, \omega_{\rS_0 \cap \rS^-_{-\mu}})
\end{equation}
is an isomorphism for sufficiently dominant $\lambda$.

Now, it follows from~\cite[Proposition~3.6(iii)]{baugau} that if $\lambda$ is sufficiently dominant we have
\[
\rS_0 \cap \rS^-_{-\mu} = z^{-\lambda} \cdot (\Gr^\lambda \cap \rS^-_{\lambda-\mu}).
\]
It is then clear that~\eqref{eqn:restriction-intersections} is indeed an isomorphism in this case, which concludes the proof.

\appendix
\section{Intersection cohomology with coefficients in a principal ideal domain}
\label{sec:PID}

\newcommand{\tor}{{\mathrm{tor}}}
\newcommand{\Frac}{{\mathrm{Frac}}}
\newcommand{\fm}{\mathfrak{m}}
\newcommand{\fp}{\mathfrak{p}}

\subsection{Setting}

We fix a base field $\F$. Below, for simplicity we will write ``scheme''
for ``separated scheme of finite type over $\F$.''  
In this appendix, we consider sheaves on schemes (especially intersection cohomology complexes) with coefficients in a principal ideal domain $\bk$. As usual
we assume that we are in one of the following two settings:
\begin{enumerate}
\item $\F$ is arbitrary, $\ell$ is a prime number invertible in $\F$, $\bk$ is the ring of integers in a finite extension of $\Q_\ell$, and we work with the bounded derived category of $\bk$-sheaves on $X$ constructed from the categories of \'etale sheaves with coefficients in $\bk/\varpi^n \bk$ where $\varpi$ is a uniformizer via a direct limit construction;
\item $\F = \C$, $\bk$ is an arbitrary principal ideal domain, and we work with the bounded derived category of $\bk$-sheaves in the analytic topology.
\end{enumerate}
In this section we will make use of finiteness properties, so it is important that we work with constructible complexes and not with the $\infty$-categories $\Shv$ of~\S\ref{ss:complexes-sheaves}. In addition, the $\infty$-category framework will not be used in any way, so we simply work in the setting of triangulated categories.

For every prime ideal $\fp \subset \bk$, we set
\[
\kappa(\fp) = \bk_\fp /\fp \bk_\fp.
\]
In particular, $\kappa(0)$ is the fraction field of $\bk$. Our goal is to prove that if the intersection cohomology complex $\IC_X(\kappa(\fp))$ with coefficients in the field $\kappa(\fp)$ (associated with the constant local system on the smooth part of $X$) is ``independent of $\fp$'' in a suitable sense, then one can obtain a description of the (co)stalks of the intersection cohomology complex $\IC_X(\bk)$ with coefficients in $\bk$ in terms of those of the complexes $\IC_X(\kappa(\fp))$; see Proposition~\ref{prop:IC-indep} for a precise statement.

For a finitely generated $\bk$-module $M$, let $M_\tor \subset M$ be its torsion submodule.  Then $M/M_\tor$ is a free $\bk$-module, and we have
\[
\kappa(0) \otimes_\bk M \cong \kappa(0) \otimes_\bk (M/M_\tor).
\]
For a not necessarily free $\bk$-module $M$, we define its \emph{rank}, denoted $\rank M$, to be the rank of $M/M_\tor$; we therefore have
\[
\dim_{\kappa(0)} (\kappa(0) \otimes_\bk M) = \rank M.
\]

Next, if $M$ is a bounded chain complex of finitely generated $\bk$-modules, we can consider its Euler characteristic $\chi(M)$, defined by
\[
\chi(M) = \sum_i (-1)^i \cdot \rank H^i(M).
\]
The Euler characteristic is also given by
\begin{equation}\label{eqn:euler-equal}
\chi(M) = \sum_i (-1)^i \cdot \dim_{\kappa(\fp)} H^i(\kappa(\fp) \otimes_\bk^L M)
\end{equation}
for any prime ideal $\fp$.  In particular, the right-hand side is independent of $\fp$.

Finally, let $\cF$ be a constructible complex of $\bk$-sheaves on some scheme $X$. For any prime ideal $\fp \subset \bk$, we have a ``universal coefficient theorem:'' for each $m \in \Z$, there is a natural short exact sequence
\begin{equation}\label{eqn:univcoeff}
0 \to \kappa(\fp) \otimes_\bk \mathcal{H}^m(\cF) \to \mathcal{H}^m(\kappa(\fp) \otimes_\bk^L \cF) \to \mathcal{T}\!\mathit{or}_1^\bk( \kappa(\fp), \mathcal{H}^{m+1}(\cF)) \to 0.
\end{equation}
(Here, the last term means $\mathcal{H}^{-1}(\kappa(\fp) \otimes_\bk^L \mathcal{H}^{m+1}(\cF))$.) Of course, $\kappa(0)$ is flat over $\bk$, so for $\fp = 0$, the last term vanishes and we have
\[
\kappa(0) \otimes_\bk \mathcal{H}^m(\cF) \cong \mathcal{H}^m(\kappa(0) \otimes_\bk^L \cF).
\]

\subsection{A criterion for a constructible sheaf to be a local system}

\begin{lem}\label{lem:pid-locsys}
Let $X$ be a smooth, connected scheme, and let $\cF$ be a constructible $\bk$-sheaf whose stalks are free $\bk$-modules.  Assume that for every prime ideal $\fp \subset \bk$, the sheaf $\kappa(\fp) \otimes_\bk \cF$ is a local system.  Then $\cF$ is a local system.
\end{lem}
\begin{proof}
Since $\cF$ is constructible, there is an open subscheme $U \subset X$ such that $\cF_{|U}$ is a local system.  Choose a geometric point $\bar x$ in $U$, and let $M = \cF_{\bar x}$.  The local systems $\cF_{|U}$ and $\kappa(0) \otimes_\bk \cF$ correspond to the vertical maps in the following diagram (where, depending on the context, we consider either the \'etale or topological fundamental groups):
\[
\begin{tikzcd}
\pi_1(U,\bar x) \ar[r] \ar[d] & \pi_1(X,\bar x) \ar[d] \ar[dl, dashed] \\
\mathrm{Aut}_\bk(M) \ar[r] & \mathrm{Aut}_{\kappa(0)}(\kappa(0) \otimes_\bk M).
\end{tikzcd}
\]
This diagram commutes, its upper horizontal arrow is surjective (either by~\cite[\href{https://stacks.math.columbia.edu/tag/0BQI}{Tag 0BQI}]{stacks-project} or by~\cite[Lemma~2.1.22]{achar-book}), and the bottom horizontal map is injective (because $M$ is free over $\bk$), so there is a unique way to fill in the dotted arrow $\pi_1(X,\bar x) \to \mathrm{Aut}_\bk(M)$. This dotted arrow determines a $\bk$-local system $\cG$ on $X$ such that
\[
\cG_{|U} \cong \cF_{|U}
\qquad\text{and}\qquad
\kappa(0) \otimes_\bk \cG \cong \kappa(0) \otimes_\bk \cF.
\]

Let $Z=X \smallsetminus U$ (which we endow with the reduced subscheme structure), and denote by $i : Z \hookrightarrow X$ and $j: U \hookrightarrow X$ be the inclusion maps. First, we claim that the canonical morphism $\cG \to \mathcal{H}^0(j_*\cG_{|U})$ is an isomorphism. In fact, consider the distinguished triangle
\[
i_* i^! \cG \to \cG \to j_* \cG_{|U} \xrightarrow{[1]}.
\]
Locally, $\cG$ is isomorphic to a sum of copies of the constant sheaf $\underline{\bk}_X$, i.e.~to the shifted dualizing complex $\underline{\mathbb{D}}_X[-2\dim(X)]$, hence $i^! \cG$ identifies (again, locally) with a sum of copies of $\underline{\mathbb{D}}_Z[-2\dim(X)]$. In particular, this complex is concentred in degrees $\geq 2(\dim(X) - \dim(Z)) \geq 2$. Considering the long exact sequence of cohomology associated with the triangle above, we deduce the claim.

Using this isomorphism we obtain a natural map
\begin{equation}\label{eqn:locsys-compare}
\cF \to \mathcal{H}^0(j_*\cF_{|U}) \cong \mathcal{H}^0(j_*\cG_{|U}) \cong \cG;
\end{equation}
to finish the proof, we must show that this is an isomorphism. 
Using the same distinguished triangle as above, but now for the object $\cF$, we see that the kernel and cokernel of~\eqref{eqn:locsys-compare} are $\mathcal{H}^0(i_*i^!\cF)$ and $\mathcal{H}^1(i_*i^!\cF)$ respectively, so we must show that
\begin{equation}\label{eqn:loccst-crit}
\mathcal{H}^0(i^!\cF) = \mathcal{H}^1(i^!\cF) = 0.
\end{equation}

For every prime ideal $\fp \subset \bk$, we already know that $\kappa(\fp) \otimes_\bk \cF$ is locally constant, so the counterpart of~\eqref{eqn:loccst-crit} with $i^!\cF$ replaced by the object
\[
i^!(\kappa(\fp) \otimes_\bk \cF) \cong i^!(\kappa(\fp) \otimes_\bk^L \cF) \cong \kappa(\fp) \otimes_\bk^L i^!\cF
\]
holds: that is, we have
\[
\mathcal{H}^0(\kappa(\fp) \otimes^L_\bk i^!\cF) = \mathcal{H}^1(\kappa(\fp) \otimes^L_\bk i^!\cF) = 0.
\]
Using~\eqref{eqn:univcoeff}, we deduce that
\[
\kappa(\fp) \otimes_\bk \mathcal{H}^0(i^!\cF) = 
\kappa(\fp) \otimes_\bk \mathcal{H}^1(i^!\cF) = 0.
\]
Since this holds for all $\fp \subset \bk$, we deduce that~\eqref{eqn:loccst-crit} is true, and hence that~\eqref{eqn:locsys-compare} is an isomorphism.
\end{proof}

\subsection{Extension to complexes}

\begin{lem}
\label{lem:pid-construc}
Let $X$ be a smooth, connected scheme, and let $\cF$ be a constructible complex of $\bk$-sheaves on $X$ that satisfies the following conditions:
\begin{itemize}
\item For every prime ideal $\fp \subset \bk$, the complex $\kappa(\fp) \otimes_\bk^L \cF$ has locally constant cohomology sheaves.
\item For each $m \in \Z$, the integer
\[
\rank \mathcal{H}^m(\kappa(\fp) \otimes_\bk^L \cF) 
\]
is independent of $\fp$.
\end{itemize}
Then for all $m \in \Z$, $\mathcal{H}^m(\cF)$ is a locally constant sheaf whose stalks are free $\bk$-modules, of rank the rank of $\mathcal{H}^m(\kappa(\fp) \otimes_\bk^L \cF)$ for any $\fp \subset \bk$.
\end{lem}
\begin{proof}
We proceed by induction on the number of nonzero cohomology sheaves of $\cF$.  If $\cF = 0$, there is nothing to prove.  Otherwise, let $a$ be the smallest integer such that $\mathcal{H}^a(\cF) \ne 0$, and let
\[
\cG = \mathcal{H}^a(\cF).
\]
Let us introduce the notation
\[
r_m = \rank \mathcal{H}^m(\kappa(\fp) \otimes_\bk^L \cF) \quad\text{for any $\fp \subset \bk$.}
\]

Choose a geometric point $\bar y$ in $X$ such that the stalk $\cG_{\bar y}$ is nonzero, and let $N = \cG_{\bar y}$. By~\cite[\href{https://stacks.math.columbia.edu/tag/00HN}{Tag 00HN}]{stacks-project} and the Nakayama lemma, there exists a maximal ideal $\fm$ such that $N/\fm N \ne 0$, and hence that $\kappa(\fm) \otimes_\bk \cG$ has a nonzero stalk at $\bar y$.  By~\eqref{eqn:univcoeff}, we see that $\mathcal{H}^a(\kappa(\fm) \otimes_\bk^L \cF)$ also has a nonzero stalk at $\bar y$, so
\[
r_a > 0.
\]

Now let $\bar x$ be any geometric point of $X$, and let $M = \cG_{\bar x}$.  The stalk of $\kappa(0) \otimes_\bk \cG = \mathcal{H}^a(\kappa(0) \otimes_\bk^L \cF)$ at $\bar x$ is $\kappa(0) \otimes_\bk M$, so $M$ has rank $r_a$.

Next, we claim that $M_\tor = 0$.  If not, as above there is some maximal ideal $\fm \subset \bk$ such that $M_\tor / \fm M_\tor \ne 0$, and then we see that the $\kappa(\fm)$-vector space $\kappa(\fm) \otimes_\bk M$ has dimension at least $r_a+1$. By~\eqref{eqn:univcoeff}, the stalk of $\mathcal{H}^a(\kappa(\fm) \otimes_\bk^L \cF)$ at $\bar x$ also has dimension at least $r_a+1$, contradicting our assumptions.

We have shown now that all stalks of $\cG$ are free $\bk$-modules of rank $r_a$. In particular, $\cG$ is a flat $\bk$-sheaf.  Now consider a prime ideal $\fp \subset \bk$, and the truncation distinguished triangle
\begin{equation}\label{eqn:locconst-trunc}
\kappa(\fp) \otimes_\bk^L \cG[-a] 
\to
\kappa(\fp) \otimes_\bk^L \cF
\to
\kappa(\fp) \otimes_\bk^L (\tau^{\ge a+1}\cF) \xrightarrow{+1}.
\end{equation}
By flatness, the first term is concentrated in cohomological degree $a$.  By~\eqref{eqn:univcoeff}, the last term must be concentrated in degrees${}\ge a$.  But if $\mathcal{H}^a(\kappa(\fp) \otimes_\bk^L (\tau^{\ge a+1}\cF))$ were nonzero, the long exact sequence of cohomology sheaves would show that the stalks of $\mathcal{H}^a(\kappa(\fp) \otimes_\bk^L \cF)$ have dimension${}>r_a$, which is impossible.  We deduce that the third term of~\eqref{eqn:locconst-trunc} is concentrated in degrees${}\ge a+1$, i.e., that the whole triangle can be identified with
\[
\tau^{\le a}(\kappa(\fp) \otimes_\bk^L \cF) \to \kappa(\fp) \otimes_\bk^L \cF \to \tau^{\ge a+1}(\kappa(\fp) \otimes_\bk^L \cF) \xrightarrow{+1}.
\]
The first term shows that $\kappa(\fp) \otimes_\bk \cG \cong \mathcal{H}^a(\kappa(\fp) \otimes_\bk^L \cF)$, so $\kappa(\fp) \otimes_\bk \cG$ is a local system.  By Lemma~\ref{lem:pid-locsys}, $\cG$ is a local system.

On the other hand, the distinguished triangles above show that $\tau^{\ge a+1}\cF$ satisfies the assumptions of the lemma.  As it has fewer nonzero cohomology sheaves than $\cF$, the conclusions of the lemma hold for it by induction.
\end{proof}

\subsection{Main result}

We can finally come to the main result of this section. Let $X$ be an irreducible scheme, let $(X_s)_{s \in \mathscr{S}}$ be a finite stratification by smooth, connected, locally closed subschemes, and for $s \in \mathscr{S}$ let $j_s: X_s \hookrightarrow X$ be the inclusion map. We consider the intersection cohomology complex $\IC_X(\bk)$ over $\bk$ associated with the constant sheaf on the unique open stratum in $X$ and, for any prime ideal $\fp \subset \bk$, the similar complex $\IC_X(\kappa(\fp))$ over the field $\kappa(\fp)$.

\begin{prop}
\label{prop:IC-indep}
Assume that the following conditions hold:
\begin{itemize}
\item For every prime ideal $\fp \subset \bk$, $\IC_X(\kappa(\fp))$ is constructible with respect to the stratification $(X_s)_{s \in \mathscr{S}}$.
\item The function $d: \mathscr{S} \times \Z \to \Z$ given by
\[
d(s,m) = \rank \mathcal{H}^{-\dim X_s -m}(\IC_X(\kappa(\fp))_{|X_s})
\]
is independent of $\fp$.
\end{itemize}
Then the following hold:
\begin{enumerate}
\item For each $s \in \mathscr{S}$ and each $m \in \Z$, $\mathcal{H}^m(j_s^!\IC_X(\bk))$ and $\mathcal{H}^m(j_s^*\IC_X(\bk))$ are locally constant sheaves with free stalks.  Moreover,\label{it:free-stalks}
\[
\rank \mathcal{H}^{-\dim X_s+m}(j_s^!\IC_X(\bk)) = 
\rank \mathcal{H}^{-\dim X_s-m}(j_s^*\IC_X(\bk)) = d(s,m).
\]
\item For every prime ideal $\fp \subset \bk$, we have\label{it:IC-scalar}
\[
\kappa(\fp) \otimes_\bk^L \IC_X(\bk) \cong \IC_X(\kappa(\fp)).
\]
\end{enumerate}
\end{prop}
\begin{proof}
We proceed by induction on the number of strata in the stratification.  If there is a single stratum, then $X$ is smooth, and of course $\IC_X(\bk) \cong \underline{\bk}_X[\dim X]$.  The proposition is trivial in this case.

Suppose now that there is more than one stratum.  Let $X_t$ be a closed stratum in our stratification, and let $U = X \smallsetminus X_t$.  Let $j_t: X_t \hookrightarrow X$ and $h: U \hookrightarrow X$ be the inclusion maps.  By induction, part~\eqref{it:free-stalks} of the proposition holds for all $s \ne t$, and
\[
\kappa(\fp) \otimes_\bk^L \IC_U(\bk) \cong \IC_U(\kappa(\fp))
\]
for any prime ideal $\fp \subset \bk$.
In addition, since $\kappa(0)$ is flat over $\bk$, part~\eqref{it:IC-scalar} automatically holds for $\fp = 0$:
\begin{equation}\label{eqn:IC-scalar-frac}
\kappa(0) \otimes_\bk^L \IC_X(\bk) \cong \IC_X(\kappa(0)).
\end{equation}

Recall that $\pH^m(j_t^!\IC_X(\bk))$, resp.~$\pH^m(j_t^!\IC_X(\bk))$, vanishes for $m \le 0$, resp.~$m \ge 0$. If $\fp \subset \bk$ is a prime ideal, using the perverse cohomology counterpart of~\eqref{eqn:univcoeff} we see that
\begin{align*}
\pH^m(j_t^!(\kappa(\fp) \otimes_\bk^L \IC_X(\bk))) &= 0 &&\text{for $m < 0$ (for $m \le 0$ if $\fp = 0$)} \\
\pH^m(j_t^*(\kappa(\fp) \otimes_\bk^L \IC_X(\bk))) &= 0 &&\text{for $m \ge 0$.}
\end{align*}
In particular, $\kappa(\fp) \otimes_\bk^L \IC_X(\bk)$ is a perverse sheaf. Let $\cJ_{\fp} = \pH^0(j_t^!(\kappa(\fp) \otimes_\bk^L \IC_X(\bk)))$, which we consider as a perverse sheaf on $X$.  Then there is a natural short exact sequence of perverse sheaves
\begin{equation}\label{eqn:IC-ses}
0 \to \cJ_{\fp} \to \kappa(\fp) \otimes_\bk^L \IC_X(\bk) \to \IC_X(\kappa(\fp)) \to 0.
\end{equation}
We have already seen in~\eqref{eqn:IC-scalar-frac} that $\cJ_0 = 0$.  We will now prove that $\cJ_\fp = 0$ for all prime ideals $\fp \subset \bk$.

Suppose $\cJ_\fp \ne 0$, and consider its support $\mathrm{supp}\,\cJ_{\fp}$, a closed subscheme of $X_t$.  Let $Y \subset \mathrm{supp}\,\cJ_{\fp}$ be a smooth, connected, open subscheme such that $\cJ_{\fp|Y}$ is a shifted local system, i.e.,
\[
\cJ_{\fp|Y} \cong \mathcal{L}[\dim Y]
\]
for some nonzero $\kappa(\fp)$-local system $\mathcal{L}$ on $Y$.  (Note that $Y$ is locally closed in $X_t$.)  Let $\bar y$ be a geometric point of $Y$, and let $i: \bar y \to X$ be the inclusion map.  We then have
\begin{equation}\label{eqn:IC-ses1}
i^!\cJ_\fp = (\mathcal{L}_\fp)_{\bar y}[-\dim Y],
\end{equation}
see e.g.~\cite[Theorem~2.2.13]{achar-book}.
Next, let $i': \bar y \to X_t$ be the inclusion map, so that $i = j_t \circ i'$.  Since $j_t^!\IC_X(\kappa(\fp))$ has locally constant cohomology sheaves, as above we have
\[
i^!\IC_X(\kappa(\fp)) \cong (i')^!j_t^!\IC_X(\kappa(\fp)) \cong (j_t^!\IC_X(\kappa(\fp)))_{\bar y}[-2\dim X_t].
\]
We deduce that 
\begin{multline*}
\dim H^m(i^!\IC_X(\kappa(\fp)))
= \rank \mathcal{H}^{m-2 \dim X_t}(j_t^!\IC_X(\kappa(\fp))) \\
= \begin{cases}
d(t,m-\dim X_t) & \text{if $m > \dim X_t$,} \\
0 & \text{otherwise.}
\end{cases}
\end{multline*}
Let
\[
A = \chi(i^!\IC_X(\kappa(\fp))) = \sum_{m > 0} (-1)^{m-\dim X_t}d(t,m).
\]
From~\eqref{eqn:IC-scalar-frac} and~\eqref{eqn:IC-ses}, we have the following Euler characteristic calculations:
\begin{align*}
\chi(i^!(\kappa(0) \otimes_\bk^L \IC_X(\bk))) &= A, \\
\chi(i^!(\kappa(\fp) \otimes_\bk^L \IC_X(\bk))) &= A + (-1)^{\dim Y} \dim (\mathcal{L}_\fp)_{\bar y}.
\end{align*}
But these two numbers should both agree with $\chi(\IC_X(\bk))$ (see~\eqref{eqn:euler-equal}), a contradiction.  We conclude that $\cJ_\fp = 0$, and that
\[
\kappa(\fp) \otimes_\bk^L \IC_X(\bk) \cong \IC_X(\kappa(\fp))
\]
for all $\fp \subset \bk$.  We have proved part~\eqref{it:IC-scalar} of the proposition.  Moreover, this implies that
\[
\kappa(\fp) \otimes_\bk^L j_t^!\IC_X(\bk) \cong j_t^!\IC_X(\kappa(\fp))
\qquad\text{and}\qquad
\kappa(\fp) \otimes_\bk^L j_t^*\IC_X(\bk) \cong j_t^*\IC_X(\kappa(\fp)).
\]
Part~\eqref{it:free-stalks} of the proposition for $s = t$ then follows by applying Lemma~\ref{lem:pid-construc} to the objects $j_t^!\IC_X(\bk)$ and $j_t^*\IC_X(\bk)$.
\end{proof}

\end{document}